\documentclass{amsart}

\usepackage{amsthm}
\usepackage[pagebackref]{hyperref}        
\hypersetup{
    colorlinks=true,    
    linkcolor=blue,          
    citecolor=blue,      
    filecolor=blue,      
    urlcolor=blue           
}
\usepackage{cleveref}
\usepackage{relsize}

\usepackage{epsfig}
\usepackage{flafter}
\usepackage{comment}
\usepackage{mathrsfs}
\usepackage{enumitem}
\usepackage{xcolor}

\usepackage[normalem]{ulem}

\newcommand{\customref}[2]{\hyperref[#1]{#2}}

\usepackage{tikz}
\usetikzlibrary{graphs,positioning,arrows,shapes.misc,decorations.pathmorphing}
\tikzstyle{vertex}=[
    draw,
    circle,
    fill=black,
    inner sep=1pt,
    minimum width=5pt,
]
\usetikzlibrary{matrix,arrows,decorations.pathmorphing}
\usepackage{tikz-cd}
\usepackage[position=top]{subfig}
\usepackage{color}

\usepackage{amsmath}
\usepackage{amsfonts}
\usepackage{graphics}
\usepackage{amsfonts}
\usepackage{amscd}
\usepackage{latexsym}
\usepackage[abbrev,alphabetic,initials,nobysame]{amsrefs}
\usepackage{amssymb}
\usepackage[all]{xy}
\usepackage{mathtools}

\calclayout

\usetikzlibrary{decorations.pathmorphing}
\tikzstyle{printersafe}=[decoration={snake,amplitude=0pt}]

\DeclareFontFamily{U}{wncy}{}
\DeclareFontShape{U}{wncy}{m}{n}{<->wncyr10}{}
\DeclareSymbolFont{mcy}{U}{wncy}{m}{n}
\DeclareMathSymbol{\Sh}{\mathord}{mcy}{"58} 

\newcommand{\End}{\operatorname{End}}

\newcommand{\bs}{\backslash}
\newcommand{\op}{\operatorname{op}}
\newcommand{\proj}{\operatorname{Proj}}

\newcommand{\id}{\operatorname{id}}

\newcommand{\codim}{\operatorname{codim}}

\newcommand{\rank}{\operatorname{rank}}

\newcommand{\Spec}{\operatorname{Spec}}

\newcommand{\an}{\operatorname{an}}

\newcommand{\coker}{\operatorname{coker}}

\newcommand{\ord}{\operatorname{ord}}

\newcommand{\rar}{\rightarrow}	
\newcommand{\ol}{\overline}

\newcommand{\trcon}{\mathrm{\overline{tr}}}

\DeclareMathOperator{\Supp}{Supp}

\DeclareMathOperator{\Diff}{Diff}

\def\O#1.{\mathcal {O}_{#1}}			
\def\pr #1.{\mathbb P^{#1}}				
\def\af #1.{\mathbb A^{#1}}			
\def\ses#1.#2.#3.{0\to #1\to #2\to #3 \to 0}	
\def\xrar#1.{\xrightarrow{#1}}			
\def\K#1.{K_{#1}}						
\def\bA#1.{\mathbf{A}_{#1}}			
\def\bM#1.{\mathbf{M}_{#1}}
\def\bN#1.{\mathbf{N}_{#1}}
\def\bL#1.{\mathbf{L}_{#1}}				
\def\bB#1.{\mathbf{B}_{#1}}				
\def\bK#1.{\mathbf{K}_{#1}}			
\def\subs#1.{_{#1}}					
\def\sups#1.{^{#1}}						
\def \neone#1.{\overline{{\rm NE}(#1)}}

  		\newtheorem{theorem}{Theorem}[section]
  		\newtheorem{lemma}[theorem]{Lemma}

  		\newtheorem{claim}[theorem]{Claim}

\newtheorem{lem}[theorem]{Lemma}
\newtheorem{prop}[theorem]{Proposition}
\newtheorem{cor}[theorem]{Corollary}
\newtheorem{thm}[theorem]{Theorem}

\theoremstyle{definition}
  		
  		\newtheorem{notation}[theorem]{Notation}
  		\newtheorem{definition}[theorem]{Definition}
        \newtheorem{property}[theorem]{Property}
  		\newtheorem{example}[theorem]{Example}
        \newtheorem{setup}[theorem]{Setup}

        \newtheorem{remark}[theorem]{Remark}

        \newtheorem{constdef}[theorem]{Construction-Definition}
        \newtheorem{const}[theorem]{Construction}

        \newtheorem{eg}[theorem]{Example}
\newtheorem{defn}[theorem]{Definition}
\newtheorem{rem}[theorem]{Remark}

\theoremstyle{remark}

\AddToHook{env/introdef/begin}{\crefalias{theorem}{introdef}}
\AddToHook{env/introcon/begin}{\crefalias{theorem}{introcon}}
\AddToHook{env/introthm/begin}{\crefalias{theorem}{introthm}}
\AddToHook{env/introlem/begin}{\crefalias{theorem}{introlem}}
\AddToHook{env/introcor/begin}{\crefalias{theorem}{introcor}}
\AddToHook{env/introprop/begin}{\crefalias{theorem}{introprop}}
\AddToHook{env/intrormk/begin}{\crefalias{theorem}{intrormk}}
\AddToHook{env/introrem/begin}{\crefalias{theorem}{introrem}}

\AddToHook{env/theorem/begin}{\crefalias{theorem}{theorem}}
\AddToHook{env/lemma/begin}{\crefalias{theorem}{lemma}}
\AddToHook{env/proposition/begin}{\crefalias{theorem}{proposition}}
\AddToHook{env/corollary/begin}{\crefalias{theorem}{corollary}}
\AddToHook{env/claim/begin}{\crefalias{theorem}{claim}}
\AddToHook{env/reduction/begin}{\crefalias{theorem}{reduction}}
\AddToHook{env/conjecture/begin}{\crefalias{theorem}{conjecture}}

\AddToHook{env/lem/begin}{\crefalias{theorem}{lemma}}
\AddToHook{env/prop/begin}{\crefalias{theorem}{proposition}}
\AddToHook{env/cor/begin}{\crefalias{theorem}{corollary}}
\AddToHook{env/thm/begin}{\crefalias{theorem}{theorem}}
\AddToHook{env/defn/begin}{\crefalias{theorem}{definition}}
\AddToHook{env/rem/begin}{\crefalias{theorem}{remark}}

\AddToHook{env/notation/begin}{\crefalias{theorem}{notation}}
\AddToHook{env/definition/begin}{\crefalias{theorem}{definition}}
\AddToHook{env/property/begin}{\crefalias{theorem}{property}}
\AddToHook{env/example/begin}{\crefalias{theorem}{example}}
\AddToHook{env/setup/begin}{\crefalias{theorem}{setup}}
\AddToHook{env/remarks/begin}{\crefalias{theorem}{remarks}}
\AddToHook{env/problem/begin}{\crefalias{theorem}{problem}}
\AddToHook{env/question/begin}{\crefalias{theorem}{question}}
\AddToHook{env/construction/begin}{\crefalias{theorem}{construction}}
\AddToHook{env/computation/begin}{\crefalias{theorem}{computation}}
\AddToHook{env/exma/begin}{\crefalias{theorem}{exma}}
\AddToHook{env/remark/begin}{\crefalias{theorem}{remark}}
\AddToHook{env/constdef/begin}{\crefalias{theorem}{constdef}}
\AddToHook{env/const/begin}{\crefalias{theorem}{const}}
\AddToHook{env/eg/begin}{\crefalias{theorem}{eg}}

\AddToHook{env/introclaim/begin}{\crefalias{theorem}{introclaim}}
\AddToHook{env/introex/begin}{\crefalias{theorem}{introex}}
\AddToHook{env/conconj/begin}{\crefalias{theorem}{conconj}}

\def\bA{\mathbb{A}}
\def\bB{\mathbb{B}}
\def\bD{\mathbb{D}}
\def\bH{\mathbb{H}}
\def\bN{\mathbb{N}}
\def\bZ{\mathbb{Z}}
\def\bQ{\mathbb{Q}}
\def\bR{\mathbb{R}}
\def\bC{\mathbb{C}}
\def\bP{\mathbb{P}}
\def\bL{\mathbb{L}}

\def\cA{\mathcal{A}}

\def\cE{\mathcal{E}}

\def\cL{\mathcal{L}}

\def\cO{\mathcal{O}}
\def\cP{\mathcal{P}}

\def\cU{\mathcal{U}}
\def\cV{\mathcal{V}}

\def\cX{\mathcal{X}}
\def\cY{\mathcal{Y}}
\def\cZ{\mathcal{Z}}

\def\fT{\mathfrak{T}}

  \newcommand\bfM{{\bf M}}
  \newcommand\bfN{{\bf N}}

\newcommand\bfA{{\bf A}}
\newcommand{\dlt}{\operatorname{dlt}}

\def\define{\mathrm{def}}
\def\Spec{\operatorname{Spec}}
\def\reduced{\mathrm{red}}

\def\reg{\mathrm{reg}}
\def\an{\mathrm{an}}
\def\gr{\operatorname{gr}}

   \def\trans{\mathrm{tr}}
   \def\mini{\mathrm{min}}
   \def\orig{\mathrm{orig}}
\def\rk{\operatorname{rk}}
\def\pmax{m}
\def\Aut{\operatorname{Aut}}

\def\End{\operatorname{End}}
\def\bfM{\mathbf{M}}

\def\SL{\mathrm{SL}}
\def\Res{\operatorname{Res}}
\def\tr{\operatorname{tr}}
\def\BB{\mathrm{BB}}
\def\bdd{\mathrm{mg}}
\def\curve{\mathrm{curve}}
\def\fT{\mathfrak{T}}

\renewcommand{\bar}[1]{\ol{#1}}
\def\Sym{\operatorname{Sym}}
\def\fY{\mathfrak{Y}}
\def\fU{\mathfrak{U}}
\def\BBH{\mathrm{BBH}}

\numberwithin{equation}{section}

\makeatletter
\@namedef{subjclassname@2020}{
  \textup{2020} Mathematics Subject Classification}
\makeatother

\newcommand{\Exterior}{\mathchoice{{\textstyle\bigwedge}}%
    {{\bigwedge}}%
    {{\textstyle\wedge}}%
    {{\scriptstyle\wedge}}}

\makeatletter
\def\MR#1{}%
\def\MRhref#1#2{}
\makeatother

\def\testsp{\Theta}

\begin{document}

\title[Baily--Borel compactifications and  b-semiampleness]{Baily--Borel compactifications of period images and the b-semiampleness conjecture}
 \author[B. Bakker]{Benjamin Bakker}
\address{\noindent B. Bakker:  Dept. of Mathematics, Statistics, and Computer Science, University of Illinois at Chicago, Chicago, USA.}
\email{bakker.uic@gmail.com}
\author[S. Filipazzi]{Stefano Filipazzi}
\address{\noindent S. Filipazzi: Department of Mathematics, Duke University,
120 Science Drive,
117 Physics Building,
Campus Box 90320,
Durham, NC 27708-0320,
USA.}
\email{stefano.filipazzi@duke.edu}
\author[M. Mauri]{Mirko Mauri}
\address{\noindent M. Mauri: Institut de Math\'ematiques 
de Jussieu-Paris Rive Gauche,
Universit\'e Paris Cit\'e,
Place Aur\'elie Nemours, 75013 Paris, France.}
\email{mauri@imj-prg.fr}
\author[J. Tsimerman]{Jacob Tsimerman}
\address{\noindent J. Tsimerman:  Dept. of Mathematics, University of Toronto, Toronto, Canada.}
\email{jacobt@math.toronto.edu}


\begin{abstract}
We address two questions related to the semiampleness of line bundles arising from Hodge theory.  First, we prove there is a functorial compactification of the image of a period map of a polarizable integral pure variation of Hodge structures for which the Griffiths bundle extends amply.  In particular the Griffiths bundle is semiample.  We prove more generally that the Hodge bundle of a Calabi--Yau variation of Hodge structures is semiample subject to some extra conditions, and as our second result deduce the b-semiampleness conjecture and the existence of a functorial Hodge-theoretic compactification of moduli spaces of polarized Calabi--Yau varieties.
The semiampleness results (and the construction of the Baily--Borel compactifications) crucially use o-minimal GAGA, and the deduction of the  b-semiampleness conjecture uses work of Ambro and results of Koll\'ar on the geometry of minimal lc centers to verify the extra conditions.  
\end{abstract}

\thanks{
BB was partially supported by NSF DMS-2401383. SF was partially supported by ERC starting grant \#804334 and by Duke University. MM was supported by Université Paris Cité and Sorbonne Université, CNRS, IMJ-PRG, F-75013 Paris, France. JT was supported by a Simons investigator grant.
}

\maketitle

\setcounter{tocdepth}{1}

\tableofcontents

\section{Introduction}

\subsection{Baily--Borel compactifications of images of period map}

Let $(X,D)$ be a log-smooth algebraic space, such that $X\bs D$ supports a polarizable integral pure variation of Hodge structures $V=(V_{\bZ},F^{\bullet}V_\cO)$. Letting $\bD$ be the associated period domain and $\Gamma$ an arithmetic group containing the monodromy of $V_{\bZ}$, we obtain a \textit{period map} $\phi:(X\bs D)^{\an}\to \Gamma\bs\bD.$ 

In general, the space $\Gamma\bs\bD$ cannot be equipped with an algebraic structure \cite{GRT13}. Nonetheless, Griffiths conjectured \cite[p.259]{G70Summary} that the closure of the image of a period
map is naturally a quasiprojective variety, which was proven by Griffiths assuming the image is proper, Sommese \cite{Som75} in the case of isolated singularities, and in general in \cite[Theorem 1.1]{BBT23}. We call $\ol{\phi((X\bs D)^\an)}$ a \emph{period image}. One may therefore intrinsically define period images as closed Griffiths transverse algebraic subvarieties of $\Gamma\bs\bD$.

In the classical case of Shimura varieties such as the moduli space of principally polarized abelian varieties $\cA_g$, the space $\Gamma\bs\bD$ itself has an algebraic structure, and is therefore a universal period image. The celebrated work of Baily--Borel \cite{BB66} provides a canonical projective compactification.  In \cite[\S9,10]{G70Summary}, Griffiths laid out a conjectural picture which described how an analogue of the Baily--Borel compactification for arbitrary period images might exist.

Our first main result is to construct such a compactification (see \Cref{thm:bailyborel} for the precise statement).  As in the classical case, it satisfies the following extension property:

\begin{thm}\label{thm:BBmain1}
Let $Y$ be a period image. Then there exists a functorial projective compactification $Y^{\BB}$ such that for any log smooth algebraic space $(\testsp,D_\testsp)$, any morphism $\testsp\bs D_\testsp\to Y$ for which the resulting morphism $(\testsp\bs D_\testsp)^\an\to \Gamma\bs\bD$ is locally liftable\footnote{The local liftability is equivalent to the condition that $(\testsp\bs D_\testsp)^\an\to\Gamma\bs\bD$ be the period map of a variation.  In fact, by the
definability of period maps \cite[Theorem 1.3]{bkt} and the definable Chow theorem of Peterzil--Starchenko \cite[Corollary 4.5]{definechow}, such a morphism is equivalent to a variation on $\testsp\bs D_\testsp$ whose period map factors through $Y^\an$.  If $\Gamma$ is torsion-free the local liftability condition is automatic.} extends to a morphism $\testsp\to Y^{\BB}$. 

\end{thm}
In fact, a stronger extension property holds with respect to analytic maps from punctured polydisks---see \Cref{thm:borelextmain} below---which was proven in the classical case by Borel \cite{borelextend}.

\def\Grif{\operatorname{Griff}}
\subsubsection{Griffiths bundle}
Returning to the variation on $X\bs D$, there is a natural line bundle which descends to a polarization on $Y$ and clarifies the statement of \Cref{thm:BBmain1}.  It is provided by the Griffiths bundle:
$$L_{X\bs D}=\bigotimes_p \Exterior^{\rk F^pV_{\cO}} F^pV_{\cO}.$$
Importantly, this can be realized as the deepest piece of the Hodge filtration of the variation $\Grif(V)\coloneqq  \bigotimes_p \Exterior^{\rk F^pV_{\cO}} V$.

Each power $L_{X\bs D}^n$ of the Griffiths bundle has a natural extension via the nilpotent orbit theorem of Schmid \cite{Schmid} to a line bundle $(L^n_{X\bs D})_X$ on all of $X$ which is nef (and compatible with tensor power) if the local monodromy is unipotent and big if the period map is in addition generically immersive.   It is essentially conjectured in \cite{GGLR17} that $L_X$ is semiample if the local monodromy is unipotent. We prove this conjecture, and use it to give an intrinsic characterization of $Y^{\BB}$ as follows.

It is easy to show \cite[Lem 6.12]{BBT23} that a power of $L_{X\bs D}^m$ descends to a line bundle $L^{(m)}_Y$. We define a section of $L^{(mk)}_Y$ to have \textit{moderate growth} if its pullback to $X\bs D$ for some (hence any) $(X,D)$ extends to a section of $L_X^{mk}$. Finally, we define the ring of moderate growth sections
$B_Y\coloneqq  \bigoplus_k H^0_{\bdd}(Y,L^{(mk)}_Y)$. We prove in \Cref{thm:bailyborel} the following more precise statement, which was originally suggested by Griffiths \cite[(10.7) Remark]{G70Summary}:

\begin{thm}[\Cref{thm:bailyborel}]\label{thm:BBmain2}
Let $Y$ be a period image.  Then

\begin{enumerate}

\item The ring $B_Y$ is finitely generated, $Y^{\BB}\coloneqq  \proj B_Y$ is a projective compactification of $Y$ such that \Cref{thm:BBmain1} holds, and the ample bundle $\cO_{Y^\BB}(n)$ (for sufficiently divisible $n$) naturally restricts to $L_Y^{(n)}$.  
\item Under the extension $\testsp\to Y^\BB$ from \Cref{thm:BBmain1}, $\cO_{Y^\BB}(n)$ (for sufficiently divisible $n$) pulls back to $(L^n_{\testsp\bs D_\testsp})_\testsp$.

\end{enumerate}

\end{thm}
\begin{cor}\label{cor:BB}Let $(X,D)$ be a log smooth algebraic space and $V$ a polarizable integral pure variation of Hodge structures on $X\bs D$ with unipotent local monodromy.  Then the Griffiths bundle $L_X$ of $V$ is semiample.
\end{cor}

It is easy to see that $Y^\BB$ is uniquely determined up to normalization by the properties that the Griffiths bundle extends amply and \Cref{thm:BBmain1} with the compatibility in \Cref{thm:BBmain2}(2) is satisfied.  As in the classical case, it follows from the construction that $Y^\BB$ is stratified by locally closed subvarieties equipped with polarizable variations of Hodge structures with quasifinite period maps.\footnote{This stratification and the associated variations are more complicated than those described in \cite{GGLR17}.  The stratification of $Y^\BB$ is not clearly canonical, and the variations mentioned above which descend from boundary strata in $X$ will only be pieces of a certain tensor operation applied to the limit mixed Hodge structure. This is ultimately due to the lack of a canonical graded polarization on the limit mixed Hodge structure.}

\subsubsection{Past work}
Satake \cite{satake} first constructed compactifications of Siegel modular varieties as ringed spaces, and suggested that they should in fact be analytic spaces.  This was confirmed by Baily \cite{baily}, who moreover proved the compactifications were projective varieties, given as $\proj$ of a finitely generated ring of automorphic forms.  Baily--Borel \cite{BB66} then constructed the analogous projective compactification of any Shimura variety, building on further work of Satake \cite{satake2,satake3}.

Shortly thereafter, Griffiths \cite[\S9]{G70Summary} realized a full compactification of $\Gamma\bs\bD$ in the non-classical case is too much to hope for, and conjectured the existence of a \emph{partial} compactification of $\Gamma\bs \bD$ with an extension property with respect to Griffiths transverse morphisms from log smooth sources as in \Cref{thm:BBmain1}.  This idea provided motivation for the development of the theory of degenerations of Hodge structures (e.g. \cite{Schmid,steenbrinklimit,kashiwara,cks}) filling out a conjectural picture originally due to Deligne.  Attempts have been made to construct such a partial compactification in some special cases, for example in the weight two case by Cattani--Kaplan \cite{CKextend}, where the extension property was demonstrated for one-parameter period maps.  A different perspective was proposed by Kato--Usui \cite{katousui}, who described a conjectural partial compactification of $\Gamma\bs\bD$ in the category of logarithmic manifolds with the property that the closure of the image of any period map would be a proper algebraic space \cite{usuiextend} that would be more closely analogous to the toroidal compactifications of \cite{amrt}.  This line of inquiry has been taken up recently by Deng \cite{deng2021extensionperiodmapspolyhedral,deng2025generalizedtoroidalcompletionperiod}, Deng--Robles \cite{deng2023completiontwoparameterperiodmaps}, and Deng--Tsimerman \cite{deng2025generalizedtoroidalcompletionperiod}.

As mentioned above, Griffiths \cite[(10.7) Remark]{G70Summary} also explicitly suggested an alternative generalization of the work of Baily--Borel (focusing on its relation to automorphic forms) would be to show that the ring of moderate growth sections of the Griffiths bundle on $Y$ is finitely generated, and that its $\proj$ provides a compactification of the period image.  \Cref{thm:BBmain2} confirms this expectation.  The closely related question of the semi-positivity of vector bundles of Hodge-theoretic origin has been considered by various authors starting with Griffiths \cite{G70III} and continuing with, for example, Fujita, Zucker, and Kawamata \cite{fujitasemi,zuckersemi,kawamatasemi} (see also the references below regarding the canonical bundle formula).  

The question of the semiampleness of the Griffiths bundle has been revived recently by Green, Griffiths, Laza, and Robles in \cite{GGLR17}, where Baily--Borel type compactifications and their connection to other compactifications arising from moduli theory (such as KSBA compactifications) are discussed.  Their work, together with subsequent work of Green--Griffiths--Robles \cite{GGRtorsion}, has had an important influence on this one.  Most notably, Green--Griffiths--Robles establish the torsion combinatorial monodromy of the Griffiths bundle in \cite[Thm 5.21]{GGRtorsion} (see \Cref{thm:GGR}), which is a crucial ingredient.  Implicit in their work is also the idea that the Griffiths bundle $L_X$ is flat (with no residues) with respect to the natural logarithmic connection of the Deligne extension along any subvariety $Z$ of $X$ on which $L_X$ is numerically trivial (see \Cref{rem:GGR flat}).  For us, this is ultimately upgraded to the existence of the local period maps described in (3) of the proof outline below.  Finally, the $\dim X=2$ case of \Cref{cor:BB} is proven in \cite{GGLR17} and \cite{GGRtorsion}; the one-dimensional cases of \Cref{thm:BBmain1} and \Cref{thm:BBmain2} are easy.  

\subsubsection{Applications}

As a corollary, we obtain a Baily--Borel compactification of any moduli space of polarized varieties with a local Torelli theorem.

\begin{cor}[of \Cref{thm:bailyborel}]\label{cor:intro griffiths moduli}  Let $\cY$ be a reduced separated Deligne--Mumford stack equipped with a quasifinite period map.  Then the coarse space $Y$ has a compactification $Y^\BB$ for which some power $L_Y^{(n)}$ of the Griffiths bundle extends to an ample bundle $\cO_{Y^\BB}(n)$
and such that for any morphism $g:\testsp\bs D_\testsp\to \cY$
for a log smooth algebraic space $(\testsp,D_\testsp)$, the map on coarse spaces extends to $\bar g:\testsp\to Y^\BB$ with the property that $\bar g^*\cO_{Y^\BB}(n)$ pulls back to $(L^n_{\testsp\bs D_\testsp})_\testsp $.
\end{cor}
\Cref{cor:intro griffiths moduli} for instance applies to any moduli stack of polarized varieties with an infinitesimal Torelli theorem, such as Calabi--Yau manifolds and most hypersurfaces (see \cite{BBT23}).

\subsection{ b-semiampleness conjecture} 

Let $(Y,\Delta)$ be a pair with log canonical singularities and $f \colon Y\to X$ a projective morphism with connected fibers to a normal variety such that $K_Y+\Delta\sim_\bQ f^*L$ for a $\bQ$-Cartier divisor $L$ on $X$.  Such a morphism is an example of an \emph{lc-trivial fibration} (\Cref{lc-trivial.def}); the condition essentially means the fibers are log Calabi--Yau pairs.

The \emph{canonical bundle formula},
proven in increasing generality in
\cite{Kod1,Kod2,Fuj86,Kaw1998,FM00, AmbroPhD, Amb04,
Amb05,FG14},
implies that for an lc-trivial fibration, we can write
\[K_Y+\Delta\sim_\bQ f^*(K_X + B_X + M_X)\]
where:
\begin{itemize}
\item $B_X$ is the \emph{boundary divisor}, which is a $\bQ$-divisor encoding the singularities of the degenerate fibers of $f$ in codimension $1$; and
\item $M_X$ is the \emph{moduli part}, which is a $\bQ$-b-divisor which is  b-nef and encodes the variation of the generic fiber of the family.  
\end{itemize}
Recall that a  b-divisor is essentially an assignment of a divisor to any sufficiently high birational model of $X$.  In particular, there is a modification $\pi \colon X'\to X$ and an lc-trivial fibration model $f' \colon Y'\to X'$ of the base-change of $f$ for which $M_{X'}$ is a nef $\bQ$-divisor such that for any further modification $\pi' \colon X''\to X'$ the moduli part pulls back, $M_{X''}=\pi'^*M_{X'}$.
Such a model is called an Ambro model.

Our second main result is to prove the  b-semiampleness conjecture,
formulated in increasing generality by Mori, Kawamata, Shokurov, Ambro, and Prokhorov--Shokurov; see \cite{Mor87,AmbroPhD,PS09} and references therein.

\begin{thm}[b-semiampleness]\label{thm:Bsemimain}
Let $f \colon (Y,\Delta)\to X$ be an lc-trivial fibration from a pair $(Y,\Delta)$
such that $\Delta$ is effective over the generic point of $X$.
Then the moduli part $\bfM$ is  b-semiample.  
\end{thm}
Equivalently, $M_{X'}$ is semiample for any Ambro model $f' \colon Y'\to X'$ of $f$. Note that the algebricity of $Y$ and $X$ can be dropped. Indeed, in \Cref{thm:bsemiampleanal}, we show the  b-semiampleness of the moduli part for projective morphisms of complex analytic spaces. Some immediate applications of the  b-semiampleness conjecture are discussed in \S \ref{sec:applicationbsemiampleness}.  

On a suitably chosen alteration of $X$, the moduli part is the Schmid extension of the lowest Hodge filtration piece of the variation of Hodge structures on middle cohomology of the generic part of the family; see \Cref{canonicalbundleformula}.
Thus, after some reductions, \Cref{thm:Bsemimain} is deduced from the following purely Hodge-theoretic result.  By a CY variation of Hodge structures, we mean a variation of Hodge structures whose deepest nonzero Hodge filtration piece has rank one.  We refer to this deepest piece (or its Schmid extension) as the \emph{Hodge bundle}.
\begin{thm}[\Cref{thm:semiample Hodge}]\label{thm:Hodgesemimain}
Let $(X,D)$ be a proper log smooth algebraic space, $V$ a polarizable integral pure CY variation of Hodge structures on $X\bs D$ with unipotent local monodromy, and $M_X$ the Hodge bundle on $X$.  If $ M_X$ is integrable with torsion combinatorial monodromy, then it is semiample. 
\end{thm}

The integrability condition (see \Cref{defn:integrable}) means that for any subvariety $Z\subset X$ for which the Hodge bundle is not big, there is some piece (defined over $\bQ$) of the limit mixed Hodge structure variation on $Z$ which contains the Hodge bundle and whose period map is not generically finite.  The torsion combinatorial monodromy condition (see \Cref{defn:torsion combo}) means that for any nodal curve $g \colon C\to X$ for which $g^*M_X$ is numerically trivial (hence torsion on each component), it is in fact torsion.  Both conditions are needed in the statement of \Cref{thm:Hodgesemimain}---see \Cref{eg:int not tor} and \Cref{eg:tor not int}---and therefore the proof that the conditions are satisfied in the case of \Cref{thm:Bsemimain} relies on the geometry of log Calabi--Yau pairs.

\subsubsection{Past work}

Kodaira first studied the canonical bundle formula in the context of the classification of surfaces.
In particular, the first instance of the formula is Kodaira's formula for the canonical bundle of a smooth minimal elliptic surface $S \to C$ \cite{Kod1,Kod2}.
Famously, in this case 
$M_C = \frac{1}{12} j^* \mathcal{O}_{\mathbb P ^1}(1)$, where $j$ denotes the $j$-map $j \colon C \to \mathbb P ^1$.
Later, Fujita extended the formula to smooth varieties admitting an elliptic fibration \cite{Fuj86}.
In the case of elliptic fibrations over higher-dimensional bases, the $j$-map is not necessarily a morphism;
thus, since Fujita's work, it became clear that to ensure positivity properties of the Hodge bundle, one must pass to a suitable higher birational model of the base.

In \cite[Rmk.~5.15.9.(ii)]{Mor87}, Mori suggested the connection between the moduli part and the Hodge bundle of a family of Calabi--Yau varieties whose period domain is a classical Shimura variety. 
Fujino explored this direction in \cite{Fuj03} and showed that the moduli divisor of an abelian- or K3-fibration is  b-semiample; see \cite{Ueno1978} for a precursor of these results. Fujino's work has then been extended by Kim to the case of primitive symplectic varieties \cite{Kim25}. More broadly, in \cite[9.12]{Mor87}, Mori suggested the semiampleness of the moduli divisor for more general Calabi--Yau fibrations; see also \cite[Conj.~6]{AmbroPhD} and \cite[Conj.~7.13.1]{PS09}. 

So far, the progress on the b-semiampleness conjecture has been limited. In \cite{Kaw1997, Kaw1998}, Kawamata showed that the moduli part is b-nef, and settled its b-semiampleness in relative dimension 1; see also \cite{PS09}. Then, the works \cite{Fil20,babwild} settled the conjecture in relative dimension 2, complementing Fujino's work. Independent works of Ambro and Fujino--Mori explored more general lc-trivial fibrations and settled weaker positivity properties of the moduli divisor;
see \cite{Amb04,Amb05, FM00}.
While these earlier results only hold for fibrations with generically klt singularities, Fujino and Gongyo extended these results to fibrations with lc singularities \cite{FG14}.
The idea of this latter work is to relate the Hodge bundle of the original fibration to the Hodge bundle of the fibration induced by the source of lc singularities.
This is a key idea in the proof of \Cref{thm:Bsemimain}.

In recent years, there has also been progress in the study of lc-trivial fibrations when the general fiber is not normal;
see \cite{MR4458545,MR4458546}.
In our work, we circumvent this problem by reducing to the fibration corresponding to minimal lc centers, which are necessarily normal. Another important idea in this work is the study of the Hodge bundle along the subvarieties where it fails to be big.
This direction has been originally explored in \cite{FL19,Flo23}.
In particular, in \cite{Flo23}, the author pursues similar ideas as in this work in considering the gluing of various period morphisms along snc configurations of subvarieties. 

The relation between Baily--Borel compactifications of period images and the b-semiampleness conjecture has been clear to the experts, including, e.g., Ambro, Birkar, Fujino, Kawamata, Koll\'{a}r, Mori, Prokhorov, and Shokurov.  In the last decade, these ideas have been popularized by Laza in particular.

\subsubsection{Applications}

We also prove the existence of a Baily--Borel compactification as in \Cref{thm:BBmain2} for the Hodge bundle, but subject to a normality condition.

\begin{thm}[\Cref{thm:BBHodge}]\label{thm:BBHodgemain}
Let $\cY$ be a reduced separated normal Deligne--Mumford stack with a polarizable integral pure CY variation $V$.  Assume that the Hodge bundle $M_\cY$ is strictly nef, integrable, and has torsion combinatorial monodromy.  Let $Y$ be the coarse space of $\cY$ and $M_Y^{(m)}$ the descent of some power of $M_\cY$.    Then

\begin{enumerate}
\item The ring $C_Y\coloneqq  \bigoplus_k H^0_{\bdd}(Y,M^{(mk)}_Y)$ is finitely generated, and $Y^{\BBH}\coloneqq  \proj C_Y$ is a normal projective compactification of $Y$ for which the ample bundle $\cO_{Y^\BBH}(n)$ (for sufficiently divisible $n$) restricts to $M_Y^{(n)}$.
\item For a log smooth algebraic space $(\testsp,D_\testsp)$ and any morphism $g:\testsp\bs D_\testsp\to \cY$, the morphism on coarse spaces extends to $\bar g:\testsp\to Y^\BBH$ and $\bar g^*\cO_{Y^\BB}(n)$ is identified with the Schmid extension $(M_{\testsp\bs D_\testsp}^n)_\testsp$.  Moreover, $Y^\BBH$ is the unique normal compactification of $Y$ for which some power of the Hodge bundle $M_Y^{(n)}$ extends to an ample line bundle and satisfies this property. 

\end{enumerate}
\end{thm}

Experts, including Koll\'ar and Shokurov, conjectured that the moduli part of an lc-trivial fibration should be the pullback of an ample $\bQ$-divisor along a rational map to a compactified moduli space of the general fibers, as in the case of elliptic fibrations;
see, e.g. \cite[\S 8.3.8]{Koll07} and \cite{shokurov13}. 
The compactification $Y^\BBH$ accomplishes it from a Hodge-theoretic viewpoint.

Functorial compactifications of moduli spaces of varieties with positive or negative canonical bundle have been extensively studied in the literature, see, e.g., the monographs \cite{kbook, Xu2025}. The case of trivial canonical bundle was a long-standing open question; cf.~for instance \cite{KX20,Bir22, babwild,BL24}. As a corollary, we prove the existence of a Hodge-theoretic compactification of the moduli space of smooth Calabi–Yau varieties on which the Hodge line bundle extends amply, which was conjectured by several authors, including, e.g., Shokurov in the 90s, Koll\'{a}r, and in more recent years Laza and Odaka.

\begin{cor}[cf. \Cref{cor:CYmod}]\label{cor:introCYmod}  Let $\cY$ be a moduli stack of 
polarized smooth Calabi--Yau varieties.
Then the coarse space $Y$ has a unique normal compactification $Y^\BBH$ for which some power $M_Y^{(n)}$ of the Hodge bundle of the variation of Hodge structures on middle cohomology extends to an ample bundle $\cO_{Y^\BBH}(n)$ and such that for any family $g:\testsp\bs D_\testsp\to \cY$
for a log smooth algebraic space $(\Theta,D_\Theta)$, the map on coarse spaces extends to $\bar g:\testsp\to Y^\BBH$ with the property that $\bar g^*\cO_{Y^\BBH}(n)$ pulls back to $(M^n_{\testsp\bs D_\testsp})_\testsp$.
\end{cor}

As in \Cref{thm:BBmain2}, the compactifications from \Cref{thm:BBHodgemain} and \Cref{cor:introCYmod} are stratified by varieties which are naturally equipped with CY variations with quasifinite period maps,
so their codimension can be estimated from the numerics of the limit mixed Hodge structure---see \Cref{rem:CYnumerics}.  In fact, along codimension one strata, these variations are the transcendental part of the middle cohomology of the minimal lc center, or \emph{source}; in general, they are obtained by iterating this procedure. In \cite[p.5]{shokurov13}, Shokurov predicted that the image of this compactified moduli map can be described in terms of the equivalence relation determined by having crepant birational sources; this is true for $Y^\BBH$, albeit only up to finite ambiguity, and we do not pursue this here.
More on the Hodge theory of sources will appear in \cite{Laza2025}.
The smoothness assumption in \Cref{cor:introCYmod} may be dropped at the expense of taking the normalization of the coarse space $Y$---see \Cref{cor:CYmod}.  It would be interesting to compare the compactification $Y^{\BBH}$ with the compactifications in \cite{KX20} and \cite[Theorem 1.14]{Bir22} and the conjectural compactifications in \cite[Conj.~B.1]{Odaka2022} and in \cite[\S 4.1.3]{Spotti25}. 

Note that in the classical case where $Y$ is a Shimura variety, we may form both $Y^\BB$ and $Y^\BBH$, and there is no difference between them.  Indeed, for any level $\leq 2$ polarizable variation of Hodge structures $V$ with Hodge bundle $F^mV_\cO$, the Griffiths bundle of $V$ is a power of the Hodge bundle of $\Exterior^{\rk F^mV_\cO}V$, since $F^mV_\cO\cong (\gr_F^{m-2}V_\cO)^\vee$ implies $\det F^mV_\cO\cong \det F^{m-1}V_\cO$ as $F^{m-2}V_\cO=V_\cO$ has trivial determinant.  This also applies to the variation on middle cohomology for a family of surfaces.  In general, for instance, for many moduli spaces of higher-dimensional Calabi--Yau varieties as in \Cref{cor:CYmod}, they can be different---see \Cref{eg:BBvsBBH}.  There is however always a morphism $Y^\BB\to Y^\BBH$.

We deduce a version of the Borel extension theorem for both $Y^\BB$ and $Y^\BBH$:
\begin{thm}[\Cref{thm:bbextensino}]\label{thm:borelextmain}  Let $Y$ be as in \Cref{thm:BBmain2} (resp. \Cref{thm:BBHodgemain}), and $Y^\BB$ (resp. $Y^\BBH$) its Baily--Borel compactification with respect to the Griffiths (resp. Hodge) bundle.  Then any analytic morphism $(\Delta^*)^k\to Y^\an$ for which $(\Delta^*)^k\to\Gamma\bs\bD$ is locally liftable extends to a morphism $\Delta^k\to Y^{\BB,\an}$ (resp. $\Delta^k\to Y^{\BBH,\an}$).
\end{thm}

A version of Borel extension showing that any holomorphic map $(\Delta^*)^k\to Y^\an$ to a period image\footnote{In fact, to a variety admitting a complex variation of Hodge structure with period map with discrete fibers.} satisfying the local liftability extends meromorphically with respect to any compactification was proven by Deng \cite{dengpicard} and also follows directly from the definability of the period map \cite{bkt} (see the proof of \Cref{thm:bbextensino}).

Lastly, \Cref{thm:Bsemimain} has applications to foundational statements in the MMP.
Among others, it allows one to descend lc singularities along lc-trivial fibrations and to formulate adjunction and inversion thereof for arbitrary lc centers.
For the precise statements, we refer the reader to \S \ref{sec:applicationbsemiampleness}.

\subsection{Proof outline}
\def\quotientY{Z}

In the classical case, Baily--Borel's proof \cite{BB66} first builds the quotient space $Y^\BB$ set-theoretically as a union of Shimura varieties, then endows it with a sheaf of analytic functions provided by certain modular forms, and upgrades this to an algebraic structure using GAGA. In this more general setting, a quotient space $Y^\BB$ which is stratified by period images may be constructed\footnote{Although there are some subtleties---see \Cref{lem:equiv reln stuff}, where we must use \Cref{lem:unpol orbit}.}, and each stratum is algebraic by \cite{BBT23}.  On the one hand, there are algebraic sections of the Griffiths bundle on strata which separate points but are not ``Hodge-theoretic'' since they ultimately come from algebraic geometry rather than universally on $\Gamma\bs\bD$, so it is not clear how to lift them to a neighborhood of a stratum.  On the other hand, analytic ``Hodge-theoretic'' sections may be constructed locally in $Y^\BB$ around any particular stratum, but their behavior off of that stratum is unclear.  In particular, it is not clear there are enough such sections locally at a point of a stratum to separate points in a larger stratum specializing to it.

We blend these two perspectives and work inductively, adding strata one at a time, starting with the highest-dimensional stratum, showing (i) that the resulting space is a definable analytic space, and (ii) using definable GAGA \cite{BBT23} to algebraize it.  Furthermore, we show this procedure can be carried out in general for a CY variation $V$ given the hypotheses in \Cref{thm:Hodgesemimain}, so the proofs of \Cref{thm:BBmain2} and \Cref{thm:Hodgesemimain} are intertwined.  Step (i) is achieved by the methods of \cite[Theorem 5.4]{BBT23} (see \Cref{thm:vanishing}) combined with the construction of ``Hodge-theoretic'' sections which exist definably locally on the stratum using a certain part of the period map of the stratum which lifts to a neighborhood.   

A more precise summary is as follows:

\begin{enumerate}[itemsep=.5em]

\item\label{step:begin} We start with a proper log smooth algebraic space $(X,D)$ equipped with a variation of Hodge structures $V$ on $X\bs D$.  The algebraic space $X$ comes equipped with a natural locally closed stratification $\{X_\Sigma\}$.  We also form an auxiliary CY variation $E$ on $X\bs D$ with Hodge bundle $M_X$:  
\vskip.5em
\begin{enumerate}[leftmargin=7em]
\item[(Thm \ref{thm:BBmain2})]  We take $E=\Sym^N\Grif(V)$ for an appropriate $N$, so $M_X$ is the $N$th power of the Griffiths bundle of $V$.
\item[(Thm \ref{thm:Hodgesemimain})]  We take $E=V$.
\end{enumerate}

\item\label{step:relations} We define a proper algebraic equivalence relation $R$ on $X$ and show that we may modify $(X,D)$ so as to make $R$ as nice as possible.   This allows us to construct a quotient space $\quotientY$ as a (definable) topological space with an induced stratification $\{\quotientY_T\}$. 
\vskip.5em
\begin{enumerate}[leftmargin=7em]
\item[(Thm \ref{thm:BBmain2})]  Let $Y$ be the image of the period map associated to $V$.  The closure of the equivalence relation on $X\bs D$ defining the map $X\bs D\to Y$, together with the relation of being connected by curves of degree 0 with respect to $M_X$, generates an equivalence relation $R$ on $X$.
\item[(Thm \ref{thm:Hodgesemimain})]  We take $R$ to be the relation of being connected by curves of degree 0 with respect to $M_X$.  
\end{enumerate}
Both of these relations must be proven to be algebraic.

\item\label{step:boundary period map} Each stratum $X_\Sigma$ of $X$ is naturally equipped with a variation of mixed Hodge structures $E_\Sigma$ coming from the part of the limit mixed Hodge structure which is invariant under local monodromy.  There is a smallest subquotient $E_\Sigma^\trans$ containing $M_X|_{X_{\Sigma}}$,
which is a pure variation, and a smallest quotient $E_\Sigma^\mini$ containing $M_X$, which is a mixed variation.  We naturally have $E_\Sigma^\trans\subset E_\Sigma^\mini$, and the underlying local systems $E_{\Sigma,\bQ}^\trans,E_{\Sigma,\bQ}^\mini$ extend to a tubular neighborhood $X_\Sigma\subset \fT_X(\Sigma)\subset X$.

The full period map $\widetilde{X_\Sigma}^{E^\trans_\Sigma}\to\bD_\Sigma$ of $E^\trans_\Sigma$, where $\widetilde{X_\Sigma}^{E^\trans_\Sigma}$ is the minimal cover on which $E^\trans_\Sigma$ is trivialized, factors through some $\widetilde{\quotientY_T}$ but does not lift to $\widetilde{\fT_X(\Sigma)}^{E^\trans_\Sigma}$.  We show that the map $\widetilde{X_\Sigma}^{E^\mini_\Sigma}\to\bP(E_{\Sigma,\bC,x}^\mini)$ obtained by only remembering $M_X\subset E_\Sigma^\mini$ will: (i) have the same (connected components of) fibers along $\widetilde{X_\Sigma}^{E^\mini_\Sigma}$ as the period map $\widetilde{X_\Sigma}^{E^\trans_\Sigma}\to\bD_\Sigma$, and (ii) lift to $\widetilde{\fT_X(\Sigma)}^{E^\mini_\Sigma}$.  This construction provides the definable analytic ``Hodge-theoretic'' sections described above definably locally on $\quotientY$ and crucially uses the fact that we are working with a CY variation, which is why even in the case of \Cref{thm:BBmain2} it is important to consider the auxiliary CY variation $E$.  It is also in this step that, in the case of \Cref{thm:Hodgesemimain}, we use the integrability to ensure (i) and the torsion combinatorial monodromy condition to ensure that $\widetilde{X_\Sigma}^{E^\mini_\Sigma}\to\bP(E_{\Sigma,\bC,x}^\mini)$ has compact connected components of fibers.

\item\label{step:induction} Proceeding inductively, we algebraize larger and larger unions of strata. Suppose $U\subset\quotientY$ is an open union of strata and $\quotientY_T\subset U$ a stratum which is closed in $U$ such that $U'\coloneqq  U\bs \quotientY_T$ has been algebraized. To algebraize $U$, we first use \cite[Theorem 5.4]{BBT23} to produce global sections of a power of $M_X$ which separate fibers of $X\to Z$ over $U'$.  Combining these with the sections of $M_X$ coming from (3) which exist definably locally on $\quotientY_T$ and separate points on $\quotientY_T$, we obtain a definable analytic projective embedding definably locally on $U$ (using the compactness of the fibers of the map in (3)), and the definable analytic structures of the images glue to give a definable analytic structure on $U$.  This then gives an algebraic structure on $U$ by \cite[Theorem 1.3]{BBT23} (see \Cref{thm:image}) to which a power of $M_X$ descends to a line bundle $M_{U}^{(m)}$ by definable GAGA which is moreover ample by the ampleness criterion \cite[Theorem 5.4]{BBT23} again.  By induction, the quotient map $X\to Z$ is algebraic and a power of $M_X$ descends to an ample bundle $M_Z^{(m)}$ on $Z$. 

\item\label{step:finish}  This completes the proof of \Cref{thm:Hodgesemimain}, and \Cref{thm:BBHodgemain} easily follows since in this case the algebraic structure on $Z$ may be taken to be normal. The corresponding statements in \Cref{thm:BBmain2} is more subtle, since the normality assumption is now dropped. Indeed, if $Y$ is normal, the quotient map $X\to \quotientY$, which may be assumed to have connected fibers, admits a preferred algebraization for which the morphism is a fibration, and this algebraic structure is determined by the underlying topology. We show \Cref{thm:BBmain2} via a descent argument, after applying \Cref{thm:BBHodgemain} to the normalization of $Y$. The descent along the normalization is delicate and critically uses the fact that $E$ comes from $V$ via the $\Grif(-)$ construction.

\end{enumerate}

Finally, the  b-semiampleness conjecture (\Cref{thm:Bsemimain}) follows from \Cref{thm:Hodgesemimain},
provided we verify that a ``geometric'' Hodge bundle, i.e., coming from an lc-trivial fibration $f \colon (Y, \Delta) \to X$, is automatically integrable with torsion combinatorial monodromy. To this end, we explore the geometric significance of these two conditions.

\begin{enumerate}\setcounter{enumi}{5}
\item
The key geometric input is the notion of \emph{source} of an lc-trivial fibration, roughly the smallest stratum of $(Y, \Delta)$ dominating the base; see \Cref{defnsource} for the formal definition, and cf.~also \cite[\S 4.5]{kol13}.
We study how the moduli part of $f \colon (Y,\Delta) \to X$ restricts to a prime divisor $D_X \subset X$. Up to an alteration of the base, we identify the restriction to $D_X$ of the moduli part of $f$ with the moduli part of the source $(S, \Delta_S) \to D_{X}$ of the restricted lc-trivial fibration $(Y,\Delta)\times_{D_X} X \to D_X.$

\item The integrability of the Hodge bundle follows from the fact that the variation of the source is maximal if and only if its Hodge bundle is big; cf.~\cite{Amb05}.

\item The torsion combinatorial monodromy is a consequence of the isotriviality of lc-trivial fibrations with torsion moduli part (over a normal base) due to \cite{Amb05}, and the finiteness of  b-representations, proved in increasing level of generality by \cite{NU73,Fuj00,Gon13,HX2011,FG14b}. The latter gives that the group $\mathrm{Bir}(F, \Delta_F)$ of crepant birational automorphisms of the general fiber $(F, \Delta_{F})$ of $(S, \Delta_S) \to D_X$ acts finitely on $H^0(F, \omega^{[m]}_{F}(m\Delta_F))$ for $m \gg 1$. Now, to check torsion combinatorial monodromy, consider a testing nodal curve $C \to X$ such that the moduli part vanishes on each irreducible component of $C$. Up to some technical reductions, the sources of the pullback family $Y_{C} \to C$ are isotrivial families along each irreducible of $C$, connected by crepant birational automorphisms of their fibers over the nodes of the $C$ by Koll\'ar's theory of $\bP^1$-linking \cite[\S 4.4]{kol13}. Hence, the monodromies of multiples of the Hodge bundle factor through the finite  b-representations of $\mathrm{Bir}(F, \Delta_F)$, which entails the torsion of the combinatorial monodromy.
\end{enumerate}

\subsection{Paper outline}

The paper is organized as follows.
\begin{enumerate}
\item[($\S2$)]  We recall the variations of (monodromy invariant) limiting mixed Hodge structures one obtains from a variation $V$ on each stratum of a log smooth space $(X,D)$, and introduce their transcendental and CY-minimal quotient pieces.  We define the integrability and torsion combinatorial monodromy conditions and explain why both conditions are satisfied in the case of the Griffiths bundle.  The latter is a result of \cite{GGRtorsion}. 
\item[($\S3$)]  We discuss the algebraicity of equivalence relations of Hodge-theoretic nature as in \eqref{step:relations} of the proof outline.  We also prove some lemmas regarding refinement of log smooth spaces with respect to these data.  Finally, we construct the maps $\widetilde{X_\Sigma}\to\bP(E_{\Sigma,\bC,x}^\mini)$ from \eqref{step:boundary period map} of the proof outline.
\item[($\S4$)] We prove the semiampleness statement in \Cref{thm:BBmain2} (i.e. \Cref{cor:BB}) and \Cref{thm:Hodgesemimain} as in \eqref{step:induction} of the proof outline.
\item[($\S5$)]  We deduce \Cref{thm:BBHodgemain} and prove \Cref{thm:BBmain2} as in \eqref{step:induction} and \eqref{step:finish} of the proof outline.  We also prove the Borel extension result \Cref{thm:borelextmain}.
\item[($\S 6$)]  We recall the notion of moduli part and source of an lc-trivial fibration. We compare the variation of Hodge structures associated to them, and we make some preliminary reductions for the proof of \Cref{thm:Bsemimain}.
\item[($\S 7$)] Using the preparations from $\S6$, we verify the integrability and torsion combinatorial monodromy conditions for the moduli part to deduce \Cref{thm:Bsemimain}. Then we discuss applications of the  b-semiampleness conjecture in birational geometry (\S \ref{sec:applicationbsemiampleness}) and for the moduli theory of log Calabi--Yau pairs (\S\ref{sec:moduli}).
\end{enumerate}
\subsection*{Notation}
\begin{itemize}
\item Throughout, analytic spaces, definable analytic spaces, and algebraic spaces are always taken over $\bC$ and to be separated.  Algebraic spaces are always of finite type over $\bC$.
\item By a log smooth algebraic space or snc pair $(X,D)$ we mean the datum of a smooth algebraic space $X$ together with a divisor $D\subset X$ with simple normal crossings.  By a morphism $f:(X,D)\to (X',D')$ of log smooth algebraic spaces we mean a morphism $f:X\to X'$ with $ f^{-1}(D')^\reduced\subset D$.
\item Throughout by a fibration we mean a proper morphism $f:X\to Y$ between normal spaces such that the pullback map $\cO_Y\xrightarrow{\cong}f_*\cO_X$ is an isomorphism.
\item CY variations play the central role in the paper, so we refer to them by the letter $V$ in $\S2$-4,6 and typically use the letter $L$ for the Hodge bundle.  In the case of \Cref{thm:BBmain2} where the relevant CY variation is obtained by the $\Grif(-)$ construction, we refer to the original variation as ${_\orig V}$, so $V=\Grif(_\orig V)$. Most of $\S5$ is devoted to the Griffiths case, so we once again reserve $L$ for the Griffiths bundle and use $M$ for the Hodge bundle. 
\end{itemize}

\subsection*{Acknowledgements} We would like to thank Florin Ambro, Fabio Bernasconi, Harold Blum, Yohan Brunebarbe, Philip Engel, Christopher D. Hacon, Giovanni Inchiostro, Radu Laza, Konstantin Loginov, Yuji Odaka, Colleen Robles, 
Vyacheslav V. Shokurov, Roberto Svaldi, and Chenyang Xu for useful mathematical discussions.
We thank the Oberwolfach Research Institute for Mathematics (MFO), the Hausdorff Institute for Mathematics in Bonn (HIM), and the American
Institute of Mathematics (AIM) for providing a supportive research environment. 

\def\Coh{\operatorname{Coh}}
\section{Hodge-theoretic preliminaries}

\subsection{Definable analytic spaces}

We shall use the notion of definable analytic spaces from \cite{BBT23} freely. As in the majority of that paper, we shall only ever work with the o-minimal structure $\bR_{\an,\exp}$.  For the convenience of the reader, we reproduce here the three main results from \cite{BBT23} we will be using.

\begin{thm}[{Definable GAGA, \cite[Theorem 1.4]{BBT23}}]\label{thm:GAGA} Let $X$ be an algebraic space and $X^\define$ the associated definable analytic space.  The definabilization functor $\Coh(X)\to\Coh(X^\define)$ is fully faithful, exact,
 and its essential image is closed under subobjects and quotients.
\end{thm}

\begin{thm}[{Definable images, \cite[Theorem 1.3]{BBT23}}]\label{thm:image} Let $X$ be an algebraic space and $\phi:X^\define\to\cZ$ a proper morphism of definable analytic spaces.  Then there is a factorization
\[\begin{tikzcd}
X^\define\ar[dr,swap,"f^\define"]\ar[rr,"\phi"]&&\cZ\\
&Y^\define\ar[ur,swap,"\iota"] &
\end{tikzcd}\]
where $f:X\to Y$ is a proper dominant\footnote{Here we mean ``scheme-theoretically'' dominant, that is, $X\to Y$ is surjective on points and $\cO_Y\to f_*\cO_X$ is injective.} morphism of algebraic spaces and $\iota:Y^\define \to \cZ$ is a closed embedding of definable analytic spaces.  Moreover, $f$ is uniquely determined as a morphism with fixed source.
\end{thm}

\begin{setup}\label{setup}
Let $L_Y$ be a line bundle on an algebraic space $Y$ with the following property.  For every reduced closed subspace $Z \hookrightarrow Y$ and any proper log smooth algebraic space $(X,D)$ with a proper birational morphism $\pi:X\bs D\to Z$, the pullback $L_{X\bs D}\coloneqq  \pi^*L_Z$ of the restriction $L_{Z}$ extends to a nef and big line bundle $L_{X}$ on $X$.  Moreover, for any two such $Z,Z'$ with corresponding $(X,D),(X',D')$ and a morphism $g:(X',D')\to (X,D)$ with $\pi\circ g|_{X'\bs D'}=\pi'$ we have $g^*L_X\cong L_{X'}$.
\end{setup}

\begin{defn}\label{vanishingsect}
Assume \Cref{setup}.  Given a closed subscheme $Z \hookrightarrow Y$, we say a section $s$ of $L^n_Z$ \emph{vanishes at the boundary} if for some (hence any) $(X,D)$ as above resolving and compactifying the reduction $Z^\mathrm{red}$, the section $s$ pulls backs and extends to a section of $L_{X}^n(-D)$.  We let $H^0_{\mathrm{van}}(Z,L^n_Z)\subset \Gamma(Z,L^n_Z)$ denote the linear subspace of sections vanishing at the boundary. If $Z$ is reduced, then $H^0_{\mathrm{van}}(Z,L^n_Z)$ is finite-dimensional as $H^0_{\mathrm{van}}(Z,L^n_Z)$ injects into $H^0(X,L_{ X}^n)$.
\end{defn}

\begin{thm}[{\cite[Theorem 5.4]{BBT23}}]\label{thm:vanishing}
Assume \Cref{setup}.  Then $Y$ is a scheme and $L_Y$ is an ample line bundle. Moreover, for every $n \gg 1$, the natural morphism $Y \rightarrow \mathbb{P}(H^0_{\mathrm{van}}(Y,L_Y^n)^\vee)$ is defined everywhere and is a locally closed embedding.
\end{thm}

Finally, the following notions will be useful; see \cite[\S3]{BBTshaf} for more details.

\begin{defn}\label{defn:pi1definable}For a definable analytic space $\cX$ with a choice of basepoint $x\in\cX$, a definable analytic space $\cP$ admitting an action of $\pi_1(\cX,x)$ by definable analytic automorphisms, and a covering space $\widetilde{\cX}\to \cX^\an$, we say a $\pi_1(\cX,x)$-equivariant analytic morphism $\phi:\widetilde{\cX}\to \cP^\an$ is \emph{$\pi_1$-definable analytic} if its restriction to any continuous lift of a definable open $\cU\subset\cX$ is definable analytic.  A \emph{$\pi_1$-definable analytic coherent sheaf} $\cE$ on $\widetilde{X}$ is an analytic coherent sheaf which is equipped with the structure of a definable analytic coherent sheaf on any such lift $\cU$, compatibly with respect to intersections and the $\pi_1(\cX,x)$ action. It is clear that this is equivalent to a definable analytic coherent sheaf on $\cX$.    
\end{defn}
Observe that if $\mathfrak{X} $ is a definable topological space and $\widetilde{\mathfrak{X}}$ a covering space, a $\pi_1(\mathfrak{X},x)$-equivariant sheaf of $\bC$-algebras $\cO$ on $\widetilde{\mathfrak{X}}$ (with respect to covers pulled back from $\mathfrak{X}$), which is the structure sheaf of a definable analytic space on any continuous lift of a definable open subset of $\mathfrak{X}$, is equivalent to a definable analytic space structure on $\mathfrak{X}$.

\subsection{Trivializing covers of local systems}\label{subsect:covers}
For a local system $V$ on an analytic space $X$, we denote by $\widetilde {X}^V\to X$ the minimal covering space of $X$ on which $V$ is trivialized---that is, the cover corresponding to the quotient of $\pi_1(X,x)$ given by the monodromy representation $\pi_1(X,x)\to \mathrm{GL}(V_x)$ for some choice of basepoint $x$.  If the analytic space is $X^\an$ for an algebraic space $X$, we denote the cover by $\widetilde{X}^V$ (as opposed to $\widetilde{X^\an}^V$).


\subsection{DR-neighborhoods}\label{sect:DR}
Let $X$ be a definable topological space. For any locally closed definable topological space $Z$, we say that a neighborhood $Z\subset \fT_X(Z)\subset X$ is a \textit{DR-neighborhood} of $Z$ if it has a strong deformation retraction onto $Z$.

\begin{lemma}\label{lem:subsgood}Let $X$ be a definable topological space, $i:Z\to X$ the inclusion of a locally closed definable subspace, $j:X\bs Z\to X$ the inclusion of the complement, $\mathfrak{T}(Z)\coloneqq  \fT_X(Z)$ a DR-neighborhood of $Z$, and $E$ a local system on $X\bs Z$.  For any subsheaf $G_Z\subset i^*j_*E$ which is a local system, there is a unique subsheaf $G(Z)\subset j_*E|_{\mathfrak{T}(Z)}$ which is a local system for which $i^*G(Z)=G_Z$.
\end{lemma}

\begin{proof}
Without loss of generality, we may assume that both $Z$ and $X\bs Z$ are connected and nonempty, and $X=\fT(Z)$. Take $z\in Z$.  For any $x\in X\bs Z$, pick a path $\gamma$ such that $\gamma(0)=z,\gamma(1)=x,\gamma((0,1])\subset X\bs Z$, and define $G^\gamma(Z)_{x}\subset E_{x}$
to be the space of elements which are the restriction of a section of $\tilde{\gamma}^*j_*E(U)$, where $\tilde{\gamma} \colon U \to X$ is any open immersion from a contractible topological space $U$, containing $[0,1]$, extending $\gamma$, i.e., $\tilde{\gamma}|_{[0,1]}=\gamma$.  

We claim this is independent of $\gamma$. To see this, suppose $\gamma'$ is another such path. Set $\gamma'(0)=z' \in Z$.
Since $X$ strongly deformation retracts onto $Z$, there is an arc $\gamma_0$ in $Z$, with endpoints $z'$ and $z$, such that $\gamma$ and $\gamma'' \coloneqq  \gamma' \gamma_0$ are homotopic (as paths from $z$ to $x$). Up to push $\gamma_{0}((0,1])$ inside $X \setminus Z$,  we can replace $\gamma'$ with $\gamma''$. Indeed, parallel transport along $\gamma_0$ identifies $G^{\gamma'}(Z)_x = G^{\gamma''}(Z)_x$. Now,
we may assume there is a homotopy $H$ between $\gamma$ and $\gamma''$ such that $H\big((0,1]\times [0,1]\big)\subset X\bs Z$. Then by continuing along $H$ we see that $G^\gamma(Z)_x=G^{\gamma''}(Z)_x$. We thus obtained a well-defined local system $G(Z)\subset E$, and it is clear that $G(Z)$ extends to a subsheaf of $j_*E$ as required.
\end{proof}

If $(X,D)$ is a log smooth algebraic space and $Z\subset X$ a connected component of an intersection of irreducible components of $D$, we may take $\fT_X(Z)$ to have the property that the retraction preserves $D$ and that the resulting map $\fT_X(Z)\setminus D\to Z$ strongly deformation retracts onto a torus fibration.


\subsection{Griffiths and Hodge bundles}\label{subsect:G/H bundles}
Let $(X,D)$ be a proper log smooth algebraic space.  For a complex local system $V_\bC$ on $(X\bs D)^\an$ we denote by $V_\cO$ the algebraic flat vector bundle on $X\bs D$ (with regular singularities) corresponding to $V_\bC$ via the Riemann--Hilbert correspondence, so $(V_\cO,\nabla)^\an\cong (V_{\cO^\an},\nabla)\coloneqq  (\cO_{(X\bs D)^\an}\otimes V_\bC,d\otimes 1)$.  We denote by $(\cV,\nabla)$ the lower canonical Deligne extension of $(V_\cO,\nabla)$, that is, the unique logarithmic flat vector bundle on $(X,D)$ whose analytification extends the flat vector bundle $(V_{\cO^\an},\nabla)$ whose residues have eigenvalues with real part contained in $[0,1)$. If $j:X\bs D\to X$ is the inclusion, we have a natural inclusion
$j^\an_*V_\bC\hookrightarrow \cV^\an$.

\begin{defn}
    Let $(X,D)$ be a proper log smooth algebraic space.  Let $(V_\bC,F^\bullet V_\cO)$ be a graded polarizable admissible mixed complex variation of Hodge structures on $X\bs D$ and $(\cV, F^\bullet \cV)$ its Deligne/Schmid extension to $X$.  We say $(V_\bC,F^\bullet V_\cO)$ is a CY variation if the deepest nonzero part of the Hodge filtration $F^{\pmax} V_\cO\cong\gr^\pmax_F V_\cO$ has rank one, in which case we call $L=F^\pmax\cV\cong\gr_F^\pmax\cV$ the Hodge bundle.  

\end{defn}

We will need a version of the transcendental part of a CY variation.
\begin{defn}\label{defn:minimal}Let $(X,D)$ be a proper log smooth algebraic space and $V=(V_\bQ,W_\bullet V_\bQ,F^\bullet V_\cO)$ an admissible graded polarizable rational mixed variation of Hodge structures on $X\bs D$.  Let $F^mV_\cO\neq 0$ be the smallest nonzero piece of the Hodge filtration and assume $\rk F^m V_\cO=1$.  
\begin{enumerate}
\item There is a unique minimal quotient $V\to V^\mathrm{min}$ in the category of rational mixed variations for which $\gr^m_F V_\cO^\mathrm{min}\neq 0$, which factors through all other such quotients.  We refer to $V^\mathrm{min}$ as the CY-minimal quotient.
\item There is a unique minimal subvariation $U\subset V$ with $\gr_F^m U_\cO\neq0$ which we call the CY-minimal subobject.  We don't introduce notation for the CY-minimal subobject.
\item The CY-minimal subobject of the CY-minimal quotient of $V$ is canonically identified with the CY-minimal quotient of the CY-minimal subobject of $V$; we call it the transcendental part $V^\trans$.  It is naturally identified with the smallest Hodge substructure $V^\trans\subset \gr_k^W V$ for which $\gr_F^m V^\trans\neq 0$, where $k$ is the unique weight with $\gr_F^m\gr^W_k V\neq 0$.
\end{enumerate}

\end{defn}
Note that if $V$ is a graded polarizable mixed Hodge structure with deepest nonzero Hodge filtration piece $F^mV_\bC$ of rank 1, there is a unique $k$ such that $\gr_F^m\gr^W_kV_\bC\neq 0$.  Then $V\to V^\mini$ factors through the quotient $V/W_{k-1}V$, $U\subset W_kV$, and $V^\trans$ is both the lowest weight subobject of $V^\mathrm{min}$ and the highest weight quotient of $U$.  Moreover, $V^\trans$ is simple, since $W_kV^\mini$ is a pure polarizable Hodge structure.

For a polarizable rational pure variation $V$ which is not necessarily CY with deepest nontrivial Hodge filtration piece $F^mV_\cO$, we will sometimes refer to $V^\trans\subset V$ as the smallest rational subvariation for which $F^mV^\trans=F^m V$.

\begin{lem}\label{lem:inject autos}Let $V$ be a graded-polarizable rational mixed Hodge structure such that the smallest nonzero piece $F^mV_\bC$ of the Hodge filtration has $\rk F^m V_\bC=1$.  Let $V\to V^\mini$ be the CY-minimal quotient and $V^\trans\subset V^\mini$ the transcendental part.  Then
\[\Aut(V^\mini)\hookrightarrow \Aut(V^\trans) \hookrightarrow \Aut(\gr_F^m V_\bC)=\Aut(\gr_F^m V_\bC^\mini)=\Aut(\gr_F^m V_\bC^\trans)\]
via the restriction maps, where the first two groups are automorphisms in the category of rational mixed Hodge structures and the last three are automorphisms in the category of vector spaces.
\end{lem}
\begin{proof}For any $f$ in the kernel of either of the above two restriction maps, $\ker(f-\id)$ is a Hodge substructure with $\gr_F^m\ker(f-\id)=0$, which must therefore be trivial.
\end{proof}

In the general context of variations of Hodge structures, we will pass to a certain CY variation whose Hodge bundle detects variation in any filtration piece of the original variation.
\begin{defn}\label{defn:griffiths wedge}Let $(X,D)$ be a proper log smooth algebraic space.  Let $_\orig V=(_\orig V_\bZ,F^\bullet _\orig V_\cO)$ be a polarizable pure integral variation of Hodge structures on $X\bs D$.  We define
\[V=\bigotimes_p\Exterior^{\rk F^p_\orig V_\cO}{_\orig V}.\]
It is a polarizable integral pure CY variation of Hodge structures which has unipotent local monodromy if $_\orig V$ does.  We refer to the Hodge bundle $L=\bigotimes_p\det F^p_\orig \cV$ of $V$ as the Griffiths bundle of $_\orig V$.

\end{defn}


\subsection{Boundary variations}\label{subsect:boundary data}
Let $(X,D)$ be a proper log smooth algebraic space.  After a modification, we may assume the irreducible components of $D$ are smooth, and that the intersection of any number of irreducible components of $D$ is connected (though possibly empty). We call such a space a proper \textit{strictly} log smooth algebraic space.

Note that the set of irreducible components of $D$ is naturally identified with $\pi_0(D^\reg)$ by taking closure. Moreover, by the above assumption, any component of the natural locally closed stratification of $X$ induced by $D$ can be uniquely characterized by which irreducible components of $D$ they are contained in, or, equivalently, the dual complex of $D$ is simplicial.  Thus, we make the following definition:  for any subset $\Sigma\subset \pi_0(D^\reg)$, we take $X_\Sigma \coloneqq  \bigcap_{E\in \Sigma} \bar E\bs \bigcup_{E\notin \Sigma} \bar E$. 
\vskip1em
Let $V=(V_\bZ,F^\bullet V_\cO)$ be a polarizable integral pure variation of Hodge structures of weight $w$ on $X_\varnothing=X\bs D$.  We collect here the various variations one obtains on the strata by degenerating; see, for example, \cite{peterssteenbrink} for background.

\subsubsection{For each stratum $X_\Sigma $ there is:}
\begin{itemize}
    \item A DR-neighborhood $X_\Sigma \subset   {\mathfrak{T}_X(X_\Sigma) } \eqqcolon \fT(\Sigma)$ in the sense of \Cref{sect:DR}.  Set $\fT^*(\Sigma) \coloneqq  {\mathfrak{T}(\Sigma) }\bs D$.
 
    \item For each $E\in \Sigma$, there is a globally
defined nilpotent operator  $N_{E}:V_\bQ|_{\fT^*(\Sigma)}\to V_\bQ|_{\fT^*(\Sigma)}$ given by the logarithm of the unipotent part of the local monodromy around $E$.  Indeed, there is a dense Zariski open set $U\subset X$ meeting $X_\Sigma$ on which every divisor in $\Sigma$ has a local defining equation.  If $\mathfrak{T}(U\cap X_\Sigma)$ is a bundle over $U\cap X_\Sigma$ with fiber $(\Delta^*)^{\Sigma}$, then we have a commutative diagram with exact rows
\begin{equation}\label{diagram:fibration}\begin{tikzcd}
1\ar[r]&\pi_1((\Delta^*)^{\Sigma})\ar[d, equals]\ar[r]&\pi_1(U\cap \fT^*(\Sigma))\ar[d,two heads]\ar[r]&\pi_1(U\cap \fT(\Sigma))\ar[d,two heads]\ar[r]&1\\
&\pi_1((\Delta^*)^{\Sigma})\ar[r]&\pi_1(\fT^*(\Sigma))\ar[r]&\pi_1(\fT(\Sigma))\ar[r]&1
\end{tikzcd}\end{equation}
where the top row is split by the defining equations.  Hence, the meridian winding around $E$ is central in $\pi_1(\fT^*(\Sigma))$, and so $N_{E}$ intertwines the monodromy representation of $V_\bQ|_{\fT^*(\Sigma)}$. 
   \item Associated to the local monodromy logarithms $\{N_E\}_{E\in\Sigma}$ there is a weight filtration $W(\Sigma)_\bullet V_\bQ|_{\fT^*(\Sigma)}$ on $V_\bQ|_{\fT^*(\Sigma)}$ for which each $N_E$ is degree $-2$.  Each one has a natural saturated integral structure which we denote $W(\Sigma)_\bullet V_\bZ|_{\fT^*(\Sigma)}$.

    \item The Hodge filtration $F^\bullet V_\cO$ extends to a locally split filtration $ F^\bullet \cV $ of the Deligne extension $\cV$.
    \item The restriction of the Deligne extension $ \cV|_{\mathfrak{T}(\Sigma)}$ is naturally filtered by logarithmic flat vector sub-bundles $W(\Sigma)_\bullet  \cV|_{\mathfrak{T}(\Sigma)}$ which are the Deligne extensions of the local systems $W(\Sigma)_\bullet V_\bQ|_{\fT^*(\Sigma)}$.  There are natural flat morphisms $N_E:\cV|_{\mathfrak{T}(\Sigma)}\to\cV|_{\mathfrak{T}(\Sigma)}$ with degree $-2$ with respect to $W(\Sigma)_\bullet  \cV|_{\mathfrak{T}(\Sigma)}$.  
    
\end{itemize}
\subsubsection{If we further suppose $V_\bZ$ has unipotent local monodromy, then:}
\begin{itemize}
\item Each $\gr^{W(\Sigma)}_k V_\bZ|_{\fT^*(\Sigma)}$ extends as a local system to $\mathfrak{T}(\Sigma)$. We somewhat abusively denote the extension by $\gr^{W(\Sigma)}_k V_\bZ$.
\item Define 
\[V(\Sigma)_\bQ|_{\fT^*(\Sigma)}\coloneqq  \coker\left(\oplus_{E\in \Sigma}N_E:\bigoplus_{E\in \Sigma}V_\bQ|_{\fT^*(\Sigma)}\to V_\bQ|_{\fT^*(\Sigma)}\right)\]
which has a natural integral structure $V(\Sigma)_\bZ|_{\fT^*(\Sigma)}$ coming from the image of $V_\bZ|_{\fT^*(\Sigma)}$, and a natural filtration $W_\bullet V(\Sigma)_\bQ|_{\fT^*(\Sigma)}$ coming from the image of $W(\Sigma)_\bullet V_\bQ|_{\fT^*(\Sigma)}$.  Then $V(\Sigma)_\bZ|_{\fT^*(\Sigma)}$ (together with its filtration) extends to $V(\Sigma)_\bZ$ on $\mathfrak{T}(\Sigma)$.

Observe that $\gr^W_k V(\Sigma)_\bQ$ is naturally identified with the primitive
part of $\gr^{W(\Sigma)}_k V_\bQ$ with respect to the (simultaneous) hard Lefschetz decomposition with respect to the maps $N_E:\gr^{W(\Sigma)}_{k}V_\bQ\to \gr^{W(\Sigma)}_{k-2}V(\Sigma)_\bQ$ for $E\in\Sigma$.

\item The graded pieces $\gr^{W(\Sigma)}_k \cV|_{\mathfrak{T}(\Sigma)}$ of the Deligne extension have connection with no residue and are identified with the Deligne extensions of $\gr^{W(\Sigma)}_kV_\bZ$.  The same is true for 
\[ V(\Sigma)_\cO\coloneqq  \coker\left(\oplus_{E\in \Sigma}N_E:\bigoplus_{E\in \Sigma} \cV|_{\mathfrak{T}(\Sigma)}\to \cV|_{\mathfrak{T}(\Sigma)}\right).\]

    \item  $V_\Sigma\coloneqq  (V(\Sigma)_\bZ|_{X_\Sigma}, W_\bullet V(\Sigma)_\bQ|_{X_\Sigma},F^\bullet V(\Sigma)_\cO|_{X_\Sigma})$ is an admissible graded polarizable integral variation of mixed Hodge structures.  Recall that a polarization on $\gr^W_{w+k}V_\Sigma$ for $k\geq 0$ is given by $q(-,N^k-)$ for any $N$ in the cone $\sum_{E\in\Sigma} \bR_{>0}N_E $, where $q$ is a polarization on $V$.  Importantly, the polarization depends on the local monodromy. 
    \end{itemize}

\subsubsection{    If finally $V$ is a CY variation with unipotent local monodromy and Hodge bundle $F^\pmax V_\cO$, then:}\label{subsub:CY with unip}
  
    \begin{itemize}
    \item  For each $\Sigma$ there is a unique integer $k_\Sigma$ such that $\gr^\pmax_F\gr^W_{k_\Sigma}V_\Sigma\neq 0$.  Then $\gr^W_{k_\Sigma}V_\Sigma$ is a CY variation with Hodge bundle $F^m\gr^W_{ k_\Sigma}V_\Sigma$.
    \item We denote 
    by $V_\Sigma\to V_\Sigma^\mathrm{min}$ the CY-minimal quotient of $V_\Sigma$ and by $ V^\trans_\Sigma$ the transcendental part of $V_\Sigma$ in the sense of \Cref{defn:minimal}.  Note that $V_\Sigma^\trans$ is also the transcendental part of $\gr^W_{k_\Sigma}V_\Sigma$.  Note also that even on the stratum $\Sigma=\varnothing$ where $X_\varnothing=X\bs D$ and $V_\varnothing=V$, the transcendental part $V^\trans_\varnothing$ and the CY-minimal quotient $V^\mathrm{min}_\varnothing$ (which is equal to $V^\trans_\varnothing$) may be strictly smaller than $V$. 

    \item For a single point $x\in X$, there is a unique $X_\Sigma$ containing $x$ and we define $V^{\min}(x)$ (resp. $V^{\trans}(x)$) as the CY-minimal quotient (resp. the transcendental part) of the mixed Hodge structure $V_{\Sigma,x}$.  Note that at a very general point $x\in X_\Sigma(\bC)$ we simply have $V^\trans(x)=V^\trans_{\Sigma,x}$ and $V^\mathrm{min}(x)=V^\mathrm{min}_{\Sigma,x}$.

    \item The restriction of the Schmid extension of the Hodge bundle $F^\pmax V_\cO$ to $\overline{ X_\Sigma}$ is naturally identified with the Schmid extension of the Hodge bundle of each of the following:
    \[\gr^{W(\Sigma)}_{k_\Sigma}\cV|_{X_\Sigma}\]
    \[V_\Sigma\]
     \[\gr^W_{\geq k_\Sigma}V_\Sigma\coloneqq  V_\Sigma/W_{k_\Sigma-1}V_\Sigma\]
   \[\gr^W_{k_\Sigma}V_\Sigma\]
   \[V^\trans_\Sigma.\]
   Indeed, for the first two this is because the above constructions are compatible on the level of filtered logarithmic flat vector bundles, and for the last four this is standard.  

\item After choosing a very general basepoint $x_\Sigma\in X_\Sigma$, we thereby obtain a period map
\[\phi^\trans_\Sigma:\widetilde {X_\Sigma}^{\cV^\trans_{\Sigma,\bQ}}\to \check\bD(V^\trans_{\Sigma,\bC,x_\Sigma})^\an \]
associated to $V^\trans_\Sigma$, where $\check\bD(V^\trans_{\Sigma,\bC,x_\Sigma})$ is a flag variety of filtrations on $V^\trans_{\Sigma,\bC,x_\Sigma}$.

\item When $V$ arises from a general variation $_\orig V$ as in \Cref{defn:griffiths wedge}, and assuming the local monodromy of $_\orig V$ is unipotent, we further define $_\orig V_\Sigma^{\gr}$ to be the graded polarizable integral mixed variation with underlying local system $_\orig V_{\Sigma,\bZ}^{\gr}=\bigoplus_k\gr_k^{W(\Sigma)}{_\orig V_\bZ}|_{X_\Sigma}$.  For $x\in X_\Sigma$ we take ${_\orig V}^{\gr}(x)={_\orig V}_{\Sigma,x}^{\gr}$.

\item We let $V^{\vee}$ denote the dual variation to $V$. We may canonically identify

\[V(\Sigma)^\vee_\bQ|_{\fT^*(\Sigma)}\coloneqq  \ker\left(\oplus_{E\in \Sigma}N^\vee_E:V^{\vee}_\bQ|_{\fT^*(\Sigma)}\to \bigoplus_{E\in \Sigma}V^{\vee}_\bQ|_{\fT^*(\Sigma)}\right)\subset j_*(V_\bQ^\vee)|_{\fT^*(\Sigma)}\]
and thus $V(\Sigma)_\bQ^\vee\subset j_*(V_\bQ^\vee)|_{\mathfrak{T}(\Sigma)}$.

 \item Moreover, we have a subvariation $V^{\min,\vee}_\Sigma\coloneqq  (V^{\min}_\Sigma)^{\vee}\subset V^{\vee}_\Sigma $, which is the smallest subvariation which pairs non-trivially with $\gr^m_F V_\Sigma$, and likewise $V^{\trans,\vee}_\Sigma$ is the highest graded piece of $V_\Sigma^{\mathrm{min,\vee}}$.  Note that the Hodge bundle of $V_\Sigma^{\trans,\vee}$ is $F^{m-k_\Sigma}V_\Sigma^{\trans,\vee}$.

 \end{itemize}

\begin{lemma}\label{lem:compat}

\begin{enumerate}

\item  The local system $V^{\min,\vee}_{\Sigma,\bQ}$ is the restriction of a local system $V^{\min,\vee}(\Sigma)_\bQ$ on $\mathfrak{T}(\Sigma)$, which is naturally a subsheaf of $\cV^\vee\mid_{\mathfrak{T}(\Sigma)}$ consisting of flat sections. The fibers of $V^{\min,\vee}(\Sigma)_{\bQ,x}$ at a point $x\in \mathfrak{T}^*(\Sigma)$ consist of all elements $s_x\in \cV^{\vee}_x$ which extend to a flat section $s$ along any contractible subset $U\subset \mathfrak{T}(\Sigma)$ whose restriction to $X_{\Sigma}$ lands in $V^{\min,\vee}_{\Sigma,\bQ}.$ Moreover, it is sufficient to check this condition for a single contractible $U$ which non-trivially intersects $X_\Sigma$.

\item If $\Sigma \subset \Sigma'$, we have the containments
$V^{\min,\vee}(\Sigma')_\bQ\subset V^{\min,\vee}(\Sigma)_\bQ$ within $\mathfrak{T}(\Sigma)\cap \mathfrak{T}(\Sigma')$.

\end{enumerate}

\end{lemma}

\begin{proof}

For (1), first note that $V^{\min,\vee}_{\Sigma,\bQ}$ extends as a subsheaf of $j_*V^\vee_\bQ\mid_{\fT(\Sigma)}$ by \Cref{lem:subsgood}, and hence as a subsheaf of $(\cV^\vee)^{\an}\mid_{\fT(\Sigma)}$ since $j_*V^\vee_\bQ\subset (\cV^\vee)^{\an}$. Moreover, the image of any section of $j_*V^\vee_\bQ$ is flat. Finally, it is clear by our construction that any element of $V^{\min,\vee}(\Sigma)_{\bQ,x}$ extends along any contractible set $U$ to a section, and the fiber along $X_{\Sigma}\cap U$ will land in $V_{\Sigma,\bQ}^{\min,\vee}$. To see that it is sufficient to consider a single contractible open set $U$ which intersects $X_{\Sigma}$, note that any such $U$ gives a canonical identification of stalks which preserves the restrictions of global flat sections.

For (2), first note that the Deligne extension of $V^{\min,\vee}_{\Sigma,\bQ}$ along $X_{\Sigma'}$ is naturally a flat subbundle of $\cV^{\vee}|_{\bar{ X}_\Sigma}$ whose restriction to $X_{\Sigma'}$ underlies a rational subvariation of Hodge structures with $\gr_F^{-m}\neq 0$  hence contains $V^{\min,\vee}_{\Sigma',\bQ}$. It is clear that $V(\Sigma')^\vee_\bQ\subset V(\Sigma)^\vee_\bQ$ on $\mathfrak{T}(\Sigma)\cap \mathfrak{T}(\Sigma')$. Now take a contractible $U \subset \fT(\Sigma')$ which intersects $X_{\Sigma'}$ (and therefore also $X_{\Sigma}$).  Let $x\in \fT(\Sigma)\cap U\bs D$ be a point whose path-component in $\fT(\Sigma)\cap U$ meets $X_{\Sigma}$.  By (1) any element of $V^{\min,\vee}(\Sigma')_{\bQ,x}$ extends to a flat global section $s$ over $U$. By the above, the restriction $s\mid_{X_{\Sigma}}$ is contained in $V^{\min,\vee}_{\Sigma,\bQ}$, so by (1) again the conclusion follows.
\end{proof}


\subsection{The CY-minimal quotient}\label{subsect:map to Pn}

Let $(X,D)$ be a proper strictly log smooth algebraic space and $V=(V_\bZ,F^\bullet V_\cO)$ a polarizable integral pure CY variation on $X\bs D$ with unipotent local monodromy.  By \cite{BBT23}, for any $\Sigma\subset\pi_0(D^\reg)$
we may consider the factorization of the period map of $V^\trans_\Sigma$
\[\begin{tikzcd}
X_\Sigma^\define\ar[r,"f_\Sigma^\define"]&Y_\Sigma^\define\ar[r,"\psi_\Sigma"]&\Gamma_\Sigma\backslash \bD_\Sigma
\end{tikzcd}\]
where $f_\Sigma$ is dominant with geometrically connected general fiber, $Y_\Sigma$ is normal, and $\psi_\Sigma$ is finite. Note that $f_\Sigma$ does not depend on $\Gamma_\Sigma$ as it is the Stein factorization of any relative compactification of any period map associated to $V^\trans_\Sigma$.  If the monodromy of $V^\trans_\Sigma$ is neat, then $V^\trans_\Sigma$ is pulled back from a variation on $Y_\Sigma$, say $V^\trans_\Sigma=f_\Sigma^*U$.  Recall that $V^\trans_\Sigma$ is the lowest weight piece of the CY-minimal quotient $V_\Sigma\to V_\Sigma^\mathrm{min}$, that is, $V_\Sigma^\trans=W_{k_\Sigma}V^\mathrm{min}_\Sigma$.

\begin{lem}\label{lem:tr split}  Assume the monodromy of $V_\bZ$ is neat.  There is a dense open subset $Y_\Sigma^\circ\subset Y_\Sigma$ such that, setting $X_\Sigma^\circ=f_\Sigma^{-1}(Y_\Sigma^\circ)$, we have that $V_\Sigma^\mathrm{min}|_{X_\Sigma^\circ}$ is pulled back from a local system on $Y_\Sigma^\circ$. 
\end{lem}

\begin{proof} Let $Y^\circ_\Sigma\subset Y_\Sigma$ be a dense open set for which the fibers of $f_\Sigma$ are smooth and $f_\Sigma$ is a topological fibration.  As the fibers are connected, it suffices to show $V^{\mini,\vee}_\Sigma$ is trivial on a very general fiber $Z$ of $f_\Sigma$ over $Y_\Sigma^\circ$.  The connection yields a morphism $\gr_F^{m-k_\Sigma} V^{\mini,\vee}_\Sigma|_Z\to \gr_F^{m-k_\Sigma-1}V^{\mini,\vee}_\Sigma|_Z\otimes \Omega_{Z}$
and we have a commutative diagram
\[\begin{tikzcd}
\gr_F^{m-k_\Sigma} W_{-k_\Sigma-1}V^{\mini,\vee}|_Z\ar[d]\ar[r]&\gr_F^{m-k_\Sigma} V^{\mini,\vee}_\Sigma|_Z\ar[r,"\cong"]\ar[d]&\gr_F^{m-k_\Sigma} V^{\trans,\vee}_\Sigma|_Z\ar[d,"0"] \\
0=\gr_F^{m-k_\Sigma-1}W_{-k_\Sigma-1}V^{\mini,\vee}_\Sigma|_Z\otimes \Omega_{Z} \ar[r]&\gr_F^{m-k_\Sigma-1} V^{\mini,\vee}|_Z\otimes \Omega_{Z} \ar[r]& \gr_F^{m-k_\Sigma-1}V^{\trans,\vee}_\Sigma|_Z\otimes \Omega_{Z}
\end{tikzcd}\]
where the bottom left vanishing follows from the fact that $F^m \gr_{>k_\Sigma}^WV_\Sigma^\mini=0$, so $V_\Sigma^\mini$ has no Hodge weight $(p,q)$ piece for $p\geq m$ and $p+q>k_\Sigma$, hence $V_\Sigma^{\mini,\vee}$ has no Hodge weight $(p,q)$ piece for $q\leq -m$ and $p+q<-k_\Sigma$, and in particular for $p=m-k_\Sigma-1$.  Thus, the middle vertical morphism vanishes, and $F^{m-k_\Sigma} V_\Sigma^{\mini,\vee}$ is flat on $Z$.
It is therefore isomorphic as a flat bundle to $\gr_F^{m-k_\Sigma}V_\Sigma^{\trans,\vee}$, hence has trivial monodromy.
By the theorem of the fixed part \cite{andredeligne}, the fixed part $H^0(Z,V_\Sigma^{\mini,\vee})$ comes from a sub-variation of mixed Hodge structures which contains the Hodge bundle, hence $(V_\Sigma|_{Z})^{\mini,\vee}$ has trivial monodromy.  Since $Z$ is a very general fiber, we have $(V_\Sigma|_{Z})^{\mini,\vee}=V_\Sigma^{\mini,\vee}|_Z$, and this proves the claim.
\end{proof}

\begin{cor}\label{cor:strata lift}Assume the monodromy of $V_\bZ$ is neat.  For each stratum $\Sigma$ we have a commutative diagram
\begin{equation}
\begin{tikzcd}
\widetilde{\fT^\circ(\Sigma)}^{V^{\mini}_{\Sigma,\bQ}|_{\fT^\circ(\Sigma)}}\ar[rrr,"\rho(\Sigma)"]&&&\bP(V^\mathrm{min}_{\Sigma,\bC,x_\Sigma})^\an\\
\widetilde{X^\circ_\Sigma}^{V^{\mini}_{\Sigma,\bQ}|_{X^\circ_\Sigma}}\ar[u]\ar[r]&\widetilde{X^\circ_\Sigma}^{V^\trans_{\Sigma,\bQ}|_{X^\circ_\Sigma}}\ar[r,"\phi^{\trans}_\Sigma"]&\check\bD(V^\trans_{\Sigma,\bC,x_\Sigma})^\an\ar[r]&\bP(V^\trans_{\Sigma,\bC,x_\Sigma})^\an\ar[u]
\end{tikzcd}\label{fixed part}
\end{equation}
where $\fT^\circ(\Sigma)$ is a DR-neighborhood of $X^\circ_\Sigma$ as in \Cref{lem:tr split}.  Here, the bottom right horizontal map is the forgetful map which only remembers the Hodge line, the right vertical map is obtained from the inclusion $V_\Sigma^\trans\hookrightarrow V_\Sigma^\mathrm{min}$, and the top map is obtained by taking the image of the Hodge bundle $F^m\cV$ under the natural morphism of logarithmic flat bundles
\begin{equation}\label{eqn quotient}\cV|_{\fT^\circ(\Sigma)}\to V^\mathrm{min}(\Sigma)_\cO|_{\fT^\circ(\Sigma)} \end{equation}
where $V^\mathrm{min}(\Sigma)_\cO|_{\fT^\circ(\Sigma)}$ is the flat bundle associated to the extension of the local system $V^\mathrm{min}_{\Sigma,\bQ}|_{X^\circ_\Sigma}$ to $\fT^\circ(\Sigma)$.
 
\end{cor}
\begin{proof} The morphism \eqref{eqn quotient} arises from the dual of the inclusion of the extension guaranteed by \Cref{lem:compat}. The above composition is full rank on $F^m\cV$ in restriction to $X_\Sigma$ and this is an open condition.
\end{proof}

\begin{rem}\label{rem:GGR flat} The analysis of the differential in \Cref{lem:tr split} (in the case of the Griffiths bundle) is related to the ``infinitesimal period relation'' of \cite[\S2.6]{GGRtorsion}. The flat connection on the Griffiths bundle in tubular neighborhoods of numerically Hodge-trivial subvarieties constructed therein is obtained by the trivializing sections pulled back via the map $\rho(\Sigma)$ of \Cref{cor:strata lift}.  
\end{rem}

\subsection{Integrability}\label{subsect:integrability}
In the following for a stratum $\Sigma\subset \pi_0(D^\reg)$ we define $D_\Sigma\coloneqq  \bigcup_{E\notin\Sigma}\bar E|_{\overline{X_\Sigma}}$ to be the natural log smooth divisor of $\overline{X_\Sigma}$.
\begin{lem}\label{lem:everything is compatible}Let $(X,D)$ be a proper log smooth algebraic space and $V=(V_\bC,F^\bullet V_\cO)$ a CY polarizable complex variation of Hodge structures on $X\bs D$ with unipotent local monodromy.  Let $L=F^\pmax \cV$ be the Hodge bundle, that is, the Schmid extension of $F^\pmax V_\cO$.

\begin{enumerate}
\item $L$ is nef.
\item The Schmid extension of any power $(F^mV_\cO)^k$ is naturally identified with the same power $L^k$ of the Hodge bundle.
\item For any proper log smooth $(Y,D_Y)$ and morphism $g\colon (Y,D_Y)\to (\overline{X_\Sigma},D_\Sigma)$, the pullback of the Hodge bundle $g^*L$ is naturally identified with the Hodge bundle of the pullback $g|_{Y\bs D_Y}^*V^\trans_\Sigma$. 
\end{enumerate}
\end{lem}
\begin{proof}  Part (1) is as in \cite[Lemma 6.15]{BBT23}.  Parts (2) and (3) follow from the fact that the Deligne extension is functorial with respect to pullbacks and tensor operations if the local monodromy is unipotent. 
\end{proof}

\begin{defn}\label{defn:integrable}  Let $(X,D)$ be a proper log smooth algebraic space and $V=(V_\bC,F^\bullet V_\cO)$ a polarizable integral CY variation of Hodge structures on $X\bs D$ with unipotent local monodromy with Deligne/Schmid extension $(\cV,F^\bullet\cV)$.  We say the Hodge bundle $L=F^\pmax\cV$ is \emph{integrable} if, after replacing $(X,D)$ with a strictly log smooth modification, for any irreducible proper strictly log smooth $(Y,D_Y)$ with a morphism $g\colon (Y,D_Y)\to (\overline{X_\Sigma},D_\Sigma)$ which is generically finite onto its image, $g^*L$ is big whenever either of the following equivalent conditions is satisfied:
\begin{enumerate}
\item The Griffiths bundle of the transcendental part of the pullback $(g|_{Y\bs D_Y}^*V^\trans_\Sigma)^\trans$ is big.

\item The period map of the transcendental part of the pullback $(g|_{Y\bs D_Y}^*V^\trans_\Sigma)^\trans$ on $Y\bs D_Y$ is generically immersive.

\end{enumerate}
\end{defn}
Concretely, this means that if $g^*L$ is not big, then some subvariation of the pullback $g|_{Y\bs D_Y}^*V^\trans_\Sigma$ containing the Hodge bundle is isotrivial on a curve through the generic
point of $Y$.  Note the following important property of the Hodge bundle:

\begin{lem}[{\cite[Cor. 4.2.8.(iii).(b)]{DeligneII}}]\label{lem:deligne}Let $V$ be a polarizable integral CY variation of Hodge structures on an algebraic variety $Y$ with Hodge bundle $L$.  Then the following are equivalent:
\begin{enumerate}
\item $L$ is flat;
\item $L$ is torsion;
\item $V^\trans$ is isotrivial.
\end{enumerate}
\end{lem}

\begin{rem}\label{rmk:torsion Hodge}Using \Cref{lem:deligne}, integrability of the Hodge bundle as in \Cref{defn:integrable} is equivalent to either of the following conditions, which explains the name:  
\begin{enumerate}
\item For any $\Sigma$, $x\in X_\Sigma$, and reduced analytic germ $T\subset X_\Sigma$ at $x$, if $L|_T$ is flat (as a subbundle of $(V_\cO,\nabla)$) then $L|_{T^\mathrm{Zar}}$ is flat, where $T^\mathrm{Zar}$ is the Zariski closure.
\item The leaves of the foliation given by the fibers of the map which only remembers the Hodge bundle
\[\widetilde{X_\Sigma}^{V_{\Sigma,\bQ}^\trans}\to \bP(V^\trans_{\Sigma,\bC,x_\Sigma})^\an\]
are algebraic.
\end{enumerate}
Indeed, these two conditions are obviously equivalent to each other.  Moreover, with the notation of \Cref{defn:integrable}, $\widetilde{(Y\setminus D_Y)}^{g^*V^\trans_\Sigma}\to\bP(V^\trans_{\Sigma,\bC,x_\Sigma})^\an$ is generically immersive if and only if $g^*L$ is big, by \cite[Lemma 6.17]{BBT23}.  

Now, to see the equivalence of the above conditions with \Cref{defn:integrable}, first suppose the Hodge bundle is integrable in the sense of the definition, and take $Y'$ to be a log smooth compactification of a resolution of the Zariski closure of a leaf $\cL$.  If $\cL$ is not algebraic, by taking complete intersections we may find $Y\subset Y'$ for which the very general leaf of (2) on $Y$ is 1-dimensional and not algebraic.  Then $g^*L$ is not big and it is not flat on any curve through a very general point, which is a contradiction.  Conversely, if the foliation is algebraically integrable, then for $Y$ as in the definition for which $g^*L$ is not big, $Y$ is covered by positive-dimensional subvarieties (namely the leaves) along which $(g|_{Y\backslash D_Y}^*V_\Sigma^\trans)^\trans$ is isotrivial.
\end{rem}

\begin{rem}  It is proven in \cite{BBT23} that the Griffiths bundle is semiample on $X\bs D$.  The same is proven for the Hodge bundle of a CY variation whenever the Kodaira--Spencer map on the Hodge bundle is immersive. The strategy of \Cref{sect:semiample} can be used to prove the Hodge bundle of a CY variation is semiample on $X\bs D$ if it is integrable (on $X\bs D$) in the above sense.   
\end{rem}

\begin{lem}\label{lemma:Griffiths is integrable}Let $(X,D)$ be a proper log smooth algebraic space and $(_\orig V_\bC,F^\bullet {_\orig V_\cO})$ a polarizable complex variation of Hodge structures on $X\bs D$ with unipotent local monodromy with Deligne/Schmid extension $(_\orig\cV,F^\bullet{_\orig\cV})$.  Set  \[V\coloneqq  \bigotimes_p  \Exterior^{\rk F^p{_\orig V_\cO}}{_\orig V}.\]
Then:
\begin{enumerate}
\item For any subvariety $Y\subset \overline{X_\Sigma}$ with $Y_\Sigma\coloneqq  Y\cap X_\Sigma\neq\varnothing$, the following are equivalent: 
\begin{enumerate}
\item The Hodge bundle of $V$ is big in restriction to $Y$.
\item The Griffiths bundle of $V$ is big in restriction to $Y$.
\end{enumerate}
\item The Hodge bundle of $V$ is integrable.
\end{enumerate}
\end{lem}

\begin{proof}For (1), $(a)\Rightarrow(b)$ since the Griffiths bundle is the sum of the Hodge bundle and a semipositive line bundle. For the converse, first note that the Hodge bundle of $V_\Sigma$ is the Griffiths bundle of the associated graded of the entire nearby cycles functor $\psi_\Sigma ({_\orig V})$ along $X_\Sigma$.  
In particular, $(_\orig V)_\Sigma$ is isotrivial in restriction to $Y$ if and only if $\psi_\Sigma ({_\orig V})$ is by hard Lefschetz; see \S \ref{subsub:CY with unip}.  It also follows from hard Lefschetz that the Hodge bundle of $V$ on $Y$ is the sum of the Griffiths bundle of $(_\orig V)_\Sigma$ and a semipositive bundle, and that a psotive multiple of the Griffiths bundle of $(_\orig V)_\Sigma$ is the sum of the Hodge bundle of $V$ and a semipositive bundle.  

Now, recall that the Griffiths bundle is ample on the image of a period map
by \cite{BBT23}. Thus, if the Hodge bundle of $V$ on $Y$ is not big, then $(_\orig V)_\Sigma$ is isotrivial on a curve through the general point of $Y$, as therefore is $V_\Sigma$, so the Griffiths bundle of $V$ is not big on $Y$. 
Statement (2) is an immediate consequence, since if the Hodge bundle of $V$ is not big in restriction to $Y$,  $(_\orig V)_\Sigma$ is again isotrivial on a curve through the generic point of $Y$, as therefore is a part of $V_\Sigma$ which contains the Hodge bundle.
\end{proof}


\subsection{Unpolarized Hodge structures}\label{subsect:unpol orbits}

Let $\bfM$ be the Mumford--Tate group of a polarizable integral pure Hodge structure $(V_\bZ,F^\bullet V_\bC)$, $\bD$ the $\bfM(\bR)$-orbit of $F^\bullet V_\bC$ in the flag variety of filtrations on the vector space $V_\bC$, and $\bfM(\bZ)\subset \bfM(\bQ)$ the subgroup stabilizing the lattice induced by $V_\bZ$ in a chosen faithful representation of $\bfM$.
The following generalizes a result of Narasimhan--Nori \cite{Na81} in the case of abelian varieties and Huybrechts \cite[Cor 1.8]{Hu18} in the case of hyperk\"ahler varieties.
    \begin{lem}\label{lem:unpol orbit}For any $x\in \bfM(\bZ)\backslash \bD$, there are finitely many points $x'\in \bfM(\bZ)\backslash \bD$ for which the associated integral Hodge structures $V(x),V(x')$ are isomorphic.
\end{lem}

Note that as $x$ varies, the size of the set of such $x'$ is not in general bounded; see, e.g., \cite[\S 1.5]{Hu18}.
\begin{proof} In \cite{Na81}, this claim is proven for abelian varieties, and the same proof works in this more general context. We reproduce the proof for the ease of the reader.

First, note that $\bfM(\bZ)\bs\bD$ has a finite map to a usual period space $\Aut(V_\bZ,Q)\bs\bD'$ corresponding to a choice of polarized lattice $(V_\bZ,Q)$, and so it is enough to work with the latter space.

Next, consider the algebra $B=\End(V(x))$ of unpolarized Hodge endomorphisms. Since polarizable Hodge structures are a semisimple category, it follows that $B_{\bQ}$ is a semisimple algebra over $\bQ$. Moreover, the polarization gives an involution $\theta$ of $B_\bQ$; let $ B^{\theta}\subset B_\bQ$ be the elements fixed by $\theta$. Letting $G$ be the algebraic group $B^{\times}$, we see that $\theta$ gives an involution $G\to G^{\op}$. We shall prove that $V(x)$ admits only finitely many orbits of polarizations of discriminant $\mathrm{disc}(Q)$ for the action of $G(\bZ)$, which will prove the lemma.

Let $P$ be the set of polarizations of $V(x)$. There is a natural injection $\iota:P\to B_\bQ$
given by $\iota(Q')\coloneqq  \phi^{-1}_Q\circ \phi_{Q'}$ where $\phi_{Q'}:V(x)\to V(x)^{\vee}$. Let $F\subset B^{\theta}$ be the minimal lattice containing $\iota(P)$. There is a natural action of $G$ on $B^{\theta}$ given by $\pi(g)s\coloneqq  \theta(g^{-1})\circ s\circ g^{-1}$ for which $\iota$ is equivariant, and for which $G(\bZ)$ preserves $F$. Finally, let $F_1\coloneqq  \{f\in B^{\theta}_{\bC}\mid \det(f)=1\}$. As in \cite[Lemma 3.1]{Na81}\footnote{This lemma uses only that $(B_\bQ,\theta)$ is an involutive semisimple algebra.} the orbits of $G_\bC$ on $F_1$ are finite in number and closed. The result now follows by \cite[Thm 6.9]{BHC62}.
\end{proof}


\subsection{Combinatorial monodromy}\label{subsect:combin monodromy}

For any proper algebraic space $X$, a torsion line bundle $L$ has a canonical flat connection:  if $L^N\cong \cO_X$, then the local flat sections are those sections $s$ for which $s^N$ extends to a global section.  Equivalently, there is a finite \'etale cover $\pi:X'\to X$ for which $\pi^*L\cong \cO_{X'}$, and the flat connection on $L$ is inherited from the trivial connection on $X'$ whose flat sections are global sections.

If $X$ is now a nodal curve with normalization $\nu_X:X'\to X$ and $L$ a line bundle for which $\nu_X^*L$ is torsion, $L$ has a canonical flat connection whose local flat sections are sections $s$ whose restriction to $X^\reg$ are flat with respect to the above connection on $L|_{X^\reg}$, or equivalently those for which $\nu_X^*s$ is flat.

\begin{defn}\label{defn:torsion combo}
    Let $X$ be a proper algebraic space with a line bundle $L$.  We say $L$ has torsion combinatorial monodromy if for every proper nodal curve $C$ and morphism $g:C\to X$ for which $(g\circ \nu_C)^*L$ is torsion, the canonical flat connection on $g^*L$ has torsion monodromy.  Note that if $L$ is a Hodge bundle of a polarizable integral CY variation of Hodge structure, $(g\circ\nu_C)^*L$ is torsion if and only if $\deg g^*L=0$ by \Cref{lem:deligne}.
\end{defn}


\subsubsection{Torsion combinatorial monodromy of the Griffiths bundle}\label{subsect:Griffths torsion}
Thanks to a result of \cite{GGRtorsion}, the Griffiths bundle always has torsion combinatorial monodromy.

\begin{thm}[{Green--Griffiths--Robles \cite[Theorem 5.22]{GGRtorsion}}]\label{thm:GGR}Let $(X,D)$ be a proper log smooth algebraic space and $_\orig V=(_\orig V_\bZ,F^\bullet _\orig V_\cO)$ a polarizable integral pure variation of Hodge structures on $X\bs D$ with unipotent local monodromy.  Then the Griffiths bundle has torsion combinatorial monodromy.
\end{thm}
We give a slight generalization below.  We say the Hodge bundle of a CY variation has norm one combinatorial monodromy if the monodromy of the canonical connection on any Hodge degree 0 nodal curve acts by a character of complex norm one.
\begin{lem}\label{lem:GGR}Let $(X,D)$ be a proper log smooth algebraic space and $(V_\bZ,F^\bullet V_\cO)$ a polarizable integral pure CY variation of Hodge structures on $X\bs D$ with unipotent local monodromy.  If the Hodge bundle has norm one combinatorial monodromy, then it has torsion combinatorial monodromy.
\end{lem}
\begin{proof} For any connected nodal curve $g:C\to X$ with Hodge degree 0, the pointwise transcendental parts $V^\trans(g(c))$ form an isotrivial simple not-necessarily-polarizable integral pure variation of Hodge structures $V^\trans(C)$ which is polarizable on each irreducible component.  It suffices to show $V^\trans(C)$ is polarizable, which is a consequence of the following: 
\begin{claim}\label{claim:torsion}Let $U=(U_\bZ,F^\bullet U_\bC)$ be a simple polarizable integral pure CY Hodge structure with deepest Hodge filtration piece $F^\pmax U_\bC$ and $\gamma$ a Hodge automorphism of $U_\bZ$ which acts with norm one eigenvalues on $F^\pmax U_\bC$.  Then $\gamma$ preserves any polarization $q$ on $U$; in particular, it is torsion.
\end{claim}
\begin{proof} Following the proof of \cite{GGRtorsion}, consider
\[q-\gamma^*q:U_\bQ\to U_\bQ^\vee(-w)\]
where $w$ is the weight of $U$.  For $v\in F^\pmax U_\bC$, if $\gamma v=\alpha v$ with $|\alpha|^2=1$ then we have $q(v,\bar v)=|\alpha|^2q(v,\bar v)=q(\gamma v,\gamma \bar v)$.  On the other hand, for $u\in U_\bC$ with no Hodge component in $\overline{F^\pmax U_\bC}$, the same is true for $\gamma u$, and thus $q(v,u)=0=q(\gamma v,\gamma u)$.  Thus, $v\in\ker(q-\gamma^*q)$.  But $\ker(q-\gamma^*q)\subset U_\bQ$ is then a nonzero sub-$\bQ$-Hodge structure of $U_\bQ$ which must be all of $U_\bQ$ since $U_\bQ$ is simple.
\end{proof}
\end{proof}
\Cref{thm:GGR} follows from \Cref{lem:GGR} as follows.  Let $L$ be the Griffiths bundle of $_\orig V$.  As in the proof of \Cref{lem:GGR}, for any connected curve $g:C\to X$ with Griffiths degree 0 we obtain a well-defined isotrivial integral mixed variation ${ _\orig V}^{\gr}(C)$ which is graded polarizable on each irreducible component.  Set $E_k\coloneqq  \gr_k^W{ _\orig V}^{\gr}(C)$.  The polarization $q$ of $_\orig V$ gives a \emph{global} isomorphism $E_{w+k}\cong E_{w-k}^\vee (-w)$ where $w$ is the weight of $_\orig V$.  We also have global isomorphisms $F^pE_k\cong \overline{F^{k-p+1}E_k}$ for all $p$ and $k$.  Thus, 
\[F^pE_{w+k}\cong (E_{w-k}/F^{w-p+1}E_{w-k})^\vee\cong \overline{F^{p-k}E_{w-k}}^\vee.\]
If $L_k$ is the Griffiths bundle of $E_k$, then it follows that $L_{w+k}\cong \overline{L_{w-k}}^\vee$, and since $g^*L=\bigotimes_k L_{k}$, we have $g^*L\cong \overline{g^*L}^\vee$. Thus, $g^*L$ is conjugate-self dual, and hence the induced character of the $\pi_1(C)$ action on $L$ is norm 1. 

For later, we record the following consequence of the argument:

\begin{lemma}\label{lem:conjsdgriffiths}

For any $x\in X$, the Griffiths line of $_\orig V^{\gr}(x)$ is canonically conjugate self-dual.
\end{lemma}


\section{Quotient spaces}


\subsection{Equivalence relations}\label{subsect:equiv reln}Let $X$ be a proper algebraic space, and let $R\subset X(\bC)\times X(\bC)$ be a closed reflexive symmetric constructible\footnote{A closed constructible $R$ in $X(\bC)\times X(\bC)$ consists of the $\bC$-points of an algebraic variety in $X\times X$. All our relations are set-theoretic, and we refrain to call them algebraic, since we do not prescribe their schematic structure; see, e.g., \cite[Ex.~3]{K2012} for the subtleties of schematic equivalence relations.} relation and let $p_i:R\to X(\bC)$ be the two projections.  Then
\[R^+\coloneqq  (p_1\times p_2)(R\,{_{p_2}\times_{p_1}} R)\subset X(\bC)\times X(\bC)\]
is also a closed reflexive symmetric constructible relation, which on the level of points is defined by $x\sim_+ y$ if $x\sim z\sim y$ for some $z$.  Note that if for some constructible $U\subset X(\bC)$ we have 
\[R^+\cap (U\times X(\bC))=R\cap (U\times X(\bC))\]
then we also have 
\[(R^+)^+\cap(U\times X(\bC))=R\cap(U\times X(\bC)).\]
In this case, we say the $R$-related classes of $U$ are stable.  There is a maximal constructible subset $U\subset X(\bC)$ whose $R$-related classes are stable, given by $U=X(\bC)\bs p^+_1(R^+\bs R) $.

The equivalence relation generated by a closed, reflexive, symmetric, constructible relation $R$ is obtained by setting $R_0=R$ and letting $R_{j+1}=(R_j)^+$.  Then $R_j$ is the relation of being connected by a chain of $2^j$ $R$-equivalences, and $R^e\coloneqq  \bigcup_j R_j$ is the smallest equivalence relation containing $R$.  By the above, there is a maximal constructible $U_j\subset X(\bC)$ whose $R_j$-related classes are stable, and the $U_j$ form an increasing sequence of subsets of $X$.
\begin{lem}\label{lem:equiv reln}Let $X$ be a proper algebraic space and $R\subset X(\bC)\times X(\bC)$ a closed reflexive symmetric constructible relation.  Then there is a closed constructible subset $\Delta\subset X(\bC)$ such that: 
\begin{enumerate}
\item Every $x\in X(\bC)$ with non-constructible $R^e$-equivalence class is contained in $\Delta$.
\item The set of points $x\in \Delta$ with constructible $R^e$-equivalence class is contained in a countable union of nowhere dense constructible subsets of $\Delta$.
\end{enumerate}

\end{lem}
\begin{proof}By the above, the $R^e$-equivalence class of $x$ is a union of the $R_j$-related classes $C_j$ of $x$, which form an increasing sequence of constructible subsets.  If $\bigcup_j C_j$ is constructible, then it must stabilize, $C_{j'}=\bigcup_j C_j$, for some ${j'}$.  In the above notation, let $\Delta_j=X(\bC)\bs U_j$.  Then the $\Delta_j$ form a decreasing sequence of constructible subsets, hence the closures $\overline \Delta_j$ stabilize to $\Delta$, and $\Delta\cap \bigcup_jU_j$ is a countable union of nowhere dense constructible subsets of $\Delta$. 
\end{proof}

\begin{rem}For any algebraic space $X$ and proper constructible equivalence relation $R\subset X(\bC)\times X(\bC)$, the quotient $X(\bC)/R$ exists in the category of definable topological spaces and we will always mean it as such.  Note that $X(\bC)/R$ can also be endowed with a Zariski topology, which is the quotient topology obtained by endowing $X(\bC)$ with the Zariski topology.  If $q:X(\bC)\to X(\bC)/R$ is the quotient, then for any closed constructible $Z\subset X(\bC)$ its saturation $q^{-1}(q(Z))\subset X(\bC)$ is closed constructible, either by definable Chow or because it is identified with $p_2(R\cap p_1^{-1}(Z))$. 
\end{rem}


\subsection{Hodge-theoretic equivalence relations}\label{subsect:Hodge equiv reln}

Let $(X,D)$ be a proper log smooth algebraic space and $(V_\bZ,F^\bullet V_\cO)$ a polarizable integral pure CY variation on $X\bs D$ with unipotent local monodromy and integrable Hodge bundle $L$.  There are several natural equivalence relations on $X(\bC)$.  
\subsubsection{}\label{subsubsect:R_tr}We define $R_\trans\subset X(\bC)\times X(\bC)$ to be the equivalence relation defined by $x\sim_\trans y$ if $V^\trans(x)\cong V^\trans(y)$ as (unpolarized) integral pure Hodge structures.
\subsubsection{}\label{subsubsect:R_curve}We define $R_\mathrm{curve}\subset X(\bC)\times X(\bC)$ to be the equivalence relation defined by $x\sim_\mathrm{curve}y$ if there is a proper connected curve $g:C\to X$ with $x,y\in g(C)$ for which $\deg g^*L=0$. For any irreducible component $C_0$ of such a curve, and $X_\Sigma$ the unique stratum containing the generic point of $C_0$, the transcendental part of $V^\trans_\Sigma$ restricted to $C_0\cap X_\Sigma$ is isotrivial by \Cref{lem:deligne}.  It follows that the pointwise transcendental parts $V^\trans(g(c))$ for $c\in C$ form an isotrivial not-necessarily-polarizable variation of integral pure Hodge structures $V^\trans(C)$ over $C$ which is polarizable in restriction to each irreducible component.  In particular, $R_\mathrm{curve}\subset R_\trans$.
\begin{lem}\label{lem:torsion comb mono hodge vs tr}The Hodge bundle $L=F^m\cV$ has torsion combinatorial monodromy if and only if for any proper curve $g:C\to X$ with $\deg g^*L=0$, the isotrivial variation $V^\trans(C)$ has finite monodromy.
\end{lem}
\begin{proof}Let $c\in C$ be a basepoint. Since $V^\trans(C)$ is an isotrivial variation of pure Hodge structures (see previous paragraph), its monodromy acts via Hodge automorphisms.  Since $V^\trans(g(c))=V^\trans(C)_{c}$ is simple, a Hodge automorphism is trivial if and only if it is trivial on the component $\gr^m_FV^\trans(C)_{c}$. 
\end{proof}

\begin{lem}\label{lem:equiv reln stuff}\hspace{.5in}
\begin{enumerate}
\item\label{lem:equiv reln stuff p0} For any connected constructible subset $Z$ of a $R_\trans$-equivalence class, the closure $\bar Z$ is contained in an $R_\mathrm{curve}$-equivalence class.  
\item\label{lem:equiv reln stuff p1} If a $R_\trans$-equivalence class is constructible, then it is closed and its connected components are $R_\mathrm{curve}$-equivalence classes.  
\item\label{lem:equiv reln stuff p2} For any point $x\in X(\bC)$ for which $V^\trans(x)$ is maximal rank, the $R_\trans$-equivalence class of $x$ is constructible.
\item\label{lem:equiv reln stuff p3} For any closed reflexive symmetric constructible relation $R\subset X(\bC)\times X(\bC)$ for which $R_\mathrm{curve}\subset R^e\subset R_\trans$, $R^e$ is a closed constructible equivalence relation whose equivalence classes are finite unions of $R_\mathrm{curve}$-equivalence classes.
\end{enumerate}
\end{lem}
\begin{proof}
For \eqref{lem:equiv reln stuff p0}, if a $R_\trans$-equivalence class contains a constructible set $Z$, then for any proper connected curve $C\subset X$ whose generic point is contained in $Z$, $V^\trans(C)$ as above is isotrivial, so $C$ is contained in a $R_\mathrm{curve}$-equivalence class, hence $C\subset Z$.  Thus $\bar Z$ is contained in a $R_\mathrm{curve}$-equivalence class.  Part \eqref{lem:equiv reln stuff p1} follows immediately from \eqref{lem:equiv reln stuff p0}.

For \eqref{lem:equiv reln stuff p2}, let $x\in X(\bC)$ be such a point and $Z$ its $R_\trans$-equivalence class.  Then for any $z\in Z\cap X_\Sigma$, $V^\trans(x)=V^\trans_\Sigma$.  Applying \Cref{lem:unpol orbit}, $Z\cap X_\Sigma$ is constructible, hence $Z$ is.

For \eqref{lem:equiv reln stuff p3}, by \Cref{lem:equiv reln} there is a closed constructible $\Delta\subset X(\bC)$ such that outside a countable union of nowhere dense constructible subsets $\Xi\subset \Delta$, each $R$-equivalence class is non-constructible. However applying \eqref{lem:equiv reln stuff p2} to (a resolution of) each irreducible component $\Delta_0$ of $\Delta$, each $R_\trans|_{\Delta_0}$-equivalence class in $\Delta_0$ of maximal rank transcendental part is constructible, as therefore is any $R_\trans|_\Delta$-equivalence class in $\Delta$ of maximal rank transcendental part.  If $\Delta$ were nonempty, there would then be a point $x\in\Delta\bs \Xi$ of maximal rank transcendental part, hence constructible $R_\trans|_\Delta$-equivalence class $Z$.  By \eqref{lem:equiv reln stuff p1}, $Z$ is partitioned into finitely many constructible $R_\mathrm{curve}$-equivalence classes, hence also into finitely many constructible $R^e$-equivalence classes.  This is a contradiction, so $\Delta=\varnothing$. 
\end{proof}

\subsection{Properties of stratifications}\label{subsect:good strat}
Let $(X,D)$ be a proper strictly log smooth algebraic space and $V=(V_\bZ,F^\bullet V_\cO)$ a polarizable integral pure CY variation on $X\bs D$ with unipotent local monodromy.  As in \Cref{subsect:map to Pn}, again consider the period map of $V^\trans_\Sigma$ 
\[\begin{tikzcd}
X_\Sigma^\define\ar[r,"f_\Sigma^\define"]&Y_\Sigma^\define\ar[r,"\psi_\Sigma"]&\Gamma_\Sigma\backslash \bD_\Sigma.
\end{tikzcd}\]
Note that some power of the Hodge bundle naturally descends to $Y_\Sigma$.  Recall that by the Griffiths criterion, $f_\Sigma:X_\Sigma\to Y_\Sigma$  extends to a proper map $\breve f_\Sigma:\breve X_\Sigma\to Y_\Sigma$ where $\breve X_\Sigma$ is the union of strata obtained from $\overline{X_\Sigma}$ by deleting divisors $E$ along which $V^\trans_\Sigma$ has nontrivial monodromy. 
\begin{property}\label{defn:clarified}Let $(X,D)$ be a proper strictly log smooth algebraic space and $V=(V_\bZ,F^\bullet V_\cO)$ a polarizable integral pure CY variation on $X\bs D$ with  neat\footnote{Local unipotent monodromy is sufficient to define (B1) and (B3). The neatness is used only for (B2).} 
monodromy. Let $R$ be a closed constructible equivalence relation on $X(\bC)$ such that $R_\mathrm{curve}\subset R\subset R_\trans$. We define the following properties of a boundary component $\Sigma\subset \pi_0(D^\reg)$. 
\begin{enumerate}
\item[(B1)\>\>\>]\label{b1} For any irreducible curve $C\subset X_\Sigma$ whose closure $\bar C\subset\overline{X_\Sigma}$ has degree 0 with respect to the Hodge bundle, $C$ is contained in a fiber of $f_\Sigma$.  
\item[(B2)\>\>\>]\label{b2} The open set $Y^\circ_\Sigma\subset Y_\Sigma$ from \Cref{lem:tr split} is all of $Y_\Sigma$.
\item[(B3)$_R$]\label{b3} For any other stratum
 $X_{\Sigma'}$, every irreducible component of $R\mid_{X_{\Sigma}\times X_{\Sigma'}}$ is surjective onto $X_\Sigma$.
\end{enumerate}
We say $(X,D)$ satisfies (B1), (B2), or (B3)$_R$ if every boundary stratum does so. For $R=R_{\mathrm{curve}}$
we simply write (B3) with no subscript.
\end{property}
\begin{remark} 
Condition \customref{b1}{(B1)} is equivalent to the Hodge bundle being strictly nef on $Y_\Sigma$, i.e., on a log smooth compactification of a resolution, the extended Hodge bundle has positive degree on any curve meeting the interior. For any irreducible curve $C\subset X_\Sigma$ whose closure $\bar C\subset\overline{X_\Sigma}$ has degree 0 with respect to the Hodge bundle,  $(V^{\trans}_{\Sigma}|_{C})^{\trans}$ is isotrivial by \Cref{lem:deligne}, but in general $V^{\trans}_{\Sigma}|_{C}$ is not and the period map $f_{\Sigma}$ may not be constant along $C$; see \S\ref{eg:BBvsBBH}. However, the integrablity grants that the locus of such curves is contained in a strict algebraic subvariety of $X_{\Sigma}$, so, up to an snc modification, $X_{\Sigma}$ satisfies \customref{b1}{(B1)} as shown in \Cref{lem:clarified}.
Condition \customref{b3}{(B3)}${}_{R}$ grants that the  stratification of $X$ by $X_{\Sigma}$ induces a stratification $\fY_{S}$ on $\fY\coloneqq  X(\bC)/R$ making $q \colon |X| \to \fY\coloneqq  X(\bC)/R$ a morphism of stratified spaces. Condition \customref{b2}{(B2)} allows to descent $V^{\trans}_{\Sigma}$ to $\fY_{S}$; see proof of \Cref{thm:easy algebraize}. 
\end{remark}

\begin{lem}\label{lem:dominanceXYSigma}

Consider a closed constructible equivalence relation $R$ as above satisfying \customref{b1}{(B1)}. Let $R_1=R\cap X_{\Sigma}\times X_{\Sigma'}$ and let $R_0\subset R_1$ be an irreducible component. 

\begin{enumerate}

\item If $R_0$ is not dominant over $X_\Sigma$, then it is not dominant over $Y_\Sigma$.

\item If $R_0$ is dominant over $X_{\Sigma}$, then the complement of its projection $X_{\Sigma}\bs\pi_{X_{\Sigma}}(R_0)$ is not dominant over $Y_\Sigma$.

\end{enumerate}

\end{lem}

\begin{proof}
Note that for a general $y\in Y_{\Sigma}$, the fiber $f_{\Sigma}^{-1}(y)$ is equidimensional of dimension $\dim X_{\Sigma}-\dim Y_{\Sigma}$. Since $R_\curve\subset R$, it follows that there is a relation $S\subset Y_\Sigma\times Y_{\Sigma'}$ such that $R_1$ is the pullback of $S$.  By \customref{b1}{(B1)}, and since $R\subset R_{\trans}$ is a closed constructible equivalence relation, it follows from \Cref{lem:equiv reln stuff}.\eqref{lem:equiv reln stuff p3} that $S$ is quasifinite over $Y_{\Sigma}$. Since $R_0$ is an irreducible component of $R_1$ it follows that it is an irreducible component of $(f_{\Sigma}\times f_{\Sigma'})^{-1}(S_0) \subset X_{\Sigma} \times X_{\Sigma'}$
for some irreducible component $S_0\subset S$.

For (1), assume that $R_0$ is dominant over $Y_{\Sigma}$. Then for a general $y\in Y_{\Sigma}$, for $(y,y')\in S_0$ we must have that the fiber $R_{0,(y,y')}$ of $R_0$ over $(y,y')$ is an irreducible component of $f_{\Sigma}^{-1}(y)\times f_{\Sigma'}^{-1}(y')$, and hence dominates an irreducible component of $f_{\Sigma}^{-1}(y)$. Thus $R_0$ dominates a subset of $X_\Sigma$ of dimension $\dim X_{\Sigma}$, and hence is dominant over $X_\Sigma$ as desired.

For (2), assume that $R_0$ is dominant over $X_{\Sigma}$. Then it is also dominant over $Y_{\Sigma}$, and so just like the above it follows that $\pi_{X_{\Sigma}}(R_0)$ contains an irreducible component of $f_{\Sigma}^{-1}(y)$ for a generic $y\in Y_{\Sigma}$. Since $X_{\Sigma}$ is irreducible it follows that $\pi_{X_{\Sigma}}(R_0)$ in fact contains all of $f_{\Sigma}^{-1}(y)$ for a generic $y\in Y_{\Sigma}$, which implies the claim. 
\end{proof}

\begin{lem}\label{lem:clarified}Let $(X,D)$ be a proper log smooth algebraic space and $V=(V_\bZ,F^\bullet V_\cO)$ a polarizable integral pure CY variation on $X\bs D$ with unipotent local monodromy, and integrable Hodge bundle $L$.  
\begin{enumerate}
\item Let $(X',D')$ be a proper strictly log smooth algebraic space and $\pi:X'\to X$ a modification with $\pi^{-1}(D)\subset D'$.  If $X_\Sigma$ satisfies \customref{b1}{(B1)} (resp. \customref{b2}{(B2)}), the so does any stratum mapping to $X_\Sigma$.
\item There is a proper strictly log smooth algebraic space $(X',D')$ and a modification $\pi:X'\to X$ with $\pi^{-1}(D)\subset D'$ such that $(X',D')$ satisfies \customref{b1}{(B1)}.  If the monodromy of $V_\bZ$ is neat, and $R$ is a closed constructible equivalence relation as in \Cref{defn:clarified}, then we may take $(X',D')$ to satisfy \customref{b2}{(B2)} and \customref{b3}{(B3)$_{R'}$}, where $R'$ is the pullback of $R$ to $X'$.
\end{enumerate}

\end{lem}
\begin{proof}
The first part is clear, since if $X'_{\Sigma'}$ maps to $ X_\Sigma$, then $V'^\trans_{\Sigma'}$ and $V'^\mini_{\Sigma'}$ are naturally the pullbacks of $V^\trans_\Sigma$ and $V^\mini_\Sigma$.

For the second part, we may assume $(X,D)$ is strictly log smooth.  For each stratum $\Sigma\subset \pi_0(D^\reg)$, by the integrability assumption, the Hodge bundle is big on a log smooth compactification $\bar Y_0$ of a resolution $Y_0$ of $Y_\Sigma$.  Thus, there is a closed strict subvariety $Z\subset Y_\Sigma$ containing all subvarieties for which the Hodge bundle is not big on a log smooth compactification of a resolution, namely the image in $Y_\Sigma$ of the non-big locus of the Hodge bundle in $\bar Y_0$.  Let $\pi:X'\to X$ be an embedded log resolution of the preimage $f_\Sigma^{-1}( Z)$, and let $D'$ be the union of the (reduced) exceptional divisors and the reduction of $\pi^{-1}(D)$.  Then for any stratum $X'_{\Sigma'}$ of $(X',D')$ mapping to $X_\Sigma$, the non-big locus of the Hodge bundle of $Y'_{\Sigma'}$ has strictly smaller dimension, and the stratum mapping to $X_\Sigma\bs f_\Sigma^{-1}(Z)$ satisfies \customref{b1}{(B1)}.  By induction on the dimensions of the non-big locus in the period image using (1), it follows that for each $\Sigma$ there is a proper log smooth $(X'',D'')$ and a modification $\sigma:X''\to X$ with $\sigma^{-1}(D)\subset D''$ for which every stratum mapping to $X_\Sigma$ satisfies \customref{b1}{(B1)}.  We may find a proper log smooth $(X''',D''')$ with a modification $\tau:X'''\to X$ with $\tau^{-1}(D)\subset D'''$ which factors through the modification $X''\to X$ we thereby construct for each $\Sigma$, and again using (1) it follows that $(X''',D'''')$ satisfies \customref{b1}{(B1)}.  

The claim for \customref{b2}{(B2)} follows by the same argument, except in the argument from the previous paragraph, we take $Z\subset Y_\Sigma$ to be the complement of $Y^\circ_\Sigma$.

We prove the final claim for \customref{b3}{(B3)$_R$} by descending induction on the dimension of $Y_\Sigma$.  Thus, assume the condition holds for any $\Sigma\subset\pi_0(D^\reg)$ with $\dim Y_\Sigma>k$, and consider a stratum $\Sigma$ with $\dim Y_\Sigma=k$. Suppose that there are strata $X_{\Sigma'}$ such that there are components of $R\mid_{X_{\Sigma}\times X_{\Sigma'}}$ whose projections are not surjective onto $X_{\Sigma}$. For each such component $R_0$, let $Z_0$ be its projection to $X_\Sigma$ if it is not dominant onto $X_\Sigma$, and the complement in $X_{\Sigma}$ of its projection if it is dominant.  Let $Z\subset X_{\Sigma}$ be the union of the closures of all of these $Z_0$s (as $R_0$ ranges over all components), and now pass to a log resolution of $f_{\Sigma}^{-1}(\ol{f_{\Sigma}(Z)})$. By \Cref{lem:dominanceXYSigma}, it follows that only strata with period image of dimension strictly smaller than $k$ are produced. On the other hand, the stratum above $X_{\Sigma}\bs f_{\Sigma}^{-1}(\ol{f_{\Sigma}(Z)})$ now satisfies \customref{b3}{(B3)$_R$}.  Continuing in this way, we are done by induction.
 \end{proof}


\subsection{The quotient by $R_\curve$}\label{subsect:R_per}
Suppose $(X,D)$ satisfies Property \customref{b1}{(B1)} and for each stratum $\Sigma$ let $f_\Sigma:X_\Sigma\to Y_\Sigma$ be the Stein factorization of the period map as introduced therein.
Define
\[R_\Sigma\coloneqq  X_\Sigma\times_{Y_\Sigma} X_\Sigma\subset  X_\Sigma\times X_\Sigma\]
and 
\[R_{\mathrm{per}}=\bigcup_{\Sigma\subset\pi_0(D^\reg)}\overline{R_\Sigma}(\bC)\subset X(\bC)\times X(\bC)\]
which is a closed, reflexive, symmetric, constructible relation on $X$.  Denote $R_{\mathrm{per}}^e$ the equivalence relation it generates.


\begin{lem}\label{lem:equiv reln stuff curve}Suppose that $V_\bZ$ has local unipotent monodromy and $(X,D)$ satisfies Property \customref{b1}{(B1)}.  Then $R_\mathrm{curve}=R_\mathrm{per}^e$.
\end{lem}

\begin{proof}
The containment
$R_\mathrm{curve}\subset R^e_\mathrm{per}$ is immediate from Property \customref{b1}{(B1)}, so it suffices to prove $R_\mathrm{per}\subset R_\mathrm{curve}$.  For any point $(x,y)\in R_\mathrm{per}$, there is a smooth curve $g:C\to R_\Sigma$ for some $\Sigma$ containing $(x,y)$ in the closure of its image.  This means we have two maps $C\rightrightarrows X_\Sigma$ with the same composition to $Y_\Sigma$.  The base-change $(X_\Sigma)_C 
\coloneqq C \times_{Y_{\Sigma}} X_{\Sigma}$ then admits two sections and has geometrically connected generic fiber.  There is therefore a surface $S\subset (X_\Sigma)_C$ flat over $C$ containing the two sections whose generic fiber is geometrically connected.  There is then a proper surface $\bar S$ flat over $\bar C$ compactifying $S/C$ with a map $\bar S\to X$ extending the map $S\to X$ such that the fibers of $\bar S/\bar C$ are connected and one of of these fibers $F$ has image in $X$ containing both $x$ and $y$.  The Hodge bundle has degree 0 on $F$ since it does so on the generic fiber of $\bar S/\bar C$, so $x\sim_\mathrm{curve}y$.
\end{proof}

\begin{cor}\label{cor:Rcurve is alg} Let $(X,D)$ be a proper log smooth algebraic space and $V=(V_\bZ,F^\bullet V_\cO)$ a polarizable integral pure CY variation on $X\bs D$ with unipotent local monodromy and integrable Hodge bundle $L$.  Then $R_\mathrm{curve}$ is a closed constructible equivalence relation on $X(\bC)$.
\end{cor}
\begin{proof}  According to \Cref{lem:clarified} there is a modification $\pi:X'\to X$ such that $(X',D')$ satisfies Property \customref{b1}{(B1)} with respect to the pullback variation.  The relation $R_\mathrm{curve}$ on $X(\bC)$ is clearly the image of the corresponding relation on $X'(\bC)$, so the claim follows from \Cref{lem:equiv reln stuff}\eqref{lem:equiv reln stuff p3} and \Cref{lem:equiv reln stuff curve}.
\end{proof}


\subsection{Hodge strata}\label{subsect:Hodge strata}
Let $(X,D)$ be a proper strictly log smooth algebraic space and $V=(V_\bZ,F^\bullet V_\cO)$ a polarizable integral pure CY variation on $X\bs D$ with unipotent local monodromy and integrable Hodge bundle.  Let $R\subset X(\bC)\times X(\bC)$ be a closed constructible equivalence relation with $R_\curve\subset R\subset R_\trans$ and assume $(X,D)$ satisfies Property \customref{b3}{(B3)$_R$}.  Then for any strata $\Sigma,\Sigma'\subset \pi_0(D^\reg)$, there exist $x\in X_\Sigma$ and $x'\in X_{\Sigma'}$ such that $x\sim_R x'$ if and only if for every point $x\in X_\Sigma$ there is a point $x'\in X_{\Sigma'}$ such that $x\sim_R x'$.  Thus, the saturation of any stratum $X_\Sigma$ with respect to $R$ is a union of strata. 
\begin{defn}\label{defn:hodge strata}  
In the above situation, we say $\Sigma\sim_R\Sigma'$ if there are points $x\in X_{\Sigma},x'\in X_{\Sigma'}$ with $x\sim_R x'$.  We refer to an equivalence class $S\subset P(\pi_0(D^\reg))$ with respect to this relation, as well as $X_S\coloneqq  \bigcup_{\Sigma\in S}X_\Sigma$, as an $R$-stratum.  We refer to $R_\curve$-strata as Hodge strata. 

\end{defn}

\begin{lem}\label{lem:glueingalongHodge} In the above situation, suppose $(X,D)$ satisfies Properties \customref{b1}{(B1-3)}. Then there exists a local system $V^{\min,\vee}(S)_\bQ\subset j_*(V^\vee_\bQ)\mid_{\mathfrak{T}(S)}$ and a quotient local system $V^{\mathrm{min},\vee}(S)_\bQ\to V^{\trans,\vee}(S)_\bQ$ whose restriction to each $\mathfrak{T}(\Sigma)$ for $X_\Sigma\subset X_S$ agrees with $V^{\min,\vee}(\Sigma)_{\bQ}\subset j_*(V^\vee_\bQ)\mid_{\mathfrak{T}(S)}$ and $V^{\mathrm{min},\vee}(\Sigma)_\bQ\to V^{\trans,\vee}(\Sigma)_\bQ$.

\end{lem}

\begin{proof}For any union of strata $Z$, let $i_Z:Z\to X$ denote the inclusion.  By \Cref{lem:subsgood}, it suffices to show there is a subsheaf $V_{S,\bQ}^{\mathrm{min},\vee}\subset i^*_{X_S}j_*(V_\bQ^\vee)$ (resp. a quotient $V_{S,\bQ}^{\mathrm{min},\vee}\to V_{S,\bQ}^{\trans,\vee}$) restricting to $V_{\Sigma,\bQ}^{\mathrm{min},\vee}\subset i^*_{X_\Sigma}j_*(V_\bQ^\vee)$ (resp. $V_{\Sigma,\bQ}^{\mathrm{min},\vee}\to V_{\Sigma,\bQ}^{\trans,\vee}$) for each $X_\Sigma\subset X_S$.

Let $E\cap \overline{X_\Sigma}$ be a divisorial boundary component of $X_\Sigma$ in the same Hodge stratum, which means there is a curve $C\subset X_\Sigma$ whose closure meets $E\cap \overline{X_\Sigma}$ and is contracted by $f_\Sigma$. Then by the Griffiths criterion, $V_\Sigma^{\trans,\vee}$ and $f_\Sigma$ extend over $E\cap \overline{X_\Sigma}$, as therefore does $V_\Sigma^{\min,\vee}$ by \Cref{lem:tr split} using \customref{b2}{(B2)}. Thus, both extend to the closure $\breve X_\Sigma$ of $X_\Sigma$ in the Hodge stratum $X_S$ containing $X_\Sigma$.  Call these extensions $ V^{\trans,\vee}_{\breve X_\Sigma}$ and $ V^{\min,\vee}_{\breve X_\Sigma}$; note that the underlying local system $ V^{\min,\vee}_{\breve X_\Sigma,\bQ}$ is naturally a subsheaf of $i^*_{X_\Sigma}j_*(V^\vee_\bQ)$.

A subsheaf of $i_{X_S}^*j_*(V^\vee_\bQ)$ is uniquely determined by subsheaves of $i_Z^*j_*(V^\vee_\bQ)$ for each closed union of strata $Z\subset X_S$ which agree on intersections.  By Property \customref{b3}{(B3)}, any stratum $X_{\Sigma'}$ in $\breve X_\Sigma$ dominates $Y_\Sigma$, so $i^*_{X_{\Sigma'}}V_{\breve X_\Sigma,\bQ}^{\mathrm{min},\vee}=V_{X_{\Sigma'},\bQ}^{\mathrm{min},\vee}$, hence there is a subsheaf $V_{S,\bQ}^{\mathrm{min},\vee}\subset i_{X_S}^*j_*(V^\vee_\bQ)$ restricting to each $V_{\Sigma,\bQ}^{\mathrm{min},\vee}$ which is therefore a local system.  Likewise, there is a quotient $V_{S,\bQ}^{\mathrm{min},\vee}\to V_{S,\bQ}^{\trans,\vee}$ restricting to the quotient $V_{\Sigma,\bQ}^{\mathrm{min},\vee}\to V_{\Sigma,\bQ}^{\trans,\vee}$ for each $X_\Sigma\subset X_S$, and this completes the proof.
\end{proof}
By the previous lemma, we obtain the following diagram by projecting the Hodge bundle to $V^\mathrm{min}(S)_\cO$ which restricts to the diagram in \Cref{cor:strata lift} for every $X_\Sigma\subset X_S$, after choosing a path from $x_\Sigma$ to $x_S$.

\begin{equation}
\begin{tikzcd}\label{boundary period}
\widetilde{\mathfrak{T}(S)}^{V^{\mini}(S)_\bQ}\ar[rrr,"\rho(S)"]&&&\bP(V^\mathrm{min}_{S,\bC,x_S})^\an \eqqcolon \bP_S^\an\\
\widetilde{X_S}^{V^{\mini}_{S,\bQ}}\ar[rrr,bend right=15,swap,"\rho_S"]\ar[u]\ar[r]&\widetilde{X_S}^{V^\trans_{S,\bQ}}\ar[r,"\phi^{\trans}_S"]&\check\bD(V^\trans_{S,\bC,x_S})^\an\ar[r]&\bP(V^\trans_{S,\bC,x_S})^\an.\ar[u]
\end{tikzcd}
\end{equation}

\begin{prop}\label{lem:thickened period maps Hodge strata} Let $(X,D)$ be a proper log smooth algebraic space and $V=(V_\bZ,F^\bullet V_\cO)$ a polarizable integral pure CY variation on $X\bs D$ with neat monodromy, integrable Hodge bundle $L$, and $(X,D)$ satisfying Property \customref{b3}{(B1-3)}.  For any Hodge stratum $S$, the morphism $\rho(S)$ from \eqref{boundary period} is $\pi_1$-definable analytic as in \Cref{defn:pi1definable}.  The pullback of the Hodge bundle to $\widetilde{\mathfrak{T}(S)}^{V^{\mini}(S)_\bQ}$ is naturally identified (as a $\pi_1$-definable analytic line bundle) with the pullback of $\cO_{\bP_S^\define}(-1)$. Finally, if the Hodge bundle has torsion combinatorial monodromy, then the connected components of the fibers of $\rho_S$ are identified via the covering map with the fibers of $X_S\to \fY\coloneqq  X(\bC)/R_\curve$.  In particular, they are compact.

\end{prop}

\begin{proof}The definability is clear by \cite{BMull}, as is the statement about the Hodge bundle, so it remains to prove the statement about the fibers.  Let $\widetilde {F}$ be a fiber of $\rho_S$.  Since the union of $\pi_1$-translates of $\widetilde {F}$ is definable on the lift of any definable open of $X_S$, by definable Chow \cite[Corollary 4.5]{definechow} it is the inverse image of a closed algebraic $F\subset X_S$.  The Hodge bundle is flat along each component of $F$, and since the boundary satisfies Property \customref{b1}{(B1)}, the variation $V^\trans_{S,\bQ}$ is isotrivial on $F$.  For any irreducible curve $C$ in $F$ and $\bar C$ the closure in $X$, the local monodromy of $V^\trans_{S,\bQ}|_C$ is then trivial by neatness, so $\bar C\subset X_S$ and therefore $\bar C\subset F$.  Thus, $F$ is proper.  Clearly the image of $F$ in $\fY$ has dimension 0; for any proper connected curve $C$ in a fiber of $|X|\to\fY$ meeting $F$, the inverse image in $\widetilde{X_S}^{V^{\mini}_{S,\bQ}}$ is contained in a fiber of $\rho_S$ and meets $\widetilde {F}$, hence is contained in $\widetilde{F}$. Thus, $C\subset F$, and $F$ is a full fiber of $|X|\to\fY$.  Finally, by \Cref{lem:torsion comb mono hodge vs tr} and the \Cref{lem:nodal curve monodromy} below, the monodromy of $V^\trans_{S,\bQ}|_F$ is finite since the Hodge bundle has torsion combinatorial monodromy, as therefore is the monodromy of $V^\mini_{S,\bQ}|_F$ by \Cref{lem:tr split} and \Cref{lem:inject autos}.  By neatness both have trivial monodromy.  It follows that every connected component of $\widetilde{ F}$ is a copy of $F$. \end{proof}
\begin{lem}\label{lem:nodal curve monodromy}

Let $F$ be a connected algebraic space.  Then there is a nodal curve $g:C\to F$ such that $g_*:\pi_1(C^\an)\to \pi_1(F^\an)$ is surjective.
\end{lem}

\begin{proof}

Note that an algebraic space is locally contractible, and so it is enough to show that for any point $p\in F$ there exists a contractible neighborhood $p\in U\subset F$ that the points of $U$ are locally connected via algebraic curves. 
To prove this, first note that by passing to an \'etale open neighborhood of $p$, we may assume $F$ is an affine scheme.  By Noether normalization, we may assume that $F$ has a finite surjective map to affine space, and since the pullback of a curve under a finite map is still a curve, the claim follows.
\end{proof}


\subsection{Algebraizations}\label{subsect:algebraize}
\begin{defn}
Let $X$ be an algebraic space, and $\phi:|X|\to \mathfrak{M}$ be a continuous surjective map of definable topological spaces. We say $\phi$ is algebraic with source $X$ (or just algebraic if $X$ is clear from context) if there exists a morphism of algebraic spaces $f:X\to Y$ and an identification $\mathfrak{M}\cong |Y| $ such that $|f|:|X|\to |Y|$ is identified with $\phi$ via this identification.
\end{defn}

Note that the algebraic space $Y$ is not unique, but if $\phi$ is proper with connected fibers, we may require $\cO_Y\to f_*\cO_X$ to be an isomorphism, in which case the algebraic space structure on $Y$ is unique.

\begin{lem}\label{lemma:semiample on open}Let $(X,D)$ be a proper log smooth algebraic space and $V=(V_\bZ,F^\bullet V_\cO)$ a polarizable integral pure CY variation on $X\bs D$ with neat monodromy, integrable Hodge bundle $L$ with torsion combinatorial monodromy, and $(X,D)$ satisfying Property \customref{b3}{(B1-3)}.    Let $q \colon |X|\to \fY\coloneqq  X(\bC)/R_\curve$ be the quotient map, $X_\mathfrak{U}\subset X$ an open union of Hodge strata, and $\mathfrak{U}\coloneqq  q(|X_\mathfrak{U}|)$ the image.  If $|X_\mathfrak{U}|\to\mathfrak{U}$ is algebraized by a fibration $f_\mathfrak{U}:X_\mathfrak{U}\to Y_\mathfrak{U}$, then for some $k>0$, $L^k|_{X_\mathfrak{U}}=f_{\mathfrak{U}}^*A$ for an ample line bundle $A$ on $Y_\mathfrak{U}$.  Moreover, the vanishing sections (in the sense of \Cref{thm:vanishing}) of some power of $A$ define a locally closed embedding $Y_\mathfrak{U}\to \bP^N$.
\end{lem}
\begin{proof}By \Cref{lem:thickened period maps Hodge strata} the line bundle $L$ is trivial on the completion of $X$ along any fiber, hence descends to $L_Y$ on $Y_\mathfrak{U}$.  For each Hodge stratum $X_S\subset X_\fU$, the image $Y_S\subset Y_\fU$ is algebraic, and since $X_S\to Y_S$ has connected fibers, the finite part $g:Y'_S\to Y_S$ of the Stein factorization $X_S\to Y_S'\to Y_S$ is a homeomorphism, and in particular, birational.  By \Cref{lem:torsion comb mono hodge vs tr} and \Cref{lem:inject autos} the variation $V^\trans_S$ descends to $Y_S'$ and has $g^*L_{Y_S}$ as its Hodge bundle.  By \Cref{lem:everything is compatible} the Hodge bundle satisfies the requirements of \Cref{setup}, so the claim then follows from \Cref{thm:vanishing}.
\end{proof}


\section{Semiampleness}\label{sect:semiample}

In this section, we prove the following:
\begin{thm}\label{thm:semiample Hodge}
Let $( X,D)$ be a proper log smooth algebraic space,  $(V_\bZ,F^\bullet V_\cO)$ a polarizable integral pure CY variation of Hodge structures on $X\bs D$ with unipotent local monodromy, and $L$ the Hodge bundle on $X$.  If $ L$ is integrable with torsion combinatorial monodromy, then it is semiample. 
\end{thm}
We deduce the same statement for the Griffiths bundle, where the integrability and combinatorial monodromy conditions are automatic:
\begin{cor}\label{thm:semiample Griffiths}
    Let $( X,D)$ be a proper log smooth algebraic space,  $({_\orig V}_\bZ,F^\bullet {_\orig V}_\cO)$ a polarizable integral pure variation of Hodge structures on $X\bs D$ with unipotent local monodromy, and $L$ the Griffiths bundle on $ X$.  Then $ L$ is semiample. 
\end{cor}
\begin{proof}[Proof of \Cref{thm:semiample Griffiths} given \Cref{thm:semiample Hodge}]  Apply \Cref{lemma:Griffiths is integrable} and \Cref{thm:GGR}.
\end{proof}
\begin{rem}\label{rem: alteration check}We remark that both the integrability of the Hodge bundle and the torsion combinatorial monodromy condition can clearly be checked after pulling back along a dominant proper morphism $(X',D')\to (X,D)$.
\end{rem}

\subsection{Proof of \Cref{thm:semiample Hodge}}
The main step in the proof of \Cref{thm:semiample Hodge} is the following:

\begin{thm}\label{thm:easy algebraize}
    Let $(X,D)$ and $V=(V_\bZ,F^\bullet V_\cO)$ be as in \Cref{thm:semiample Hodge}.  Then the quotient $q:|X|\to \fY\coloneqq  X(\bC)/R_\mathrm{curve}$ is algebraic.
\end{thm}

\begin{proof}We first reduce to the case where the monodromy of $V_\bZ$ is neat.  Let $L$ be the Hodge bundle of $X$.  By adjoining enough level structure, there is a proper log smooth algebraic space $(X',D')$ and a morphism $g:X'\to X$ restricting to a finite \'etale cover $g|_{X'\bs D'}:X'\bs D'\to X\bs D$ such that $g|_{X'\bs D'}^*V_\bZ$ has neat monodromy and $X\bs D$ is the quotient of $X'\bs D'$ be a finite fixed-point free group action by $G$.  The finite part $h:X''\to X$ of the Stein factorization of $X'\to X$ is then the normalization of $X$ in the function field of $X'$, so the group action by $G$ extends to $X''$ and $X$ is the quotient.  Thus, there is a norm map from sections of $h^*L$ to sections of $L^{|G|}$, so $L$ is semiample if and only if $h^*L$ is, which in turn is semiample if and only if $g^*L$ is, since $X'\to X''$ is a fibration.  Thus, we may assume the monodromy is neat.  

By \Cref{lem:clarified} we may assume $(X,D)$ satisfies Property \customref{b3}{(B1-3)} after replacing $(X,D)$ with a modification.  Each Hodge stratum $X_S$ is saturated with respect to the quotient map $q$, and we therefore obtain a locally closed (in the quotient Zariski topology) stratification $\fY_S\coloneqq  q(|X_S|)$ of $\fY$.

\begin{claim}\label{claim:Bundles Descend} Let $\fU\subset \fY$ be an open union of strata, and $X_\fU\subset X$ the open subspace with underlying topological space $q^{-1}(\fU)$.  Then $|X_\fU|\to \fU$ is algebraic. 

\end{claim}

We prove the claim by induction, adding one stratum at a time.  The base case (the case of the open stratum) is a consequence of \cite[Theorem 1.1]{BBT23}.  For the general case, we may assume there is a stratum $\fY_S\subset \fU$ which is closed in $\fU$ and such that, setting $\fU'=\fU\bs \fY_S$, $|X_{\fU'}|\to \fU'$ is algebraized by a fibration $f_{
\fU'}:X_{\fU'}\to U'$.  According to \Cref{lemma:semiample on open}, there is a $N'$-dimensional space of sections of $L^k$ on $X$ which yields a morphism $X_{\fU'}\to \bP^{N'-1}$ which factors as $X_{\fU'}\to U'\to \bP^{N'-1}$ where $U'\to \bP^{N'-1}$ is a locally closed embedding.  

Recall that by \Cref{cor:Rcurve is alg} the quotient $\fU=X_{\fU}(\bC)/R_\curve$ naturally exists in the category of definable topological spaces \cite[Chap. 10, (2.15) Theorem]{Dries}.  Take a DR-neighborhoods  $X_S\subset \fT(S)\subset X_\fU$, $\fY_S\subset \fT_\fY(S)\subset\fU$ for which $X_{\fT_\fY(S)}\coloneqq  q^{-1}(\fT_\fY(S))\subset \fT(S)$.  By \Cref{lem:thickened period maps Hodge strata}, we obtain a $\pi_1$-definable analytic morphism $$\rho(S):\widetilde{\fT(S)}^{V^\mini(S)_{\bQ}}\to \bP^{\an}_S.$$  Moreover, the restriction of $V^\mini(S)_\bQ$ descends to $V^\mini_{\fY}(S)_\bQ$ on $\fT_{\fY}(S)$ by \customref{b2}{(B2)} and \Cref{lem:glueingalongHodge}, and the restriction of $\rho(S)$ to $\widetilde{X_S}^{V^\mini_{S,\bQ}}$ factors through $\widetilde{\fY_S}^{V^\mini_{\fY,S,\bQ}}$ (via $\rho_S$).
Let $\bL_S$ be the total space of $\cO_{\bP_S}(-k)^{N'}$ on $\bP_S$.  From the above $N'$-dimensional space of sections of $L^k$, we obtain a lift of (the restriction of) $\rho(S)$ to a $\pi_1$-definable analytic morphism $$\sigma(S):\widetilde{X_{\fT_\fY(S)}}^{V^\mini(S)_{\bQ}}\to \bL_S^\an$$
that has fibers with compact connected components and the finite part of whose Stein factorization is locally a closed immersion away from $\widetilde{X_S}^{V^\mini_{S,\bQ}}$.  In fact, the connected components of the fibers in $\widetilde{X_S}^{V^\mini_{S,\bQ}}$ are identified with those of the quotient $|X_\fU|\to\fU$ via the covering map: they are contained in the latter, since they are contained in the fibers of $\rho(S)$, and every Hodge degree 0 curve in $X_S$ lifts to $\widetilde{X_S}^{V^\mini_{S,\bQ}}$ and must be contracted.  Thus, the \emph{topological} Stein factorization of $\sigma(S)$ is $\widetilde{\fT_{\fY}(S)}^{V^\mini_{\fY}(S)_{\bQ}}$.

\def\fW{\mathfrak{W}}
As in \cite{BBTshaf}, by definable triangulation, there is a definable cover $\fU_i$ of $\fT_{\fY}(S)$ by contractible open subsets which therefore lift to $\widetilde{\fT_{\fY}(S)}^{V^\mini_{\fY}(S)_{\bQ}}$.  Letting $X_{\fU_i}\coloneqq  q^{-1}(\fU_i)\subset X^{\define}$ be the corresponding open definable analytic subspaces which lift to $\widetilde{X_{\fT_\fY(S)}}^{V^\mini(S)_{\bQ}}$, by definable Stein factorization \cite[Theorem 1.7]{BBTshaf} there is a Stein factorization $X_{\fU_i}\to \cU_i\to \bL_S^\define$ in the category of definable analytic spaces, and $|\cU_i|$ is canonically identified with $\fU_i$ via the quotient map, by the above.  We therefore obtain a definable analytic space structure on $\fT_{\fY}(S)$ for which the quotient map $|X_{\fT_\fY(S)}|\to \fT_{\fY}(S)$ underlies a morphism of definable analytic spaces and is compatible with $f_{\fU'}^\define :X_{\fU'}^\define \to U'^\define$, hence we obtain a morphism of definable analytic spaces $X_\fU^\define \to \cU$ which is identified with $|X_U|\to \fU$ on the level of definable topological spaces.  By the definable image theorem (\Cref{thm:image}), this morphism is algebraic, whence the claim.
\end{proof}

\begin{proof}[Proof of \Cref{thm:semiample Hodge}]  Combine \Cref{lemma:semiample on open} and \Cref{thm:easy algebraize}.
\end{proof}

\subsection{Examples and counterexamples}The integrability and combinatorial monodromy conditions in \Cref{thm:semiample Hodge} are clearly both necessary.  We give some examples showing that neither condition is sufficient on its own.
\begin{eg}\label{eg:int not tor}
We give an example of a CY variation whose Hodge bundle is integrable but has nontorsion combinatorial monodromy:

Let $F/\bQ$ be a real quadratic field with ring of integers $\cO_F$.  Consider the lattice $U_{\bZ,0}\coloneqq  \Res_{\cO_F/\bZ}\cO_F^2$ with $q_0\coloneqq  \tr_{F/\bQ}\langle\,,\,\rangle$ where $\langle\,,\,\rangle$ is the standard $\cO_F$-linear antisymmetric pairing on $\cO_F^2$. 
 The two embeddings $\iota_1,\iota_2:F\to \bR$ give embeddings $\SL_2(\cO_F)\to\SL_2(\bR)$, and the Hilbert modular surface $X=\SL_2(\cO_F)\backslash \bH^2$ parametrizes weight one Hodge structures on $U_{\bZ,0}$ polarized by $q_0$ for which the splitting over $\bR$ into $\iota_1$ and $\iota_2$ eigenspaces is a decomposition of real Hodge structures.  Let $U=(U_\bZ,F^\bullet U_\cO,q)$ be the associated polarized variation of Hodge structures on $X$.  Then $\wedge^2U$ is a CY variation whose Hodge bundle is the Griffiths bundle of $U$, hence semiample.

 Equip $\bZ(0)^2$ with the standard diagonal polarization $\lambda$ and consider the polarizable integral variation of Hodge structures  
 \[V = \bZ(0)^2\otimes_{\bZ(0)} U\]
 on $X$ which is polarized by $\lambda\otimes q$.  Let $e\coloneqq  e_1+ie_2\in\bC(0)^2$, which is $\lambda$-isotropic, and let $U_\bR=U_1\oplus U_2$ be the splitting over $\bR$ into the two eigenspaces.  Consider the variation $V'$ which has the same underlying integral local system as $V$, but the Hodge filtration is changed by shifting: 
 \[V'_\cO= \bC e\otimes_\bC (U_1)_\cO(-2,2)\oplus \bC(0)^2\otimes_\bC (U_2)_\cO \oplus \bC\bar e\otimes_\bC (U_1)_\cO(2,-2).\]
 Since $\bC e\otimes_\bC U_1$ is a $\lambda\otimes q$-isotropic sub-$\bC$-variation and its conjugate is $\bC \bar e\otimes_\bC U_1$, the shift defines a new $\bC$-variation and is still polarized by $\lambda\otimes q$.  It has the following properties:
 \begin{enumerate}
 \item $V'$ is a simple polarizable integral pure CY variation of level five with Hodge bundle $F^3V'_\cO=F^1(U_1)_\cO$. Thus, its Hodge bundle is trivial along the leaves of one of the two transcendental foliations of $X$ given by the product structure on $\bH^2$, and so the Hodge bundle is not integrable. The Hodge bundle is the canonical bundle of the other foliation. The fact that it is not semiample is a classical example of the failure of foliated abundance; see \cite[Example IV.5.5]{McQuillan2008}. 
 \item After passing to a finite-index $\Gamma\subset\SL_2(\cO_F)$, $X$ has a log smooth compactification, the connected components of whose boundary are cycles of rational curves.  The Hodge bundle $F^3\cV'$ is trivial on each of these curves.  The monodromy of $U_\bZ$ in a neighborhood of one of those connected components is given by upper-triangular matrices (see e.g. \cite[$\S$I.5]{amrt})
 \[\begin{pmatrix}\alpha&\beta\\0&\alpha^{-1}\end{pmatrix}\]
 for $\alpha$ ranging over a finite index subgroup of units in $\cO_F$ and $\beta$ ranging over some ideal class of $\cO_F$.  The combinatorial monodromy on $F^1\cU_j$ around the cycle is given by multiplication by $\iota_j(\alpha^2)$, and so the Hodge bundle of $V'$ does not have torsion combinatorial monodromy (nor does it have norm one combinatorial monodromy, in accordance with \Cref{lem:GGR}). 
 \end{enumerate}
 The CY variation $V'\otimes \wedge^2 U$ then has integrable Hodge bundle, but nontorsion combinatorial monodromy. 
\end{eg}
Note also that $U_1$ from \Cref{eg:int not tor} shows that the integral structure is important in the statement of \Cref{thm:semiample Griffiths}, and that a $\bar \bZ$-structure is not sufficient.

\begin{eg} \label{eg:tor not int}
We give an example of a CY variation whose Hodge bundle has torsion combinatorial monodromy (because there is no boundary) but is non-integrable:

We may alternatively describe the above example as follows: fixing a degree 4 CM field $K/\bQ$, $\langle\,,\,\rangle$ the standard hermitian form on $\cO_K^2$, and taking $V_{\bZ,0}\coloneqq  \Res_{\cO_K/\bZ}\cO_K^2$ with $q_0\coloneqq   \tr_{K/\bQ}\langle\,,\,\rangle$, $X$ parametrizes Hodge structures on $V_{\bZ,0}$ polarized by $q_0$ such that the eigenspaces associated to each embedding $K\to \bC$ are complex sub-Hodge structures.  The variation $V'$ is obtained from the resulting variation $V$ by shifting a conjugate pair of factors in this decomposition.

Now fix a degree $2n\geq 6$ CM field $K/\bQ$, and a hermitian pairing $\langle\,,\,\rangle$ on $\cO_K^2$ which is indefinite for exactly 2 conjugate pairs of embeddings.\footnote{Recall that the corresponding period domain has an $\bH$ factor for each conjugate pair of embeddings for which the form is indefinite.  On the other hand, since the form is definite for one conjugate pair of embeddings, the monodromy of $V_\bZ$ embeds in a compact group and therefore no nontrivial unipotents.}  Performing the same shifting construction, we obtain a CY variation on a variety $X$ with no boundary and whose Hodge bundle is not integrable. Moreover, the Hodge bundle in this case is strictly nef since there are no curves which lift to a fiber of $\bH\times \bH$ and thus the Hodge bundle vacuously has torsion combinatorial monodromy.
\end{eg}
\begin{rem}Neither of the variations in \Cref{eg:int not tor} and \Cref{eg:tor not int} is manifestly algebraic, and it is not clear to the authors whether a geometric polarizable integral pure CY variation automatically has Hodge bundle which is integrable and has torsion combinatorial monodromy.  Of course, \Cref{thm:integrability} and \Cref{thm:torsion} below show this is the case for the variation on middle cohomology for a family of $K$-trivial varieties. 
\end{rem}

\begin{eg} In both \Cref{thm:semiample Hodge} and \Cref{thm:semiample Griffiths}, the assumption of unipotent local monodromy is necessary.  This can even be seen on the level of local systems with finite monodromy.  Let $L=\cO_{\bP^1}(d)$ and consider the ruled surface $X=\bP(\cO_{\bP^1}\oplus L)$, which is the total spaces of $L$ and $L^\vee$ glued along the canonical map $L\bs 0\to L^\vee\bs 0: s\mapsto s^{\vee}$.  Let $s_0,s_\infty$ be the 0 sections of $L,L^\vee$, respectively, thought of as sections of $X\to\bP^1$.  Let $G=\bZ/2\bZ$ act on $X$ by scaling by $\pm 1$, let $\pi:X\to Y=G\backslash X$ be the quotient, and note that $(Y,D)$ with $D\coloneqq  \pi(s_0)+\pi(s_\infty)$ is log smooth.  Let $V_\bZ=(\pi_*\bZ_{(X\bs D)^\an})^-$ be the local system on $(Y\bs D)^\an$ of anti-invariant locally constant functions on $(X\bs s_0\cup s_\infty)^\an$.  The local monodromy around each component of $D$ is $\pm 1$, and the lower canonical Deligne extension $\cV$ is generated by $q^{1/2}v$ where $v$ is a flat section and $q$ is a local defining equation.  The pullback $\pi^* V$ has a global flat section which extends with simple poles along $s_0$ and $s_\infty$ as a section of $\pi^*\cV$, so we have $\pi^*\cV\cong \cO_X(-s_0-s_\infty)$.  In particular, $\pi^*\cV|_{s_0}\cong L^\vee$ and $\pi^*\cV|_{s_\infty}\cong L$, so for $d\neq 0$, $\cV$ is not semiample.
\end{eg}
\begin{rem}

There is however a natural modification of a power of the Griffiths bundle which is semiample when the local monodromy is not unipotent.  Indeed, for any polarizable integral pure variation of Hodge structures $V=(V_\bZ,F^\bullet V_\cO)$ on a log smooth algebraic space $(X,D)$, by adjoining enough level structure there is a finite cover $\pi:X'\to X$ with $X'$ normal such that $X$ is the quotient of $X'$ by a group action $G$ and such that $V'\coloneqq  \pi_{X\bs D}^*V$ has unipotent local monodromy, where $\pi_{X\bs D}$ is the corestriction to $X\bs D$.  Taking a log resolution $\pi':X''\to X$, a power $L'^k$ of the Griffiths bundle $L'$ of $V'' \coloneqq \pi'^*_{X'\bs D'}V'$ descends to a line bundle $L^{(k)}$ on $X$, and using the norm map we deduce that $L^{(k)}$ is semiample.  Likewise for the Hodge bundle, assuming the integrability and combinatorial monodromy conditions (on $X$). This will be the natural polarization of the Baily--Borel compactification discussed in the next section.
\end{rem}

\section{Baily--Borel compactifications of period images}

\subsection{Preliminaries}

Let $X$ be a smooth algebraic space and $\phi:X^{\an}\to\Gamma\backslash \bD$ be a period map associated to a polarizable integral pure variation of Hodge structures, and $\Gamma$ is any discrete group containing the image of the monodromy representation. By \cite[Theorem 1.1]{BBT23} (see also \cite[Corollary 2.11]{bbtmixed}) the closure of the image is naturally a quasiprojective variety: there is a unique dominant morphism $f:X\to Y$ to a quasiprojective variety $Y$ and a closed immersion $\iota:Y^{\an}\to\Gamma\backslash \bD$ such that $\phi=\iota\circ f^{\an}$. Our goal in this section is to prove the existence of a canonical compactification $Y^{\BB}$ of $Y$ which we call the \textit{Baily--Borel} compactification.

The following is a more intrinsic and slightly more general notion than image of a period map. Let $Y$ be a reduced and irreducible algebraic space with a quasifinite Griffiths transverse\footnote{That is, for any resolution $Z\to Y$, $Z^\an\to \Gamma\backslash\bD$ is Griffiths transverse.} morphism $\phi:Y^\an\to \Gamma\backslash \bD$, where $\Gamma$ is a discrete group preserving an integral lattice.  After passing to a finite morphism $f:Y'\to Y$ by adjoining level structure, there will be a polarizable integral pure variation of Hodge structure $_\orig V'$ on $Y'$ with monodromy contained in $\Gamma$, and the period map $Y'\to \Gamma\backslash \bD$ will factor through $\phi$.  In this situation, it follows from \cite{BBT23} that the Griffiths bundle $L_Y$ is ample on $Y$.  Henceforth, we refer to such a $Y$ as a  ``variety with quasifinite period map.''

Now, for general $\Gamma$ the Griffiths bundle exists on the stack $[\Gamma\bs\bD]$ and a power of it $L_{[\Gamma\bs\bD]}^{k_\Gamma}$ descends to the coarse space $\Gamma\bs\bD$. We shall use the notation $L_{\Gamma\backslash\bD}^{(k_\Gamma)}$ for the descent, as the descent is not necessarily a power of a line bundle.

\begin{defn}
Let $Y$ be a reduced and irreducible algebraic space with a Griffiths transverse morphism $\phi:Y\to \Gamma\backslash \bD$. For each non-negative integer $n$ we define $H^0_{\bdd}(Y,L_Y^{(nk_\Gamma)})\subset H^0(Y,L_Y^{(nk_\Gamma)})$ to be those sections $s$ whose Hodge norm $|s|$ grows sub-polynomially at the boundary.  We call these the sections \textit{of moderate growth}. We define the associated graded ring
$B_Y\coloneqq  \bigoplus_{n\geq 0}H^0_{\bdd}(Y,L_Y^{(nk_\Gamma)})$ with $H^0_{\bdd}(Y,L_Y^{(nk_\Gamma)})$ given degree $nk_\Gamma$.
\end{defn}

Note that for any proper log smooth algebraic space $(X,D)$ and dominant morphism $\pi:X\backslash D\to Y$ for which the composition with the period map $\phi\circ \pi^\an$ is locally liftable (to $\bD$) and such that the induced local system has unipotent local monodromy, a section $s\in H^0(Y,L_Y^{(nk_\Gamma)})$ has Hodge norm of moderate growth if and only if its pullback $\pi^*s$ extends to a section of the Schmid extension $(L^{nk_\Gamma}_{X\bs D})_X$ of the $nk_\Gamma$-th power of the Griffiths bundle $L_{X\bs D}$ of the pullback variation $_\orig V_{X\bs D}$ by \cite{kashiwara}. Note also that for any morphism $f:(X',D')\to (X,D)$ of proper log smooth algebraic spaces and any polarizable integral pure variation of Hodge structures $_\orig V$ on $X\bs D$, there is always an injection $f^*(L^k_{X\bs D})_X\to (L^k_{X'\bs D'})_{X'}$ of Schmid extensions of powers of the associated Griffiths bundles.  Thus, for any morphism $g:Z\to Y$, we get an induced pullback map
$g^*:B_Y\to B_Z$.

Our main result is then as follows:

\begin{thm}\label{thm:bailyborel}
Let $Y$ be a variety with quasifinite period map. Then

\begin{enumerate}

\item $B_Y$ is finitely generated, $Y^{\BB}\coloneqq  \proj B_Y$ is a projective variety, and the natural morphism $j:Y\hookrightarrow Y^{\BB}$ is an open embedding.

\item There exists a minimal positive $k_\Gamma| k_Y$ such that locally on $Y^{\BB}$ there are sections of $L_Y^{(k_Y)}$ whose Hodge norms and inverse Hodge norms have moderate growth.

\item For $k_Y|n$, $\cO_{Y^{\BB}}(n)$ exists as a line bundle and is ample for $n$ positive.  The natural inclusion $\cO_{Y^\BB}(n)\subset j_*(L_Y^{(n)})$ is the subsheaf of sections of moderate growth. 
\item 
Let $(\testsp,D_\testsp)$ be a log smooth algebraic space and $g: \testsp\bs D_\testsp\to Y$ a morphism for which the composition $(\testsp\bs D_\testsp)^\an\xrightarrow{g^\an}Y^\an\to \Gamma\bs \bD$ is locally liftable.  Then $g: \testsp\bs D_\testsp\to Y$ extends to a morphism $\bar g:\testsp\to Y^\BB$ and for $k_Y|n$, $\bar g^*\cO_{Y^\BB}(n)$ is canonically identified with the Schmid extension $(L^n_{\testsp\bs D_\testsp})_\testsp$ of the $n$th power of the Griffiths bundle $L_{\testsp\bs D_\testsp}$ of the induced variation $_\orig V_{\testsp\bs D_\testsp}$ on $\testsp\bs D_\testsp$.

\end{enumerate}
\end{thm}

We remark that the above properties generalize the construction for Shimura varieties, which is why we call it the Baily--Borel compactification. 
\def\Hod{H}

\subsubsection{Hodge case}  If $Y$ is a variety with quasifinite period map $\phi:Y^\an\to \Gamma\bs \bD$ and $\bD$ parametrizes CY Hodge structures, we may instead consider the Hodge bundle $M_Y^{(\ell_\Gamma)}$ and likewise define the moderate growth sections $H^0_{\bdd}(Y,M_Y^{(n\ell_\Gamma)})\subset H^0(Y,M_Y^{(n\ell_\Gamma)})$ and $C_Y\coloneqq  \bigoplus_{n\geq 0}H^0_{\bdd}(Y,M_Y^{(n\ell_\Gamma)})$ with $H^0_{\bdd}(Y,M_Y^{(n\ell_\Gamma)})$ given degree $n\ell_\Gamma$.  Note that in this case $M_Y^{(\ell_\Gamma)}$ may not be ample.

We say that:
\begin{itemize}
\item $M_{Y}^{(\ell_\Gamma)}$ is strictly nef if for any nonconstant irreducible smooth curve $g:C\to Y$ for which the composition $C^\an\xrightarrow{g^\an} Y^\an\to \Gamma\bs \bD$ is locally liftable and the resulting variation $V_C$ has unipotent local monodromy, the Schmid extension $M_{\bar C}$ of the Hodge bundle $M_C$ to the smooth compactification $C\subset\bar C$ has positive degree.
\item $M_Y^{(\ell_\Gamma)}$ is integrable (resp. has torsion combinatorial monodromy) if for some (hence any) proper log smooth algebraic space $(\testsp,D_\testsp)$ with a proper dominant generically finite morphism $g:\testsp\bs D_\testsp\to Y$ for which the composition $(\testsp\bs D_\testsp)^\an\xrightarrow{g^\an} Y^\an\to \Gamma\bs \bD$ is locally liftable and the resulting variation $V_{\testsp \bs D_\testsp}$ has unipotent local monodromy, the Hodge bundle $M_{\testsp}$ of $V_{\testsp\bs D_\testsp}$ is integrable (resp. has torsion combinatorial monodromy). 
\end{itemize}

\begin{thm}\label{thm:BBHodge}Let $Y$ be a normal variety with quasifinite period map to a period space parametrizing CY Hodge structures.  Assume the Hodge bundle $M_Y^{(\ell_\Gamma)}$ is strictly nef, integrable, and has torsion combinatorial monodromy in the above sense.  Then
\begin{enumerate}

\item $C_Y$ is finitely generated, $Y^{\BBH}\coloneqq  \proj C_Y$ is a normal projective variety, and the natural morphism $j:Y\hookrightarrow Y^{\BBH}$ is an open embedding.

\item There exists a minimal positive $\ell_\Gamma|\ell_Y$ such that locally on $Y^{\BBH}$ there are sections of $M_Y^{(\ell_Y)}$ whose Hodge norms and inverse Hodge norms have moderate growth.

\item 
For $\ell_Y|n$, $\cO_{Y^{\BBH}}(n)$ exists as a line bundle and is ample for $n$ positive.  The natural inclusion $\cO_{Y^\BBH}(n)\subset j_*(M_Y^{(n)})$ is the subsheaf of sections of moderate growth.  
\item Let $(\testsp,D_\testsp)$ be a log smooth algebraic space and $g: \testsp\bs D_\testsp\to Y$ a morphism for which the composition $(\testsp\bs D_\testsp)^\an\xrightarrow{g^\an}Y^\an\to \Gamma\bs \bD$ is locally liftable.  Then $g: \testsp\bs D_\testsp\to Y$ extends to a morphism $\bar g:\testsp\to Y^{\BBH}$ and for $\ell_Y|n$, $\bar g^*\cO_{Y^{\BBH}}(n)$ is canonically identified with the Schmid extension $(M^n_{\testsp\bs D_\testsp})_\testsp$ of the $n$th power of the Hodge bundle $M_{\testsp\bs D_\testsp}$ of the induced variation $ V_{\testsp\bs D_\testsp}$ on $\testsp\bs D_\testsp$.  

 Moreover, $Y^\BBH$ is the unique normal compactification of $Y$ for which a sufficiently divisible power $M_Y^{(n)}$  extends to an ample line bundle and the above property is satisfied.

\end{enumerate}
\end{thm}

\begin{rem}\label{rem:BBtoBBH}Assuming the hypotheses of \Cref{thm:BBHodge}, there is a natural morphism $Y^\BB\to Y^\BBH$. Indeed, let $(X,D)$ be a proper log smooth algebraic space with a proper birational morphism $X\bs D\to Y$. Then $Y^\BB = \mathrm{Proj}(\bigoplus_{n\geq 0}H^0(X,L_{X}^{n}))$ and $Y^\BBH = \mathrm{Proj}(\bigoplus_{n\geq 0}H^0(X,M_{X}^{n}))$. Since $L_X\otimes M^{\vee}_{X}$ is a semipositive line bundle by \Cref{lemma:Griffiths is integrable}, any curve with zero $L_{X}$-degree has zero $M_X$-degree. Hence, $X \to Y^\BBH$ factors as $X \to Y^\BB \to Y^{\BBH}$. By the functoriality of Griffiths and Hodge bundle, the morphism $Y^\BB \to Y^{\BBH}$ is independent of $X$.

The morphism $Y^\BB \to Y^{\BBH}$ often has positive-dimensional fibers, even on the part of $Y^\BB$ which maps to $\Gamma\bs \bD$---see \Cref{eg:BBvsBBH} for an example coming from a moduli space of Calabi--Yau varieties.
\end{rem}

\subsubsection{Borel extension}
We finally prove that the compactifications $Y^\BB$ and $Y^\BBH$ satisfy an extension theorem just like in the classical cases:

\begin{thm}\label{thm:bbextensino}Let $Y$ be a variety with quasifinite period map (resp. a variety with quasifinite CY period map satisfying the hypotheses \Cref{thm:BBHodge}).  Then any analytic morphism from a polydisk $\phi:(\Delta^*)^k\to Y^{\an}$ such that the resulting morphism $(\Delta^*)^k\to \Gamma\bs\bD$ is locally liftable extends to a morphism $\ol{\phi}:\Delta^k\to Y^{\BB,\an}$ (resp. $\ol{\phi}:\Delta^k\to Y^{\BBH,\an}$).  Moreover, $\bar\phi^*\cO_{Y^\BB}(n)^\an$ (resp. $\bar\phi^*\cO_{Y^\BBH}(n)^\an$) is canonically identified with the Schmid extension of $L_{(\Delta^*)^k}^n$ (resp. $M_{(\Delta^*)^k}^n$) for $k_Y|n$ (resp. $\ell_Y|n$).
\end{thm}

The rest of the section is devoted to the proofs of \Cref{thm:bailyborel}, \Cref{thm:BBHodge}, and \Cref{thm:bbextensino}. We begin with the following compatibility lemma:

\begin{lem}\label{lem:unipgriffithspullback}
Let $(X,D)$ be a log smooth algebraic space and $V$ a polarizable integral pure variation of Hodge structures on $X\bs D$ with unipotent local monodromy. Let $f:(X',D')\to (X,D)$ be a morphism of log smooth algebraic spaces. Then letting $L_X,L_{X'}$ be the Schmid extensions of the Griffiths bundle, we have $f^*L_X=L_{X'}$.  Moreover, if $V$ is a CY-variation, the same compatibility holds for the Hodge bundles.

\end{lem}

\begin{proof}

The Lemma follows from the fact that the Deligne extension on $(X,D)$ pulls back to the Deligne extension on $(X',D')$. Recall that the Deligne extension as the unique extension with residues having eigenvalues with real part in $(-1,0]$. In the case of unipotent monodromy, the eigenvalues are $0$, and so the same follows for the pullback.
\end{proof}

\subsection{Reduction to the neat case}
We first reduce parts (1), (2), and (3) of \Cref{thm:bailyborel} and \Cref{thm:BBHodge} to the case when the variation in question has neat monodromy.  The argument is the same in both cases, so we do it for \Cref{thm:bailyborel}.  By adjoining enough level structure, there is a finite morphism $\pi:Y'\to Y$ which is the quotient map for a finite group action $G$ and such that the induced map $\phi':Y'\to \Gamma\backslash \bD$ is the period map of a variation with neat monodromy.  In both cases, the pullback $\pi^*: B_Y\to B_{Y'}$ is the inclusion of the $G$-invariant subring, and by the existence of the global norm map it follows that if (1) holds for $Y'$, then it also holds for $Y$ by taking the quotient of $Y'^{\BB}$ by $G$.  The existence of the local norm map implies (2) locally, so the set of $k_Y$ in question is nonempty.  The set of all integral $k_Y$ satisfying (2) clearly form a (nontrivial) group, so taking the minimal positive one, (2) follows.  Given (2), part (3) will now follow if the sheaf of moderate growth sections of 
$j_*(L_Y^{(n)})$ is a line bundle, but since this is the $G$-invariants of the corresponding subsheaf on $Y'^{\BB}$, this is clear.

\subsection{Proof of \Cref{thm:BBHodge} (1), (2), and (3)}
By the above reduction, we may assume $\Gamma$ is neat.  By resolution of singularities, let $(X,D)$ be a proper log smooth algebraic space with a proper birational morphism $X\bs D\to Y$. Since $Y$ is normal, it follows that $X\bs D\to Y$ is a fibration. Applying Theorem \ref{thm:semiample Hodge}, we obtain a fibration $f:X\to \bar Y$ such that $M_X$ descends to an ample bundle $M_{\bar Y}$, by \Cref{lemma:semiample on open}. Since $M_{X\bs D}$ is pulled back from $Y$ and is strictly nef on $Y$ it follows that $\bar Y\bs f(D)\cong Y$ and hence that $\bar Y$ is a compactification of $Y$. We prove now that $\bar Y\cong \proj C_Y$.

Indeed, first note that every element of $H^0_{\bdd}(Y,M_Y^{n})$ pulls back to an element of $H^0(X,M^{n}_X)$ and thus descends to an element of $H^0(\bar Y,M^{n}_{\bar Y})$. Conversely, any element of $H^0(\bar Y,M^{n}_{\bar Y})$ is a section of moderate growth, and thus belongs to $C_Y$. Hence, we see that $C_Y= \bigoplus_{n\geq 0}H^0(\bar Y,M_{\bar Y}^{n})$, from which the claim follows, and (1) is proved.  Part (2) is clear since a local generator of $M_{\bar Y}^n$ pulls back to a generator of $M_X^n$.  As above, for (3) it suffices to argue that the subsheaf of moderate growth sections of $j_*(M_Y^n)$ is a line bundle provided local sections as in (2) exist, but this follows from the normality of $\bar Y$ since any moderate growth function on $Y$ (locally on $\bar Y$) extends to $\bar Y$.\qed

\subsection{Proof of \Cref{thm:bailyborel} (1)}
\def\Ynormal{{Z}}

By the above reduction, we may assume $\Gamma$ is neat.  Let $\nu:\Ynormal\to Y$ be the normalization of $Y$.  By applying \Cref{thm:BBHodge} (using \Cref{lemma:Griffiths is integrable} and \Cref{thm:GGR}), we have a Baily--Borel compactification $\Ynormal^\BB$ satisfying the requirements of \Cref{thm:bailyborel}.  Let $R_0\coloneqq  \Ynormal(\bC)\times_{Y(\bC)} \Ynormal(\bC)\subset \Ynormal(\bC)\times \Ynormal(\bC)$ be the equivalence relation defining the map to $Y(\bC)$, let $\ol{R_0}\subset \Ynormal^\BB(\bC)\times \Ynormal^\BB(\bC)$ be its closure, and $R=(\ol{R_0})^e\subset \Ynormal^\BB(\bC)\times \Ynormal^\BB(\bC)$ the equivalence relation it generates.  Let $(X,D)$ be a proper strictly log smooth algebraic space $(X,D)$ with a proper birational morphism $X\bs D\to Y$, which necessarily factors through $Z$.

\begin{defn}

For a point $x\in X$, we refer to $H(x)$ as the triple of rational mixed Hodge structures and morphisms of rational mixed Hodge structures $$\big({_{\orig}V^{\gr}}(x),\;V^{\min,\vee}(x),\;\iota_x:\gr^W V^{\min,\vee}(x)\hookrightarrow \bigotimes_p \Exterior^{\rk F^p{_\orig V}} {_{\orig}}V^{\gr}(x)^{\vee}\big)$$
with the obvious notion of isomorphism $H(x)\to H(y)$, namely, isomorphisms ${_{\orig}V^{\gr}}(x)\to {_{\orig}V^{\gr}}(y)$ and $\;V^{\min,\vee}(x)\to \;V^{\min,\vee}(y)$ for which the induced maps commute with $\iota_x$ and $\iota_y$.

\end{defn}

The fiber $L(x)$ of the Griffiths bundle at $x$ is realized as the Hodge line in $V^{\mini}(x)$, and by \Cref{claim:torsion} and \Cref{lem:conjsdgriffiths} any automorphism of $H(x)$ induces a torsion automorphism of $L(x)$, whose order is bounded by $\rk {_\orig V}$.
\begin{cor}\label{cor:trivial on power}There exists a positive integer $N$ such that for any $x,x'\in X$, there is at most one isomorphism $L^N(x)\to L^N(x')$ induced by an isomorphism $H(x)\to H(x')$.
\end{cor}

We obtain a natural equivalence relation $R_{H}$ on $X$ with $x\sim_H y $ if $H(x)\cong H(y)$.  As in \S\ref{subsect:good strat}, we have $R_{\curve}\subset R_{H}\subset R_{\trans}$, and so $R_{H}$ descends to $Z^\BB(\bC)=X(\bC)/R_\curve$.  We abusively use the same notation $R_H$ for the equivalence relation on $Z^\BB(\bC)$.

\begin{lemma}\label{lem:Rpreserveshodge}

The equivalence relation $R$ is closed constructible, finite, and $R\subset R_H$. In particular, $R_H$ descends through the quotient $q:|Z^\BB|\to \bar Y\coloneqq  Z^\BB(\bC)/R$.

\end{lemma}

\begin{proof}
Let $R_X\subset X(\bC)\times X(\bC)$ be the closure of the pullback of $R_0$; note that it surjects onto $\ol{R_0}$.  For any $(x,y)\in R_X$, there is a curve $C\subset R_X\cap (X\bs D)\times(X\bs D)$ whose closure contains $(x,y)$.  The pullback of the variation to $C$ under the two resulting maps $C\rightrightarrows X\bs D$ are equal, as therefore are the limit mixed Hodge structures, so $H(x)\cong H(y)$.  Thus, $R_\mathrm{curve}\subset (R_X\cup R_\mathrm{curve})^e\subset  R_H\subset R_\trans$, so by \Cref{lem:equiv reln stuff}\eqref{lem:equiv reln stuff p3} $(R_X\cup R_\mathrm{curve})^e$ is closed constructible.  Since $\Ynormal^\BB(\bC)=X(\bC)/R_\mathrm{curve}$, the image of $(R_X\cup R_\mathrm{curve})^e$ in $\Ynormal(\bC)\times \Ynormal(\bC)$ is $R$, so it is closed constructible as well, finite by \Cref{lem:equiv reln stuff}\eqref{lem:equiv reln stuff p1}, and contained in $R_H$ (on $Z^\BB$).
\end{proof}

Observe that, given any definable disk\footnote{It would suffice to use algebraic curves.} $\Delta^*\to \bar Y$ lifting to a definable analytic map $\Delta^*\to \Ynormal^\BB$, and any two such lifts $f,f':\Delta^*\rightarrow \Ynormal^\BB $, any choice of isomorphism $H(f(t))\cong H(f'(t))$ at a very general point $t\in \Delta^*$ extends to an isomorphism of the natural variation in the very general $H$ data on $\Delta^*$, up to shrinking $\Delta$, which then gives an isomorphism $H(f(t))\to H(f'(t))$ at every point $t\in \Delta$ via the limit mixed Hodge structure.  Moreover, the induced isomorphism $L^N(f(t))\to L^N(f'(t))$ is continuous for $t\in \Delta$.

\begin{lemma}\label{lem:refine1}
There is a proper strictly log smooth algebraic space $(X,D)$ and a proper birational morphism $X\bs D\to Y$ such that:
\begin{enumerate}
\item $(X,D)$ satisfies \customref{b1}{(B1)}, \customref{b2}{(B2)}, \customref{b3}{(B3)$_R$}.  In particular, 
\begin{enumerate}
\item The Hodge strata $X_S$ of $X$ are saturated with respect to $X\to \Ynormal^\BB$ (resp. $|X|\to \bar Y$) and descend to strata $\Ynormal^{\BB}_S$ (resp. $\bar Y_S$) of $\Ynormal^{\BB}$ (resp. $\bar Y$).
\item The $R$-strata $X_T$ of $X$ are saturated with respect to $X\to \Ynormal^\BB$ (resp. $|X|\to \bar Y$) and descend to strata $\Ynormal^{\BB}_T$ (resp. $\bar Y_T$) of $\Ynormal^{\BB}$ (resp. $\bar Y$).
\end{enumerate}
\item Each Hodge stratum $\Ynormal^\BB_S$ is smooth and each $R$-stratum $\Ynormal^\BB_T$ is a disjoint union of Hodge strata $\Ynormal^\BB_S$.
\item For each Hodge stratum $\Ynormal^\BB_S$, the $\pi_1$-definable analytic morphism $$\rho_{\Ynormal^\BB,S}:\widetilde{\Ynormal_S^\BB}^{V^\trans_{S,\bQ}}\to \bP(V^\trans_{\Ynormal^\BB,S,\bC,z(S)})^\an$$ obtained from projecting the Hodge bundle is unramified.
\end{enumerate}
\end{lemma}

\begin{proof}
We essentially redo the last part of the proof of \Cref{lem:clarified}. Namely, we construct such a stratification by descending induction on the dimension of $\Ynormal^\BB_\Sigma$ where the lemma is false, the base case being trivial.  Thus, assume the condition holds for any $S$ with $\dim \Ynormal^{\BB}_S>k$, and consider a stratum with $\dim \Ynormal^{\BB}_S=k$. Let $W\subset \Ynormal^{\BB}_S$ be the locus where either $\Ynormal^{\BB}_S$ isn't smooth or the period map isn't unramified. Note that $\dim W <k$ and consider the pullback $Z$ to $X$ of the $R$-saturation $R(\ol{W})$ of $\ol{W}$. We now pass to a log resolution of $Z$, and modify further as necessary for $(X',D')$ satisfies \customref{b3}{(B3)$_R$}. As in the proof of \Cref{lem:clarified}, only strata with period image of dimension strictly smaller than $k$ are produced. On the other hand, the new stratum $\Ynormal^{\BB}_S\bs R(W)$ now satisfies the conditions.  Continuing in this way, we are done by induction.
\end{proof}

From the proof of \Cref{thm:easy algebraize}, the variations of Hodge structures $V^\mini_{T}$ and $V^\trans_{T}$ on an $R$-stratum $X_T$ of $X$ descend to $V^\mini_{\Ynormal^\BB,T}$ and $V^\trans_{\Ynormal^\BB, T}$ on the corresponding $R$-stratum $Z^\BB_T$.  We may take DR-neighborhoods $\Ynormal_T\subset\fT_{\Ynormal^\BB}(T)\subset \Ynormal^\BB$ and $\bar Y_T\subset\fT_{\bar Y}(T)\subset \bar Y$ of the $R$-strata in both $\Ynormal^\BB$ and $\bar Y$ such that each $\fT_{\Ynormal^\BB}(T)$ (resp. $\fT_{\bar Y}(T)$) only intersects strata limiting to $\Ynormal^\BB_T$ (resp. $\bar Y_T$) and such that each $\fT_{\Ynormal^\BB}(T)$ maps to $\fT_{\bar Y}(T)$. 
\begin{cor}\label{cor:E stuff on ZBB}The conclusions of \Cref{lem:refine1} hold with respect to $E\coloneqq  \Sym^N V$ in place of $V$.  Moreover, we have the following:
\begin{enumerate}
\def\symV{E}
\item The Hodge bundle of $\symV$ descends to $\bar Y$ as a continuous line bundle $L_{\bar Y}^{(N)}$.
\item $\symV^\mini_{\Ynormal^\BB,T}$ descends to the $R$-stratum $\bar Y_T$ as a rational local system whose fibers are continuously endowed with Hodge structures, $E^\trans_{\Ynormal^\BB,T}$ descends as a subobject in this category, and the Hodge line in both is identified with $L_{\bar Y}^{(N)}$.  We call the resulting objects $\symV^\mini_{\bar Y,T}$ and $E^\trans_{\bar Y,T}$.
\item By projecting the Hodge bundle (of $E$) we have $\pi_1$-definable analytic
$$\tau(T):\widetilde{\fT_{\Ynormal^\BB}(T)}^{E^\mini_{\Ynormal^\BB,T,\bQ}}\to \bP(E^\mini_{\bar Y,T,\bC,z(T)})^\an$$
which is unramified in restriction to $\widetilde{\Ynormal^\BB_T}^{E^\mini_{\Ynormal^\BB,T,\bQ}}$ and pointwise factors through $\widetilde{\fT_{\bar Y}(T)}^{E^\mini_{\bar Y,T,\bQ}}$ (on the preimage of $\widetilde{\fT_{\bar Y}(T)}^{E^\mini_{\bar Y,T,\bQ}}$).
\end{enumerate}
\end{cor}
\begin{proof}

By \Cref{cor:trivial on power}, \Cref{lem:Rpreserveshodge}, and the observation before \Cref{lem:refine1}, $R$ gives a continuous descent datum on the Hodge bundle of $E$ (which is the $N$th power of the Hodge bundle of $V$), which proves (1).

By \Cref{lem:Rpreserveshodge} it follows that we obtain at every point of $\bar Y$ a well-defined isomorphism class of Hodge structures which are the descent of $E^\mini_{\Ynormal^\BB,T}$ and $E^\trans_{\Ynormal^\BB,T}$. By \Cref{cor:trivial on power} and \Cref{lem:inject autos} there is a canonical descent datum which is pointwise induced by an isomorphism of $H$-data, and by the observation before \Cref{lem:refine1} it is continuous, so (2) follows.

Part (3) is immediate from part (2) and part (3) of \Cref{lem:refine1}.
\end{proof}

Next, we show the objects $E_{\bar Y,T}^{\mini,\vee}$ are compatible between strata as rational local systems whose fibers are continuously endowed with Hodge structures.
In the following, we abusively denote the pullback of $R$ to $X$ by the same letter.  For any Hodge strata $S_1,S_2$ of $X$, the descent data for the $E^{\mini,\vee}_{S_i}$ and their quotients $E^{\trans,\vee}_{S_i}$ naturally gives via \Cref{lem:subsgood}
an isomorphism of local systems $R_{S_1,S_2}: p_1^*E^{\mini,\vee}(S_1)_\bQ\to p_2^*E^{\mini,\vee}(S_2)_\bQ$ on $R\cap \fT(S_1)\times \fT(S_2)$ which is compatible with the corresponding morphism on the quotients to $p_i^*E^{\trans,\vee}(S_i)_\bQ$.  Recall by \Cref{lem:compat} that if $X_{S}$ specializes to $X_{S'}$, then $E^{\mini,\vee}(S')_\bQ$ is naturally a sub-local system of $E^{\mini,\vee}(S)_\bQ$ on the intersection $\fT_{\Ynormal^\BB}(S)\cap \fT_{\Ynormal^\BB}(S')$.

\begin{lemma}\label{lem:descentcompatibility}
Let $X_{S_1}$ (resp. $X_{S_2}$) be a Hodge stratum specializing to $X_{S'_1}$ (resp. $X_{S'_2}$).
On $\big(\mathfrak{T}(S_1)\cap \mathfrak{T}(S_1')\big)\times \big(\mathfrak{T}(S_2)\cap \mathfrak{T}(S_2')\big)$ we have
$R_{S_1,S_2}\mid_{p_1^*E^{\mini,\vee}(S_1')_\bQ}=R_{S_1',S_2'}$.

\end{lemma}

\begin{proof}
Since these are maps of local systems, it is enough to check the statement at a single point in each connected component. Hence, let $\Delta^{*}\to \big(X_{S_1}\cap \mathfrak{T}(S_1')\big)\times \big(X_{S_2}\cap \mathfrak{T}(S_2')\big)$ be a definable analytic disk whose image is Zariski dense in $X_{S_1}\times X_{S_2}$, and which extends to a map $\Delta\to R$ with the origin landing in $X_{S_1'}\times X_{S_2'}$.  By the observation before \Cref{lem:refine1}, we obtain an isomorphism in the resulting two variations of the very general $H$-data over $\Delta^*$ up to shrinking $\Delta$, and therefore pointwise of the $H$-data at every point.  These isomorphisms induce $R_{S_1,S_2}$ on $E^{\mini,\vee}_{S_1}$ over $\Delta^*$ and $R_{S_1',S_2'}$ on $E^{\mini,\vee}_{S_1'}$ at $0\in\Delta$ by \Cref{cor:trivial on power} and \Cref{lem:inject autos}. By \Cref{lem:compat} the claim follows.
\end{proof}

It follows that if $\bar Y_{T}$ specializes to $\bar Y_{T'}$, the local system $E_{\bar Y}^{\mini,\vee}(T')_\bQ$ is naturally a sub-local system of $E_{\bar Y}^{\mini,\vee}(T)_\bQ$ on $\fT_{\bar Y}(T)\cap \fT_{\bar Y}(T')$, and that the restriction of the quotient $E_{\bar Y}^{\mini,\vee}(T)_\bQ\to E_{\bar Y}^{\trans,\vee}(T)_\bQ$ factors through the quotient $E_{\bar Y}^{\mini,\vee}(T')_\bQ\to E_{\bar Y}^{\trans,\vee}(T')_\bQ$.  Dually, $E_{\bar Y}^{\mini}(T')_\bQ$ is naturally a quotient of $E_{\bar Y}^{\mini}(T)_\bQ$ and the quotient map takes $E^\trans_{\bar Y}(T)_\bQ$ to $E^\trans_{\bar Y}(T')_\bQ$, again on $\fT_{\bar Y}(T)\cap \fT_{\bar Y}(T')$.

Finally, we give an algebraic structure to $\ol{Y}$. We shall follow \Cref{thm:easy algebraize}, and so we build our algebraic structure one $R$-stratum at a time, inductively.  We therefore let $ U\subset \bar Y$ be an open union of $R$-strata, $\bar Y_T\subset U$ an $R$-stratum which is closed in $U$, and we inductively suppose that $U'\coloneqq  U\bs \bar Y_T$ has been given an algebraic structure together with an algebraic map $\Ynormal^{\BB}_{U'}\coloneqq  q^{-1}(U')\to U'$ which are compatible with the definable topological space structures.  We further suppose:
\begin{enumerate}
\item[(i)] The line bundle $\cO_{\Ynormal^\BB}(M')$ of $\Ynormal^\BB$ restricted to $\Ynormal^\BB_{U'}$ descends\footnote{Note that we already established this descent as a continuous line bundle, but we want it as an algebraic line bundle.} to an ample line bundle $A'$ on $U'$.
\item[(ii)] For each $R$-stratum $\bar Y_{T'}\subset U'$, the morphism obtained by projecting the Hodge bundle
$$\tau(T'):\widetilde{\fT_{\Ynormal^\BB}(T')}^{E^\mini_{\Ynormal^\BB,T',\bQ}}\to \bP(E^\mini_{\Ynormal^\BB,T',\bC,z(T')})^\an$$
factors through $\widetilde{\fT_{\bar Y}(T')}^{E^\mini_{\bar Y,T',\bQ}}$.

\end{enumerate}
The base case $U=Y$ is trivial given the above setup.

On the one hand, by \Cref{thm:vanishing} as in \Cref{lemma:semiample on open}, we may pick a finite-dimensional homogeneous subspace of $B_Y$ yielding a linear system of sections of a power of $A'$ which extend to $\Ynormal^\BB$ and which embed $U'$ in $\bP^{N'-1}$.  On the other hand, by \Cref{cor:E stuff on ZBB} and \Cref{lem:descentcompatibility} we have $\pi_1$-definable analytic morphisms
$$\tau(T):\widetilde{\fT_{\Ynormal^\BB}(T)}^{E^\mini_{\Ynormal^\BB,T,\bQ}}\to \bP(E^\mini_{\bar Y,T,\bC,z(T)})^\an$$
which is pointwise compatible with $\tau(T')$ for each stratum $\Ynormal_{T'}^\BB\subset \Ynormal^\BB_{U'}$ by the paragraph right after \Cref{lem:descentcompatibility}.   
As in the proof of \Cref{thm:easy algebraize}, the above linear system yields a $\pi_1$-definable analytic lift
\[\pi(T):\widetilde{\fT_{\Ynormal^\BB}(T)}^{E^\mini_{\Ynormal^\BB,T,\bQ}}\to \bL_T^\an\]
where $\bL_T$ is the total space of a sum of $N'$ copies of a power of $\cO(-1)$.  Moreover, $\pi(T)$ factors through a local embedding of $\widetilde{U'} \cap \widetilde{\fT^*_{\Ynormal^\BB}(T)}^{E^\mini_{\Ynormal^\BB,T,\bQ}}$, and whose restriction to $\widetilde{\Ynormal^\BB_T}^{E^\mini_{\Ynormal^\BB,T,\bQ}}$ is both unramified and factors through $\widetilde{\bar Y_T}^{E^\mini_{\Ynormal^\BB,T,\bQ}}$ on $\widetilde{\Ynormal^\BB_T}^{E^\mini_{\Ynormal^\BB,T,\bQ}}$.  Thus, it is everywhere locally injective and factors through $\widetilde{\fT_{\bar Y}(T)}^{E^\mini_{\Ynormal^\BB,T,\bQ}}$.

Now, observe that in the analytic (resp. definable analytic) category we have: 
\begin{itemize}
\item Any morphism $f:X\to Y$ with discrete fibers factors as $X\to Z\to Y$ where $X\to Z$ is finite and $Z\to Y$ is an open embedding, up to replacing $X$ with a cover.  In the definable analytic category, this follows from \cite[Lemma 2.8]{BBTshaf}.
\item For a locally injective morphism $f:X\to Y$, a factorization $|X|\to \mathfrak{Z}\to |Y|$ with $|X|\to \mathfrak{Z}$ finite and surjective on the level of topological spaces (resp. definable topological spaces) can be uniquely lifted to a factorization $X\to Z\to Y$ for which $Z\to Y$ is unramified.  Indeed, using the previous bullet point, and the fact that any cover of $X$ can be refined by a cover consisting of the connected components of the pullback of a cover from $\mathfrak{Z}$ (which is \cite[Proposition 2.4]{BBT23} in the definable analytic category), we may assume (after passing to a cover of $\mathfrak{Z}$) that on every connected component of $X$, $X\to Y$ is a homeomorphism followed by a locally closed embedding, and that the image is identified with $\mathfrak{Z}$.  The sheaf of functions on $\mathfrak{Z}$ is then that of the image, using \cite[Proposition 2.52]{BBT23}\footnote{As we are only concerned with reduced spaces, \cite[Proposition 2.45]{BBT23} would suffice.} in the definable analytic category.

\end{itemize}
Applying the second bullet point above to $\pi(T)$ as in \Cref{claim:Bundles Descend}, it follows  that there is a definable analytic space structure on $\fT_{\bar Y}(T)$ and a morphism of definable analytic spaces $\fT_{\Ynormal^\BB}(T)\to \fT_{\bar Y}(T)$ whose underlying map on definable topological spaces is the quotient map, and therefore there is a definable analytic space structure on $U$ and a morphism of definable analytic spaces $(\Ynormal^\BB_U)^\define\to U$ whose underlying map is the quotient map and which is compatible with $(\Ynormal^\BB_{U'})^\define\to U'^\define$.  By the definable image theorem (\Cref{thm:image}), the definable analytic space structure on $U$ is (uniquely) algebraizable, as is the morphism $f_U:\Ynormal^\BB_U\to U$.  By construction, (ii) is satisfied.  Also by construction, a power $\cO_{Z^\BB_U}(M)$ descends to a definable analytic line bundle on $U$ which is naturally contained in $(f_{U*}\cO_{Z^\BB_U}(k))^\define$, hence algebraic by definable GAGA (\Cref{thm:GAGA}), and ample by \Cref{thm:vanishing}.  Thus, by induction there is an algebraic space structure on $\bar Y$ and a morphism $\Ynormal^\BB\to\bar Y$ (whose underlying map is the quotient map) such that $\cO_{Z^\BB}(k)$ descends to an ample bundle $L_{\bar Y}^{(k)}$.

To conclude, it follows that $B_{\ol{Y}}\coloneqq  \bigoplus_{k\geq 0} H^0(\ol{Y},L^{(kN)}_{\ol{Y}})\subset B_Y\subset B_{\Ynormal^{\BB}}$. Since $B_Y$ is a submodule of the finitely generated $B_{\ol{Y}}$-module $B_{\Ynormal^{\BB}}$ it follows that $B_Y$ is finitely generated, so we may define $Y^{\BB}\coloneqq  \proj B_Y$, and it follows that $Y\hookrightarrow Y^{\BB}$ since $B_{\ol{Y}}$ and hence $B_Y$ induces an embedding of $Y$.
\qed

\subsection{Proof of \Cref{thm:bailyborel} (2) and (3)}

Again by the above reduction, we may assume $\Gamma$ is neat.  We first prove part (2). Let $(X,D)$ be a log smooth algebraic space with a proper birational morphism $X\bs D\to Y$. By construction, some power $L_{X}^{(n)}$ descends to $L_{Y^{\BB}}^{(n)}$ as a line bundle. It follows that any locally generating section $s$ on $Y^{\BB}$ pulls back to a generating section on $X$ and thus has Hodge norm and inverse Hodge norms of moderate growth. Thus, the set of $k_Y$ in question is nonempty. Since it clearly forms a group, there is a minimal one and (2) follows.

We now prove (3):

\begin{lem}\label{lem:YBBfromnorm}
Consider the maps $\bar\nu:Z^{\BB}\to Y^{\BB}$, $j:Y\to Y^{\BB}$. Then we have the equality $\cO_{Y^{\BB}}=j_*\cO_Y\cap \bar\nu_*\cO_{Z^{\BB}}$, with the intersection taking place in $j_*\nu_*\cO_Z$. 
Moreover, $\cO_{Y^\BB}$ analytifies to $j^\an_*\cO_{Y^{\an}}\cap \bar\nu^{\an}_*\cO_{Z^{\BB,\an}}.$
\end{lem}

\begin{proof}
    Let $\mathcal{R}=j_*\cO_Y\cap \bar\nu_*\cO_{Z^{\BB}}$. It is clear that $\mathcal{R}$ is quasicoherent, and since it injects into $\bar\nu_*\cO_{Z^{\BB}}$ it must be coherent.

    Now consider $W=\Spec \mathcal{R}$. By construction $W$ fits into a map $Z^{\BB}\to W\to Y^\BB$. Hence, some power the Griffiths bundle $L^{(n)}$ descends to $W$ as $L^{(n)}_W$, and therefore is ample there. Thus $W=\proj B_W$. However, clearly $B_W\subset B_Y$, and thus we must have equality. It follows that $W=Y^{\BB}$ which completes the proof.

    Finally, the analytification statement would follow directly from (ordinary) GAGA if it weren't for the fact that $j_*\cO_Y$ is quasicoherent as opposed to coherent. To address that, we work locally and let $h$ be a regular function on $Y^{\BB}$ vanishing on the boundary. For $m\geq 1$ let $\mathcal{R}_m\coloneqq  h^{-m}\cO_Y\cap \bar\nu_*\cO_{Z^{\BB}}$. It is clear that $\mathcal{R}_m$ analytifies to $(\mathcal{R}^{\an})_m\coloneqq  h^{-m}\cO_{Y^{\an}}\cap \bar\nu^{\an}_*\cO_{Z^{\BB,\an}}$, and so the claim follows as $j^\an_*\cO_{Y^{\an}}\cap \bar\nu^{\an}_*\cO_{Z^{\BB,\an}}=\cup_m (\mathcal{R}^{\an})_m$. 
\end{proof}

We may define a coherent sheaf $L_{Y^{\BB}}^{(k_Y)}\subset j_*L_{Y}^{(k_Y)}$ by considering all local sections whose Hodge norms have moderate growth.

\begin{cor}\label{cor:kypowergriffiths}
 $L_{Y^{\BB}}^{(k_Y)}$ is a line bundle on $Y^{\BB}$.
\end{cor}

\begin{proof}

We work locally around a point $y\in Y^{\BB}$. By part (2), there is an affine neighborhood $y\in U$ and a local section $s\in H^0(U,L_{Y^{\BB}}^{(k_Y)})$ whose Hodge norm and its inverse have moderate growth around every point in $U\bs Y$. We claim that $s$ is a local generator around $y$. 

Suppose that $s'\in H^0(U,L_{Y}^{(k_Y)})$ is some other moderate growth section. Then $t=s'/s\in H^0(U\cap Y,\cO_{Y^\BB})$ has moderate growth, hence is bounded locally on $U\bs Y$. Since $Z^\BB$ is normal, it follows that $\nu^*t$ extends to a function on $\mathrm{int}\left(\ol{\bar\nu^{-1}(U)}\right)=\bar\nu^{-1}(\mathrm{int}(\ol{U})))$ since $\bar\nu:Z^\BB\to Y^{\BB}$ is open.  Hence $t$ extends to an element of $H^0(U,\cO_{Y^{\BB}})$ by \Cref{lem:YBBfromnorm} as desired.
\end{proof}

Finally, we complete the proof of (3). It is clear that $L_Y^{(k_Y)}$ is ample. Now it follows from \Cref{cor:kypowergriffiths} that we have $H^0(Y^{\BB},L_{Y^{\BB}}^{(k_Y)}) = H^0_{\bdd}(Y,L_{Y}^{(k_Y)})$, and thus $L_{Y^{\BB}}^{(k_Y)}$ is naturally identified with $\cO_{Y^{\BB}}(k_Y)$ as desired.\qed

\subsection{Proof of \Cref{thm:bbextensino}}
The compatibility with the Schmid extensions is immediate from parts (2) and (3) of \Cref{thm:bailyborel} (resp. \Cref{thm:BBHodge}), so we focus on the existence on the extension of the morphism.  The proof for $Y^\BBH$ is the same so we focus on the $Y^\BB$ statement.  Let $f:Y'\to Y$ be a finite \'etale cover of $Y$ with level structure so that the monodromy group is neat. There is then a finite map $\pi:\Delta^k\to\Delta^k:(z_1,\ldots,z_k)\mapsto(z_1^N,\ldots,z_k^N)$ and a commutative diagram
\[\begin{tikzcd}
(\Delta^*)^k\ar[r]\ar[d,swap,"\pi|_{(\Delta^*)^k}"]& Y'\ar[d]\\
(\Delta^*)^k\ar[r]&Y
\end{tikzcd}\]
and it is sufficient to show the top map extends.  Thus, we may assume the variation has neat monodromy.

The extension of $\phi$ is unique if it exists, so the claim is local on $(\Delta^*)^k$, and we may freely shrink $\Delta^k$.  Thus, we may assume $\phi:(\Delta^*)^k\to Y$ is definable (by \cite[Theorem 4.1]{bkt}\footnote{The statement therein should read that the local period map is definable \emph{up to shrinking $\Delta^k$}.}) and therefore that it extends meromorphically, as in \cite[Lemma 3.2]{bbtmixed}. By Hironaka's embedded resolution theorem, we may construct a tower of blowups along smooth centers $$X_r\to X_{r-1}\cdots\to X_0=\Delta^k$$ such that $\phi$ extends to a morphism $\phi_r:X_r\to Y^{\BB,\an}$ and the pair $(X_r,D_r)$ is log smooth, where $D_r$ is the union of exceptional divisors and the strict transform of the coordinate hyperplanes of $\Delta^k$.  Locally on $Y^\BB$, the Griffiths bundle has a generating section with moderate growth and whose inverse has moderate growth, and it follows that the pullback $\phi_r^*L_{Y^\BB}$ agrees with the Schmid extension of the Griffiths bundle of the variation on $X_r\bs D_r$.  But then we also have $\phi_r^*L_{Y^\BB}\cong f^* L_{\Delta^k}$ where $f:X_r\to \Delta^k$ is the blow-down and $L_{\Delta^k}$ is the Schmid extension of the variation on $(\Delta^*)^k$.  Thus, $\phi_r^*L_{Y^\BB}$ is trivial on every fiber of $f$, hence $\phi_r$ factors through $f$, as desired. \qed

\subsection{Proof of \Cref{thm:bailyborel} (4) and \Cref{thm:BBHodge} (4)}  The existence of the extension of the morphism and the compatibility with the Schmid extension is immediate from \Cref{thm:bbextensino}.  The uniqueness statement in \Cref{thm:BBHodge}(4) is standard:  if $\widehat Y$ were another such compactification of $Y$, then $X\bs D\to Y$ also extends to $X\to {\widehat Y}$, but since $\cO_{{\widehat Y}}(n)$ and $\cO_{Y^\BBH}(n)$ are both ample and pullback to the same line bundle on $X$, $X\to Y^\BBH$ factors through ${\widehat Y}$ and $X\to {\widehat Y}$ factors through $Y^\BBH$ by normality, hence ${\widehat Y}\cong Y^\BBH$.\qed


\section{Birational geometry and Hodge theory of lc-trivial fibrations}
Up to an alteration, the moduli part of a family of Calabi--Yau varieties, or more generally an lc-trivial fibration, is the Hodge bundle of a variation of Hodge structure. We discuss in detail the variation arising (\Cref{defn:modulipart}), and provide a geometric characterization of its restriction in codimension one in terms of sources of slc pairs (cf. \Cref{thm:0.19}). To this end, we first recall the notions of b-divisor, pairs, locally stable families, and canonical bundle formula. We refer to \cite{KM98} and \cite{kol13} for the standard 
terminology in birational geometry.

\subsection{B-divisors}
Let $\mathbb{K}$ denote $\bZ$ or $\bQ$.
Given a normal algebraic space $X$, 
a {\it $\mathbb{K}$-b-divisor} 
$\mathbf{D}$ is a (possibly infinite) sum 
of geometric valuations $\nu_i$ 
of $k(X)$ with coefficients in 
$\mathbb{K}$, 
\begin{align*}
 \mathbf{D}= \sum_{i \in I} b_i \nu_i, \; b_i \in \mathbb{K},
\end{align*}
such that, given any normal variety 
$X'$ birational to $X$, 
only finitely many valuations 
$\nu_{i}$ have a center of codimension 1 on 
$X'$. The {\it trace} $D_{X'}$ 
of $\mathbf{D}$ on $Y'$ is the 
$\mathbb{K}$-Weil divisor 
\begin{align*}
D_{Y'} 
 \coloneqq  \sum b_i D_i
\end{align*}
where the sum is indexed over valuations $\nu_{i}$ 
that have divisorial center $D_i\subset X'$.

Given a $\mathbb{K}$-b-divisor $\mathbf{D}$ 
over $X$, we say that $\mathbf{D}$ 
is a {\it $\mathbb{K}$-b-Cartier} 
if there exists a birational model 
$X'$ of $X$ such that $D_{X'}$ 
is $\mathbb K$-Cartier on $X'$ and for any model 
$\pi \colon X''  \rar X'$, we have
$D_{X''} = \pi^\ast D_{X'}$.
When this is the case, we will say that 
$\mathbf{D}$ \emph{descends} to $X'$
and we shall write 
$\mathbf{D}= \overline{\mathbf{D}}_{X'}$
for the $\mathbb K$-b-divisor which $D_{X'}$ determines. We say that $\mathbf{D}$ is {\it b-effective}, 
if $D_{X'}$ is effective for any model 
$X'$. We say that $\mathbf{D}$ is {\it b-nef} 
(resp.\ {\it b-semiample}), 
if it is $\mathbb{K}$-b-Cartier 
and, moreover, there exists a birational model 
$X'$ of $X$ such that 
$\mathbf{D}= \overline{\mathbf{D}}_{X'}$ 
and $D_{X'}$ is nef 
(resp.\ semiample) on $X'$. 

In all of the above, if $\mathbb{K}= \bZ$, 
we will systematically drop it from the notation.

\begin{example}
Let $(X, \Delta)$ be a log sub-pair.  
The \emph{discrepancy b-divisor} $\bfA(X, \Delta)$ is defined as follows: on a birational model $\pi : X' \to X$, its trace $A(X, \Delta)_{X'}$ is given by the identity
\[
A(X, \Delta)_{X'} = \sum_i a(D_i; X, \Delta) D_i \coloneqq  K_{X'} - \pi^*(K_X + \Delta).\]
The b-divisor $\bfA^*(X, \Delta)$ is defined by taking its trace $A^*(X, \Delta)_{X'}$ on $X'$ to be
\[
A^*(X, \Delta)_{X'}  \coloneqq  A(X, \Delta)_{X'} + \sum_{a(D_i; X, \Delta) =-1} D_i.\]
\end{example}

\subsection{Singularities of pairs} The acronyms klt, dlt, lc, sdlt, and slc describe types of singularities that occur naturally in various constructions within birational geometry.
For instance, the minimal (resp.~canonical) model of an snc pair has dlt (resp.~lc) singularities. The reduced part of the boundary of a dlt (resp.~lc) pair is sdlt (resp.~slc). The fibers of a semistable or locally stable morphism (e.g., the families of varieties parametrized by KSBA moduli spaces) have slc singularities. Finally, up to finite base change, any family of Calabi--Yau varieties over a punctured disk has a dlt log Calabi--Yau filling; see \cite{Fujino2011}. Here, we limit ourselves to recalling the relevant definitions and mentioning some properties of 
lc centers used in the following sections.  
\begin{defn}[Singularities of normal pairs]
Let $(X, \Delta)$ be a log sub-pair where $X$ is a normal algebraic space.
\begin{itemize}
\item $(X, \Delta)$ is 
\emph{Kawamata log terminal (klt)} if $\lceil \bfA(X, \Delta) \rceil \geq 0$, i.e., $a(D; X, \Delta) > -1$ for every divisor $D$. 
\item $(X, \Delta)$ is 
\emph{log canonical (lc)} if $\lceil \bfA^*(X, \Delta) \rceil \geq 0$, i.e., $a(D; X, \Delta) \geq -1$ for every divisor $D$.
\item $(X, \Delta)$ is 
\emph{purely log terminal (plt)} if $a(D; X, \Delta) > 0$ for every exceptional divisor $D$.
\item An irreducible
subvariety $Z \subset X$ of an lc sub-pair $(X, \Delta)$ is an \emph{lc center} if there exist a birational morphism $\pi \colon X' \to X$
and a divisor $E \subset X$ with $a(E; X, \Delta) =-1$ 
whose image coincides with $Z$. 
\item $(X, \Delta)$ is \emph{divisorial log terminal (dlt)} if $(X, \Delta)$ is lc and none of
its lc centres lies in the complement of the largest open locus where the sub-pair is snc. 
\end{itemize}
\end{defn}

\begin{defn}[Singularities of demi-normal pairs]
Let $(X, \Delta)$ be a log sub-pair where $X$ is demi-normal, i.e., satisfies Serre's condition $S_2$ and it is nodal in codimension $1$. Let $\nu \colon (\overline{X}, \overline{\Delta}+\overline{C}) \to (X, \Delta)$ be the normalization of $(X, \Delta)$ with conductor $\overline{C}$ and $\overline{\Delta} \coloneqq  \nu^{-1}(\Delta)$.
\begin{itemize}
\item $(X, \Delta)$ is 
\emph{semi-log canonical (slc)}, if $(\overline{X}, \overline{\Delta}+\overline{C})$ is lc.
\item $(X, \Delta)$ is 
\emph{semi-divisorial log terminal (sdlt)}, if $(X, \Delta)$ is slc and none of
its lc centres lies in the complement of the largest open locus where the sub-pair is semi-snc.
\end{itemize}
\end{defn}

\begin{defn}
A \emph{log Calabi--Yau pair} $(X, \Delta)$ is a proper lc pair with $K_X + \Delta \sim_{\bQ} 0$.
\end{defn}

All minimal lc centers of a dlt log Calabi--Yau pair are $\bP^1$-linked in the following sense. 

\begin{definition}[Standard $\mathbb{P}^1$-link]\label{def:P1link}
A \emph{standard $\mathbb{P}^1$-link} is a $\mathbb{Q}$-factorial pair $(X, D_1 + D_2 + \Delta)$ together with a proper morphism $\pi: X \to T$ such that:
\begin{enumerate}
    \item $K_X + D_1 + D_2 + \Delta \sim_{\mathbb{Q}, \pi} 0$,
    \item $(X, D_1 + D_2 + \Delta)$ is plt (in particular, $D_1$ and $D_2$ are disjoint),
    \item the morphisms $\pi|_{D_1}: D_1 \to T$ and $\pi|_{D_2}: D_2 \to T$ are isomorphisms, and
    \item every reduced fiber $X_t^{\mathrm{red}}$ is isomorphic to $\mathbb{P}^1$.
\end{enumerate}
\end{definition}
\begin{remark}
Alternatively, the total space $X$ of a standard $\mathbb{P}^1$-link is the projectivization of a split $\bQ$-vector bundle of rank 2, whose two direct summands correspond to the sections $D_1$ and $D_2$; see \cite[Thm.~1.4]{Moraga24}.
\end{remark}
\begin{definition}[$\bP^1$-linking]
Let $f \colon (Y, \Delta) \to X$ be a fibration
such that $(Y, \Delta)$ is a dlt pair and $K_{Y} + \Delta \sim_{f, \bQ} 0$,
and let \( Z_1, Z_2 \subset Y \) be two lc centers. 
\begin{itemize}
\item \( Z_1 \) and \( Z_2 \) are \emph{directly \(\mathbb{P}^1\)-linked} if there exists an lc center \( W \subset Y \) containing both \( Z_i \) such that
$
f(W) = f(Z_1) = f(Z_2)$, 
and the pair \( (W, \operatorname{Diff}^*_W(\Delta)) \) (cf. \cite[\S 4.18]{kol13}) is birational to a standard \( \mathbb{P}^1 \)-link, with \( Z_i \) mapping to \( D_i \). Observe that \( W = Y \) is allowed.
\item \( Z_1\) and \(Z_2\) are \emph{\(\mathbb{P}^1\)-linked} if there exists a sequence of lc centers $Z_1', Z_2', \dots, Z_m'$
such that \( Z_1' = Z_1 \), \( Z_m' = Z_2 \), and for each \( i = 1, \dots, m-1 \), the centers \( Z_i' \) and \( Z_{i+1}' \) are directly \( \mathbb{P}^1 \)-linked (or \( Z_1 = Z_2 \)).
\end{itemize}
\end{definition}

In particular, every \( \mathbb{P}^1 \)-linking defines a crepant birational map between the pairs \( (Z_1, \operatorname{Diff}^*_{Z_1}(\Delta)) \) and \( (Z_2, \operatorname{Diff}^*_{Z_2}(\Delta)) \).

\begin{prop}[{\cite[Thm.~4.40]{kol13}}] \label{prop:P1link} Let $f \colon (Y, \Delta) \to X$ be a projective fibration such that $(Y, \Delta)$ is a dlt pair and $K_{Y} + \Delta \sim_{f, \bQ} 0$.  All minimal lc centers of $(Y, \Delta)$ among those that intersect a fixed fiber of $f$ are $\bP^1$-linked.
\end{prop}

Among the minimal lc centers in \Cref{prop:P1link}, those dominating $X$ are of particular interest, and they are called sources of $f \colon (Y, \Delta) \to X$.

\begin{defn}[Sources]\label{defnsource} Let $f \colon (Y, \Delta) \to X$ be a fibration from an slc pair $(Y, \Delta)$ to an integral base $X$ with $K_{Y} + \Delta \sim_{\bQ, f} 0$.
A \emph{source} of $f \colon (Y,\Delta) \to X$
is a generically klt
pair obtained as an lc center, minimal among those dominating $X$, of a dlt modification $(Y^{\dlt}, \Delta^{\dlt})$ of the normalization $(\overline{Y}, \overline{\Delta}+\overline{C})$ of $(Y, \Delta)$:
\begin{equation}\label{eq:source}
(S, \Delta_S) 
\xhookrightarrow{\iota}
(Y^{\dlt}, \Delta^{\dlt}) 
\xrightarrow{\pi} 
(\overline{Y}, \overline{\Delta}+\overline{C}) 
\xrightarrow{\nu}
(Y, \Delta).
\end{equation}
It is unique up to crepant birational equivalence; see \cite[\S 4.5]{kol13}.
\end{defn}

\subsection{Locally stable families}
\begin{defn}
Let $X$ be a reduced scheme, $f \colon Y \to X$ a flat morphism of finite type
and $f \colon (Y, \Delta) \to X$ a well-defined family of pairs (see  \cite[Thm.-Def.~4.7]{kbook}). Assume that $(Y_x, \Delta_x)$ is slc for every $x \in X$. Then $f \colon (Y, \Delta) \to X$ is \emph{locally
stable} if the following equivalent conditions hold:
\begin{enumerate}
\item $K_{Y/X} + \Delta$ is $\bQ$-Cartier;
\item $(Y_{T}, \Delta_T + Y_0)$ is slc, whenever $(T, 0)$ is the spectrum of a DVR, and $f_T \colon (Y_T, \Delta_T) \to T$ is the pullback family along a morphism $T \to X$; cf.\ also \cite[Thm.-Def.~2.3]{kbook}.
\end{enumerate}
\end{defn}

We recall some properties of locally stable families.

\begin{lemma}[{\cite[Thm.~4.8]{kbook}}]\label{stabilitylocst} Let $f \colon (Y, \Delta) \to X$ be a locally stable morphism over a
reduced base $X$, and $q\colon V \to X$ be a morphism of reduced schemes. Then the family over
$V$ obtained by fiber product is locally stable.
\end{lemma}

\begin{lemma}[{\cite[Thm.~4.55]{kbook}}]\label{stabilitylocstII} Let $f \colon (Y, \Delta) \to X$ be a morphism over a
smooth scheme $X$ with $\Delta \geq 0$. 
Then $f$ is locally stable if and only if the pair
$(Y, \Delta + f^* D)$ is slc for every snc divisor
$D \subset X$.
\end{lemma}

\begin{lemma}\label{stabilitylocstIII} Let $f \colon (Y, \Delta) \to X$ be a locally stable morphism over a
smooth
base $X$, and $\nu \colon \overline{Y} \to Y$ be the normalization of $Y$ with conductor $\overline{C} \subset \overline{Y}$. 
Then $f \circ \nu \colon (\overline{Y}, \overline{\Delta}+\overline{C}) \to X$ is locally stable.
\end{lemma}
\begin{proof} If $(Y, \Delta + f^* D)$ is slc for any snc divisor $D \subset X$, then  $(\overline{Y}, \overline{\Delta}+\overline{C} + (f \circ \nu)^* D)$ is lc by \cite[Thm.~5.38]{kol13}.
\end{proof}

\begin{lemma}[{\cite[Lem.~2.12]{Z16}}]
Let $f \colon (Y, \Delta) \to C$ be a locally stable morphism over an snc curve $C$. Then $(Y, \Delta)$ is slc.
\end{lemma}

The moduli part of locally stable lc fibrations admits the following birational characterization. 
\begin{lemma}\label{lem:modulipartlctriviallocallystable}
Let $f \colon (Y, \Delta) \to X$ be a locally stable lc-trivial fibration inducing the generalized pair $(X,B, \bfM)$. Then $B=0$ and
$f^*M_X \sim_{\bQ} K_{Y/X} + \Delta$.
\end{lemma}
\begin{proof}
Let $(T, 0)$ be the spectrum of the local ring of a prime divisor $D$ on any modification of $X$. Since $f$ is locally stable, the pair $(Y_T, \Delta_T + Y_0)$ is lc by \cite[(2.3.3)]{kbook}, so $\ord_D(B)=0$ by \Cref{canonicalbundleformula}.(2). 
\end{proof}

A fibration with K-trivial general fiber can be made locally stable and lc-trivial via an alteration.

\begin{prop}\label{lem:alteration}
Let $f\colon (Y, \Delta) \to X$ be a fibration of quasiprojective varieties
whose general fiber
$(Y_{\eta_X}, \Delta_{\eta_X})$ is log Calabi--Yau. Then there is a projective, generically finite, dominant morphism $q \colon W^{\circ} \to X$, and a projective compactification $W^{\circ} \to W$, a locally stable morphism $f' \colon (Y', \Delta') \to W$ such that the pullback of the generic fiber of $f$ along $q$ is crepant birational to the generic fiber of $f'$.

Furthermore, we can assume that:
\begin{enumerate}

\item any closed locus of interest in $W$ is a simple normal crossing divisor;

\item given a polarized variation $V$ of Hodge structure supported on a smooth locally closed subset $Z^{\circ}$ in $X$, there exist a proper log smooth scheme $(R, D_{R})$ and an embedding $\iota \colon R \hookrightarrow W$ such that the composition $q \circ \iota \colon R \setminus D_{R} \to Z^{\circ}$ is projective, generically finite and surjective, and $(q \circ \iota)^*V$ has local unipotent monodromy;

\item 
$K_{Y'/W} + \Delta' \sim_{\bQ, f'} 0$; and

\item $(Y', \Delta')$ is dlt in codimension 1 over $W$, i.e., $(Y', \Delta' + f^{-1}(D))$ is dlt
over the generic point of any prime divisor $D$ in $W$.

\end{enumerate}
\end{prop}

\begin{proof} We follow closely \cite[Thm.~4.59]{kbook}. 
We can replace $X$
with a projective alteration of a compactification of $X$ satisfying (1) and (2).
To achieve (2), choose for $R$ an irreducible component, dominating $Z^{\circ}$, of a complete intersection of ample divisors in the simple normal crossing divisor $\overline{q^{-1}(Z^{\circ})}$. Eventually, first replace $X$ with a projective alteration to grant that the local monodromy of $(q \circ \iota)^{*}V$ is unipotent. 

Now, let $(\widetilde{Y}_{\eta_X}, \widetilde{\Delta}_{\eta_X}) \to (Y_{\eta_X}, \Delta_{\eta_X})$ be a log resolution of the generic fiber of $f$.
By \cite{AK00} (cf., also \cite{AGK2020}), there exists a generically finite, dominant map $q \colon W \dashrightarrow X$, from a smooth projective variety $W$, such that $(\widetilde{Y}_{\eta_X}, \widetilde{\Delta}_{\eta_X}) \times_{\eta_X} \eta_{W}$
extends to a locally stable morphism $f_{1} \colon (Y_{1}, \Delta_1) \to W$ and semistable in codimension $1$. By \cite[Thm.~4.7]{AGK2020}, $W$ can be chosen in such a way that (1) and (2) continue to hold. Observe that
$(Y_{\eta_X}, \Delta_{\eta_X}) \times_{\eta_X} \eta_{W}$ extends to a good minimal model $f_2 \colon (Y_2, \Delta_2) \to W$; see \cite[Thm.~1.1]{HX13}. By \cite[Thm.~1.7]{HH2020}, this ensures that a 
$(K_{Y_1/W} + \Delta_1)$-MMP with scaling 
of an ample divisor terminates with a minimal model 
$f' \colon (Y', \Delta') \to W$, which is again locally stable by \cite[Cor.~4.57]{kbook}.
Furthermore, since $(Y_{1}, \Delta_1)$ is semistable in codimension $1$, $(Y', \Delta')$ is dlt in codimension $1$ over $W$.
\end{proof}

We conclude the section with a technical lemma, used in \Cref{thm:torsion}, about the existence of a special sdlt modification of an slc pair over a nodal curve.

\begin{lemma}\label{lem:sdltmodification}
Let $f \colon (Y, \Delta) \to C$ be a locally stable fibration 
over a connected (strictly) snc curve $C$ with $K_{Y/C}+\Delta \sim_{\bQ} 0$. Suppose that $(Y, \Delta)$ is an sdlt pair over a dense open set $C^{\circ} \subset C$. Then there exist a
surjective
morphism $q' \colon C' \to C$ from a connected (strictly) snc curve and a fibration
$f' \colon (Y', \Delta') \to C'$ with the property that 
\begin{enumerate}
\item\label{item1:sdlt} $(Y', \Delta')$ is an sdlt pair, whose irreducible components each dominate an irreducible component in $C'$;
\item\label{item:connectednesssource} the restriction of $f'$ to any sources dominating an irreducible component of $C$ is an lc-trivial fibration (with connected fibers);
\item\label{item:pushforward} 
$f_* (\omega^{[m]}_{Y'/C'}(m\Delta')) \simeq (q')^*(f_*(\omega^{[m]}_{Y_1/C'}(m\Delta_{C'})))$ 
for any integer $m$.
\end{enumerate}
\end{lemma}

\begin{proof}
We first achieve \eqref{item:connectednesssource}, i.e., the connectedness of the fibers of the sources. Let $Z$ be the closure in $Y_{C}$ of a stratum of $Y_{C^{\circ}}$.
By \cite[Lem.~2.11]{kbook}, the restriction $(f_{C})|_Z \colon (Z, \mathrm{Diff}^*_{Z}(\Delta)) \to C_{W}$ is a locally stable morphism over an irreducible component $C_W$ of $C$, but not necessarily a fibration with connected fibers. The finite map $q_{W} \colon C'_{W} \to C_W$ in the Stein factorization $W \to C'_{W} \xrightarrow{q_W} C_W$ cannot be ramified by local stability of $(f_{C})|_Z$, so it is \'{e}tale. Since any \'{e}tale cover of $C_W$ extends to an \'{e}tale cover of $C$, there exists an \'{e}tale cover $C' \to C$ with the property that any source of $f_{C'} \colon Y_{C'} \to C'$ has connected fibers over the irreducible component that it dominates.

Note that the irreducible components of $Y_{C'}$ are normal in codimension one. Indeed, since $(Y, \Delta)$ is sdlt over $C^{\circ}$, eventual self-intersections in codimension one lie in fibers over closed points. If the branches of the self-intersection map via $f_{C'}$ to distinct branches of a node in $C'$, then $C'$ is nc not snc, which is a contradiction; otherwise, if the branches of the self-intersection dominate a single branch contained in an irreducible component $C_B$ of $C'$, then a fiber of the locally stable morphism $f_{C_B}$ would be non-reduced, which is a contradiction.

By \cite{Hashizume2021}, since the irreducible components of $Y_{C'}$ are normal in codimension one, 
then there exist an sdlt pair $(Y', \Delta')$ and a crepant birational morphism $\pi \colon (Y', \Delta') \to (Y_{C'}, \Delta_{C'})$, which is an isomorphism at all codimension 1 singular point of $Y'$ and $Y_{C'}$, such that 
\begin{equation}
\pi_* (\omega^{[m]}_{Y'/C'}(m\Delta')) \simeq \omega^{[m]}_{Y_1/C'}(m\Delta_{C'}).
\end{equation}
This gives \eqref{item1:sdlt}. Taking pushforward along $f_{C'}$ and by \cite[(2.67.2)]{kbook}, we achieve \eqref{item:pushforward}.
\end{proof}

\subsection{Canonical bundle formula}\label{sec:cbf}
We recall the notion of lc-trivial fibration.
\begin{definition} 
\label{lc-trivial.def}
Let $(Y, \Delta)$ be a sub-pair with coefficients in $\bQ$.
A projective fibration $f \colon Y \rar X$ 
is \emph{lc-trivial} if
\begin{itemize}
    \item[(i)]\label{lc-trivial-cond1} $(Y,\Delta)$ is an lc sub-pair over the 
    generic point of $X$;
    \item[(ii)]\label{lc-trivial-cond2} $\mathrm{rk} f_\ast  \O Y. 
    (\lceil \mathbf{A}^\ast (Y,\Delta)\rceil)=1$;
    \item[(iii)]\label{lc-trivial-cond3} there exists a $\bQ$-Cartier $\bQ$-divisor $L$ on $X$
    such that $\K Y. + \Delta \sim_\bQ f^\ast  
    L$.\footnote{Observe that lc-trivial stands for (relatively) 
    trivial log canonical divisor, i.e., assumption (iii), and not 
    to the type of the singularities in assumption (i).} 
\end{itemize}
\end{definition}

\begin{remark}
Note that property (ii) holds automatically if the general fiber $(Y_{\eta}, \Delta_{\eta})$ of $f$ is a klt pair.
\end{remark}

The canonical bundle formula (\Cref{canonicalbundleformula}) is a broad
term designating a formula for the $\bQ$-divisor
$L$ in (iii),
encoding the log canonical thresholds of the codimension one singularities of $f$ (boundary divisor) and the variation of the general fiber (moduli divisor). We first define these divisors; see also \cite{AmbroPhD, Amb04, Amb05, FG14, Koll07}.

\begin{defn}[Boundary divisor]\label{defn:boundary}
Let $f \colon (Y, \Delta) \to X$ be an lc-trivial fibration. 
The {\emph{boundary divisor}} $B_X$ is the $\bQ$-divisor on $X$ whose coefficient along the prime divisor $D$ is given by
 \[
 \mathrm{ord}_{D}(B_X)=\sup_E\left\lbrace 1-\frac{1+a(E; Y, \Delta)}{\operatorname{mult}_E(\pi^*D)}\right\rbrace,
\]
 where the supremum is taken over all the divisors $E$ over $Y$ which dominate $D$, and $a(E; Y, \Delta)$ is the discrepancy of $E$ with respect to $(X, \Delta)$.
 \end{defn}

 \begin{defn}[Shokurov's moduli divisor]\label{defn:Shokurovmoduli}
Let $f \colon (Y, \Delta) \to X$ be an lc-trivial fibration with general fiber $F$, and $d$ be the minimal positive integer such that $d(K_{F}+ \Delta_F)\sim 0$.
\emph{Shokurov's moduli divisor} $N_X$ is a $\bQ$-divisor on $X$ 
satisfying
the following property: there exists a rational function $\phi \in k(Y)^*$
such that
\[d(K_{Y}+ \Delta) + \mathrm{div}(\phi) = d(f^*(K_X +B_X + N_{X})).\] 
It is uniquely defined up to $d$-linear equivalence: any change in the choice of $K_Y$, $K_X$, and $\phi$ gives a new $N'_{X}$ 
such that $dN_X \sim dN'_X$.
\end{defn}

\begin{remark}\label{rmk:bdivisor}
    Let $f \colon (Y, \Delta) \to X$ be an lc-trivial fibration, and $q \colon X' \to X$ be an alteration. Let $Y'$ be the normalization of the main component of $Y \times_{X} X'$, and $(Y', \Delta')$ be the log pullback of $(Y, \Delta)$.
    \begin{itemize}
    \item If $q$ is birational, then $f' \colon (Y', \Delta') \to X'$ is an lc-trivial fibration with $q_* B_{X'}=B_{X}$ and $q_* N_{X'} = N_{X}$. Therefore, the boundary divisor and Shokurov's moduli part define b-divisors $\mathbf{B}$ and $\bfN$. Note that $\mathbf{K} + \mathbf{B}$ and $\bfN$ are $\bQ$-Cartier b-divisor; see \cite[Thm.~2.7]{Amb04}.
    \item If $q$ is an arbitrary alteration, then $\bfN' = g^*\bfN$ as b-divisors;
    see \cite[Prop.~5.5]{Amb04}.
    \end{itemize}
\end{remark}

 \begin{constdef}[Hodge-theoretic moduli divisor]\label{defn:modulipart}  Let $f \colon (Y, \Delta) \to X$ be an
 lc-trivial fibration of relative dimension $n$. Write $\Delta$ as  difference of effective divisors without common components, namely
 \[
 \Delta = E + F - G \quad \text{with } \quad E \coloneqq \Delta^{=1}, G \coloneqq \lceil \Delta^{< 0}\rceil
 \] 
 so that $F$ is the fractional part of $\Delta$ satisfying $\lfloor F \rfloor =0$. Let $d$ be the minimal positive integer such that 
 $d(K_{Y_{\eta_X}}+ \Delta_{Y_{\eta_X}})\sim 0$, where $Y_{\eta}$ denotes the generic fiber of $f$. 
 \begin{enumerate}
 \item[($\dagger$)]\label{eq:modulipartI} Suppose that there exists an snc pair $(X, D)$ such that $(Y^{\circ}, \Delta^{\circ})$ is a topologically locally trivial snc pair over $X^{\circ} \coloneqq X \setminus D$ with $d(K_{Y^{\circ}} + \Delta^{\circ}) \sim 0$, where the superscript ${}^{\circ}$ refers to the restriction of an object of interest over $X^{\circ}$.
 \end{enumerate}
Consider $L \coloneqq \mathcal{O}_{Y^{\circ}}(G^{\circ}-K_{Y^{\circ}}-E^{\circ})$. An isomorphism $L^{d} \simeq \mathcal{O}_{Y^{\circ}}(dF^{\circ})$
determines a normalized cyclic cover $a \colon Y^{\circ}_2 \to Y^{\circ}$ of degree $d$ branched along $F^{\circ}$, with quotient singularities; see \cite[2.49-53]{KM98}. 
Choose a $\mu_d$-equivariant resolution of singularities $h \colon Y^{\circ}_3 \to {Y}^{\circ}_2$. Write $f_2 \coloneqq f \circ a$ and $f_3 \coloneqq f_2 \circ h$, $E_2 \coloneqq \Supp (a^{-1} E)$, $E^{\circ}_3 \coloneqq \Supp (h^{-1} {E}^{\circ}_2)$. 

\[
\xymatrix{
(Y^{\circ}, \Delta^{\circ}) \ar[d]_-{f} & 
Y^{\circ}_2 \ar[l]_-{a} \ar[dl]_-{f_2}&
Y^{\circ}_3 \ar[l]_-h \ar[dll]^-{f_3}\\
X^{\circ}. &
&
}
\]
The sheaf $(a \circ h)_* (\omega_{Y^{\circ}_3}(E^{\circ}_{3}))$ is $\mu_d$-equivariant,
and admits a decomposition into $\mu_d$-isotypic components  
 \[
 (a \circ h)_* (\omega_{Y^{\circ}_3}(E^{\circ}_{3})) \simeq a_* (\omega_{Y^{\circ}_2}(E^{\circ}_{2})) \simeq \bigoplus^{d-1}_{i=0} a_* (\omega_{Y^{\circ}_2}(E^{\circ}_{2}))_{\chi^{i}} \simeq \bigoplus^{d-1}_{i=0} \omega_{Y^{\circ}}(\lceil i (L- F^{\circ}) \rceil),
 \]
 where $\chi$ is a generator of the character group $\hat{\mu}_d$. In particular, we get the direct summand
\begin{equation}\label{eq:pushformoduli}
(a \circ h)_* (\omega_{Y^{\circ}_3}(E^{\circ}_{3}))_{\chi} \simeq \omega_{Y^{\circ}} \otimes L \simeq \mathcal{O}_{Y^{\circ}}(G^{\circ}-E^{\circ}).
\end{equation}
The $\chi$-isotypic component of the restriction over $X^{\circ}$ of $R^{n}(f_3)_*\bC_{Y^{\circ}_{3} \setminus E^{\circ}_{3}}$ determines a complex CY variation of mixed\footnote{pure if $E^{\circ}=0$, i.e., $(Y, \Delta)$ is klt over the generic point of $X$.} Hodge structures whose deepest nonzero piece of the Hodge filtration is the line bundle $f_* \mathcal{O}_{Y^{\circ}}(G^{\circ}-E^{\circ})$; see \eqref{eq:pushformoduli} and \Cref{lc-trivial.def}.(ii). 
\begin{enumerate}[resume]
\item[($\dagger \dagger$)] \label{eq:modulipartII}  Suppose that $R^{n}(f_3)_*\bZ_{Y^{\circ}_{3} \setminus E^{\circ}_{3}}$ has local unipotent monodromy.
\end{enumerate}
 Let $V_{Y}=R^{n}(f_3)_*\bZ_{Y^{\circ}_{3} \setminus E^{\circ}_{3}}$ defined on $X^{\circ}$.  Then the moduli part $M_X$ is the Schmid extension of the Hodge bundle of the $\chi$-isotypic part of $V_Y$:
 \[
  M_{X} \coloneqq F^m(\cV_{Y})_{\chi}= F^m(\cV^{\trans}_{Y})_{\chi}.  
\]
By shifting the Hodge filtration of $V_{Y}$, the moduli part $M_{X}$ can be realized as the Schmid extension of the Hodge bundle of an admissible graded polarizable integral (not just complex!) CY variation $V'_{Y}$ of mixed Hodge structures with the same underlying local system $V_{Y,\bZ}$; see \Cref{rmk:rationalCY}.

If conditions ($\dagger$) and ($\dagger \dagger$) are not satisfied, then there exists a projective alteration $q \colon W \to X$ such that the pullback of the generic fiber of $f$ along $q$ extends to a fibration $f' \colon (Y', \Delta') \to W$ satisfying conditions ($\dagger$) and ($\dagger \dagger$) (cf. \Cref{lem:alteration}). Set the $\bQ$-divisor
\[
M_{X} \coloneqq \frac{1}{\deg(q)} q_* (M_{W}).
\]
\end{constdef}

\begin{example} 
    If $f \colon Y \to X$ is a family of smooth CY varieties of dimension $n$, then $M_{X} = f_* \omega_{Y/X}$ is simply the Hodge bundle of the variation of Hodge structures $R^nf_*\bZ_Y$. 
\end{example}

\begin{remark}\label{functoriality}
By the functoriality of Deligne/Schmid extension, $\bfM(f)$ pulls back to $\bfM (f')$ as a  b-divisor, which descends on $W$ because of the snc assumptions in Construction-Definition \ref{eq:modulipartI}; see \cite[\S 8.4.8]{Koll07}.
\end{remark}

\begin{remark}\label{rmk:rationalCY} 
The variation
$V_Y=(V_{Y,\bZ},W_\bullet V_{Y,\bQ},F^\bullet V_{Y,\cO})$ is a graded polarizable integral mixed variation of Hodge structures of Hodge-level $n$ but not CY in general, while the variation $(V_{Y,\bC})_{\chi}$ is a complex CY variation, not necessarily integral if $d>2$. Up to shifting the Hodge filtration of the $\mu_d$-isotypic components, there exists a polarizable integral pure CY variation of Hodge-level $n+4$ whose deepest piece of the Hodge filtration is $M_X$. Choose for instance 
\[V_{Y}'=(V_Y)_{\chi}(-2,2)\oplus \bigoplus_{i\not\equiv 1,-1 (d)} (V_Y)_{\chi^i}\oplus (V_Y)_{\chi^{-1}}(2,-2).\]
In particular, note that the period map of $V_{Y}$ is generically injective if and only if so is the period map of $V_{Y}'$. 
\end{remark}

\begin{remark}[]\label{rmk:moduli}
Since $Y^{\circ}_2$ has quotient singularities (hence it is a rational homology manifold), the constant sheaf $\bC_{Y^{\circ}_2}$ is a direct summand of $Rh_* \bC_{Y^{\circ}_3}$. Its ($\mu_d$-equivariant) pushforward along $a$ decomposes in isotypic components as follows
\begin{equation}\label{eq:splitIC}
a_*\bC_{Y^{\circ}_2} = \bigoplus^{d-1}_{i=0} \iota_*L_{\chi^{i}},
\end{equation}
where $L_{\chi^{i}}$ are suitable local systems on $\iota \colon Y^{\circ} \setminus F^{\circ} \hookrightarrow Y^{\circ}$, on which $\mu_d$ acts via the character $\chi^i$. 
Then there exist isomorphisms of complex  variations of Hodge structures
\begin{equation}\label{eq:VHSmoduli}
((R^{n}(f_3)_*\bZ_{Y^{\circ}_{3} \setminus E^{\circ}_{3}})^{\rm tr} \otimes \bC)_{\chi} 
\simeq 
((R^{n}(f_2)_*\bZ_{Y^{\circ}_{2} \setminus E^{\circ}_{2}})^{\rm tr}\otimes \bC)_{\chi} 
\end{equation}
which yield equivalent alternative definitions of the moduli part. In particular, \eqref{eq:VHSmoduli} is a complex subvariation of Hodge structures of $R^n f_* (\iota_*L_{\chi})$ containing $M_X$. 
Compare the previous chain of isomorphisms with the different period maps considered in \cite{Amb05}. This also fixes a minor inaccuracy in the definition of the variation of Hodge structures whose bottom piece extends to define the moduli part appearing in \cite[Def. (8.4.6)]{Koll07}.
\end{remark}

\Cref{canonicalbundleformula} asserts that 
Shokurov's moduli divisor and the Hodge-theoretic moduli divisor are linearly equivalent up to a bounded multiple. The proof of positivity results for $\mathbf{N}$ usually relies on the $\bQ$-linear equivalence with $\mathbf{M}$ established in \Cref{canonicalbundleformula}; see also \Cref{rmk:effectiveb-semiamplenesMN}.  The proof builds on and generalize \cite{Amb04, Amb05, Fuj03, FM00, Koll07, Mor87} and works for projective morphisms of complex analytic spaces.

\begin{theorem}[Canonical bundle formula] \label{canonicalbundleformula}
Let $f \colon (Y, \Delta) \to X$ be an lc-trivial fibration.
Then 
\[K_Y+\Delta \sim_{\bQ} f^\ast (K_X + B_X + M_X).\]
Moreover, there exists a linear equivalence $c M_{X} \sim c N_{X}$, where $c = 3^{r}$ and $r \coloneqq \rank R^{n}(f_3)_*\bQ_{Y^{\circ}_{3} \setminus E^{\circ}_{3}}$ (cf. \Cref{defn:modulipart}), which implies 
\[c \cdot d \cdot (K_Y+\Delta) \sim c \cdot d \cdot (f^\ast (K_X + B_X + M_X)).\]
\end{theorem}

\begin{remark} Since $\bfM$ is b-nef, the triple $(X,B_X, \bfM)$ has the structure
of a generalized sub-pair (cf. \cite{BZ16}, \cite{FS20}). We say that $f \colon (Y, \Delta) \to X$ induces $(X,B_X, \bfM)$.  When $(Y,\Delta)$ is a 
klt (resp.\ lc) pair, then 
$(X,B_X, \bfM)$ is a klt (resp.\ lc) generalized pair; see, e.g., \cite{Fil20}.
\end{remark}

\begin{remark}\label{rmk:effectiveb-semiamplenesMN}
    For instance, the (effective) b-semiampleness of $\mathbf{N}$ is equivalent to the (effective) b-semiampleness of $\mathbf{M}$.
\end{remark}

\begin{proof} Let $q \colon X_1 \to X$ be a generically finite morphism, $Y_1$ be a resolution of the main component of $Y \times_{X} X_1$ and $(Y_1, \Delta_1)$ be the log pullback of $(Y, \Delta)$ such that $f_1 \colon (Y_1, \Delta_1) \to X_1$ satisfies properties ($\dagger$) and ($\dagger \dagger$).
Recall that $d$ is the minimal positive integer such that $d(K_{Y_{1,\eta}}+ \Delta_{Y_{1,\eta}})\sim 0$, where $Y_{1,\eta}$ denotes the generic fiber of $f_1$.
Let $\phi \in k(Y_1)^*$
be a rational function such that 
\[d(K_{Y_1}+ \Delta_1) + \mathrm{div}(\phi) = d(f_1^*(K_{X_1} +B_{X_1} + N_{X_1})).\] 
Write $\Delta_1-f_1^* B_{X_1}= E + V + F - G$, where $E + V \coloneqq \lfloor (\Delta_1-f^*_1 B_{X_1})^{>0} \rfloor$,
$G \coloneqq \lceil (\Delta_1-f^*_1 B_{X_1})^{< 0}\rceil$, $F$ is the fractional part of $\Delta_1-f^*_1 B_{X_1}$ satisfying $\lfloor F \rfloor =0$, $E$ is horizontal and $V$ is vertical. 
Let $a \colon Y_2 \to Y_1$ be the normalization of $Y_1$ in $k(Y_1)[\sqrt[d]{\phi}]$
with Galois group $\mu_d$. Choose a $\mu_d$-equivariant resolution of singularities $h \colon Y_3 \to {Y}_2$.
Write $f_2 \coloneqq f \circ a$ and $f_3 \coloneqq f_2 \circ h$, $E_2 \coloneqq \Supp (a^{-1} E)$, $E_3 \coloneqq \Supp (h^{-1} {E}_2)$.
\[
\xymatrix{
(Y, \Delta) \ar[d]_-{f} & (Y_1, \Delta_1) \ar[d]_-{f_1} \ar[l] & 
Y_2 \ar[l]_-{a} \ar[dl]_-{f_2}&
Y_3 \ar[l]_-h \ar[dll]^-{f_3}\\
X & X_1. \ar[l]_-q &
&
}
\]
We proceed in several steps.

\textbf{Step 1}. \emph{In this step we show that, if $N_{X_1}$ is integral, then $M_{X_1} \sim N_{X_1}$.}\\
\noindent
The same argument giving \eqref{eq:pushformoduli} yields 
\[
(a \circ h)_* \omega_{Y_3/X_1}(E_3)_{\chi} \simeq \mathcal{O}_{Y_1}(-V+G + \lceil f_1^*N_{X_1} \rceil).
\]
Under the assumption that $N_{X_1}$ is integral, and since the line bundle $(f_{1})_*(\mathcal{O}_{Y_1}(-V+G))$ is trivial by \cite[Step 5 in the proof of Thm.~8.5.1]{Koll07}, we obtain that
\begin{align*}
   \mathcal{O}_{X_1}(M_{X_1}) \simeq F^m(V^{\trans}_{Y, \cO})_{\chi} & \simeq (f_{3})_* \omega_{Y_3/X_1}(E_3)_{\chi} \simeq (f_{1})_*(\mathcal{O}_{Y_1}(-V+G + f_1^*N_{X_1})) \\
   & \simeq (f_{1})_*(\mathcal{O}_{Y_1}(-V+G )) \otimes \mathcal{O}_{X_1}(f_1^* N_{X_1}) \simeq \mathcal{O}_{X_1}(N_{X_1}),
\end{align*}
where the second isomorphism follows by \cite[\S 4]{K1981} if $E_3=0$ and in general by \cite[Thm. 7.1]{FF2014}.

    \textbf{Step 2}. 
    \emph{In this step, we show that $N_{X_1}$ is integral.}\\
    \noindent
    Let $\pi \colon X'_1 \to X_1$ be a Galois cover with Galois group $\mathbb{G}$ such that $f'_3 \colon (Y'_3, E'_3) \to X'_3$ is a locally stable morphism for a primed analogue of the set-up above. By \cite[Prop.~5.5, Lem. 5.2.(5)]{Amb04}, $\pi^* N_{X_1} \sim N_{X'_{1}}$ and $N_{X'_{1}}$ is an integral divisor. By the functoriality of Deligne/Schmid extensions for local systems with local unipotent monodromy, $\pi^* M_{X_1} \sim M_{X'_{1}}$. By Step 1, there exists an isomorphism $\mathcal{O}_{X'_{1}}(M_{X'_1}) \simeq \mathcal{O}_{X'_{1}}(N_{X'_{1}})$, which is Galois invariant by applying again \cite[Prop.~5.5]{Amb04} and the functoriality of Deligne/Schmid extensions to the deck transformations of the Galois cover $\pi$. Thus we obtain $\mathcal{O}_{X_{1}}(M_{X_1}) \simeq \mathcal{O}_{X'_{1}}(M_{X'_1})^{\mathbb{G}} \simeq \mathcal{O}_{X'_{1}}(N_{X'_{1}})^{\mathbb{G}} \simeq \mathcal{O}_{X_{1}}(\lfloor N_{X_{1}} \rfloor)$. We conclude that $\pi^*(N_{X_1}-\lfloor N_{X_1}\rfloor)\sim N_{X'_1}-M_{X'_1}\sim 0$, i.e., $N_{X_1}$ is integral.

    \textbf{Step 3}. 
    \emph{In this step, we conclude the proof by bounding the degree of $q \colon X_1 \to X$.}\\
    \noindent
    By Steps 1 and 2, $N_{X_1}$ is integral and $M_{X_1} \sim N_{X_1}$, so that $\deg(q) M_{X} \sim \deg(q) N_{X}$. Since a level three structure is sufficient to make the monodromy of $R^{n}(f_3)_*\bZ_{Y^{\circ}_{3} \setminus E^{\circ}_{3}}$ unipotent \cite[p.~253, p.~207 Lem.]{Mum70}, we can suppose that $\deg(q)$ divides $c=3^{r}$, so that $cM_{X} \sim c N_{X}$ as desired. 
\end{proof}

\begin{remark}\label{rmk:analcanonicalbundle}
The canonical bundle formula (\Cref{canonicalbundleformula}) continues to hold for (projective) lc-trivial fibrations $f \colon (Y, \Delta) \to X$ of complex analytic spaces. Indeed, resolutions of singularities needed to achieve $(\dagger)$ and the cyclic covers in \cite[2.49-53]{KM98} can be performed in the analytic category too. To achieve ($\dagger \dagger$), we can take a suitable \'{e}tale cover of $X \setminus D$ making the local monodromy of the relevant local system unipotent and then extend it to a finite cover (with quotient singularities) by the Grauert--Remmert Extension Theorem \cite[Ch.~XII, Thm.~5.4,
p.~340]{SGA1}. The proof of \Cref{canonicalbundleformula} relies on weakly semistable reduction, which holds in this setting too by \cite[Thm 1.17]{HaconPaun}. Hence, \Cref{rmk:bdivisor} and the proof of \Cref{canonicalbundleformula} work verbatim in this analytic context too. It is unclear whether the projectivity of $f$ can be weaken to the assumption that $f$ is a K\"{a}hler morphism; cf. also \Cref{thm:bsemiampleanal}. Note also that our proofs of \Cref{thm:Bsemimain} and \Cref{thm:bsemiampleanal} require the projectivity of the morphism.
\end{remark}

\subsection{Variation of Hodge structures and source of a degeneration of CY pairs}\label{sec:variation} The main result of this section is \Cref{thm:0.19}: it identifies the transcendental part of the cohomology of the source over a divisor with that of the limiting mixed Hodge structure of a family of (log) Calabi--Yau in the punctured neighbourhood of the divisor. We first define the relevant variations of Hodge structures. To this end, we extend the \Cref{defn:modulipart} from the generic point of the base $X$ through the generic point of a divisor $D_X \subset X$.

\begin{const}\label{const:modulioverdivisor}
Let $f \colon (Y, \Delta) \to X$ be a projective fibration 
whose generic fiber $(Y_{\eta_X}, \Delta_{\eta_X})$ is klt
log Calabi--Yau of dimension $n$. Fix a smooth integral divisor $D_X \subset X$. 
Up to shrinking $X$ around the generic point of $D_X$,
by \Cref{lem:alteration}, there exist
\begin{enumerate}
\item a projective alteration $q \colon W \to X$;
\item an lc-trivial locally stable fibration $f' \colon (Y', \Delta') \to W$, topologically locally trivial over $D \coloneqq D_W$ and $W \setminus D$, such that the pullback of the generic fiber of $f$ along $q$ is crepant birational to the generic fiber of $f'$; and
\item $(Y', \Delta' + Y'_{D})$ is a dlt pair with $\lfloor \Delta' + Y'_{D} \rfloor =Y'_{D}$.
\end{enumerate}

Let $g \colon (Y_1, \Delta_1) \to (Y', \Delta' + Y'_{D})$ be a log resolution of $(Y', \Delta' + Y'_{D})$ with log pullback $(Y_1, \Delta_1)$ 
such that $g$ is an isomorphism over the snc locus of $(Y', \Delta' + Y'_{D})$; see \cite[Thm.~10.45]{kol13}.
Then, write 
 \[
 \Delta_1 \coloneqq E_1 + F_1 - G_1 \quad \text{with } \quad E_1 \coloneqq 
 \Delta^{=1}_{1}, G_1 \coloneqq \lceil \Delta^{<0}_{1} \rceil.
 \]
Since the generic fiber of $f$ is klt, $E_1$ lies over $D$. Let $d$ be the minimal positive integer
such that $d F_1$ is an integral divisor and $d(K_{Y_1} +\Delta_1) \sim_{f_1} 0$.
Consider $L_1 \coloneqq \mathcal{O}_{Y_1}(G_1-K_{Y_1}-E_1)$.
The isomorphism
 $L^{d}_1 \simeq \mathcal{O}_{Y_1}(dF_1)$
 determines a normalized cyclic cover $a \colon Y_2 \to Y_1$, with Galois group $\mu_d$, branched along $F_1$; see \cite[2.49-53]{KM98}. Since $a$ is a cyclic cover branched along an snc divisor, $Y_2$ and all strata of $Y_{2,D}$ have quotient singularities. Let $h \colon (Y_3, (Y_{3})_{D})  \to (Y_2, (Y_{2})_{D})$ be a $\mu_d$-equivariant log resolution, and set $f_2 \coloneqq f_1 \circ a$ and $f_3 \coloneqq f_2 \circ h$. Generically, the cover $a$ is one of the covers obtained by applying \Cref{defn:modulipart} to $f_1 \colon Y_1 \to W$. In particular, $f_2$ and $f_3$ are fibrations (with connected fibers).
 To summarize, we collect the introduced
 maps
 over $W$ in the following diagram
\[
\xymatrix{
(Y', \Delta') \ar[d]_-{f'}& 
(Y_1, \Delta_1) \ar[l]_-{g} \ar[dl]_-{f_1} & 
Y_2 \ar[l]_-{a} \ar[dll]_-{f_2}&
Y_3 \ar[l]_-h \ar[dlll]^-{f_3}\\
W &
&
&
}
\]
where $g$ is crepant birational, $a$ is a cyclic cover, and $h$ is birational. 

Let $(S, \Delta_S)$ be a source of $f' \colon (Y'_{D}, \Delta_{Y'_{D}}) \to D$. 
Let $S_{i}$ be a stratum of $Y_{i, D}$, generically finite onto the source $S \subset Y'$. Generically, the restriction $a \colon S_2 \to S_1$ is again one of the covers obtained by applying \Cref{defn:modulipart} to $f_1 \colon S_1 \to D$. 
Up to shrinking $X$ again, and replacing $W$ with an \'{e}tale cover obtained by spreading out the Stein factorization of $f'|_{S}$,
we can suppose that
\begin{enumerate}[resume]
\item  $f_{S} \coloneqq f'|_{S} \colon (S, \Delta_S) \to D$ is an lc-trivial fibration (in particular with connected fibers) of relative dimension $m$.
 \end{enumerate}

 By \Cref{defn:modulipart}, the moduli part of the lc-trivial fibration $f' \colon (Y', \Delta') \to W$ is the Hodge bundle of the $\chi$-isotypic component of the Schmid extension of $R^n (f_2)_* \bC_{Y^\circ_2}$ (or equivalently of $R^n (f_3)_* \bC_{Y^\circ_3}$ as in \Cref{rmk:moduli}), i.e., \[M_{W} \coloneqq F^n(\mathcal{R}^{n}(f_2)_*\bC_{Y^{\circ}_{2}})_{\chi} \simeq F^n(\mathcal{R}^{n}(f_3)_*\bC_{Y^{\circ}_{3}})_{\chi}\simeq ((f_3)_* \omega_{Y_3/W})_{\chi}.\]
 \end{const}

\begin{notation}\label{notn:subD}
In the following, we denote by $V_{Y,D}$ (resp. $V_Y^D$) the $\mu_d$-equivariant mixed variation on $D$ obtained by taking local monodromy covariants (resp. invariants) of $V_Y=R^{n}(f_3)_*\bZ_{Y^{\circ}_{3}}$ as in \S2.  We denote  $V_S=R^m(f_3)_{*}\bZ_{S^{\circ}_3}$, also as a $\mu_d$-equivariant variation. Given a $\mu_d$-equivariant integral or rational Calabi--Yau variation of Hodge structures $V$ and a $\mu_d$-character $\chi$, $V^{\mathrm{tr}}_{\chi}$ denotes the $\chi$-isotypic component of the complexification of the transcendental part of $V$. 
\end{notation}

\begin{thm}\label{thm:0.19}
In the notation above, there exists an isomorphism of variations of Hodge structures
\[
(V_{Y, D})^{\rm tr}_{\chi} \simeq (V_S)^{\rm tr}_{\chi}.
\]
\end{thm}
\begin{proof} The required isomorphism is obtained by composing the isomorphisms \eqref{eq:trmin}, \eqref{eq:minmin}, and \eqref{eq:lastiso}.

\textbf{Step 1.}
Since $V_Y$ is a CY variation of Hodge structure, there exists a unique integer $k$ such that $\gr^n_F \gr^W_{n+k} (V_{Y, D, \bC})_{\chi} \neq 0$. Recall that the logarithmic monodromy $N$ defining the weight filtration has type $(-1, -1)$.
Then the isomorphism $N^{k} \colon \gr^W_{n+k} (V_{Y, D, \bQ})_{\chi} \to \gr^W_{n-k} (V_{Y, \bQ}^D)_{\chi}$  induces an isomorphism
\[F^{n}(V_{Y,D,\bC})_{\chi} \simeq (\gr^W_{n+k} (V_{Y,D,\bC})_{\chi})^{n, k} \xrightarrow[N^{k}]{\simeq}  (\gr^W_{n-k} (V_{Y,\bC}^D)_{\chi})^{n-k,0} \xrightarrow[\mathrm{conj}.]{\simeq} \overline{(\gr^W_{n-k} (V_{Y,\bC}^D)_{\chi})^{0, n-k}} \simeq \overline{\gr^{0}_F(V_{Y,\bC}^D)_{\chi}}.\]
In particular, $N^{k}(V_{Y,D})^{\rm tr}_{\chi}$ is the $\chi$-isotypic component of the unique minimal subvariation of rational Hodge structure in $\gr^W_{n-k} (V^D_{Y, \bQ})$ containing
$\gr^0_F (V_{Y, \bC}^D)$, denoted $(V_{Y}^D)^{\trcon}_{\chi}$, i.e., 
\begin{equation}\label{eq:trmin} 
(V_{Y,D})^{\rm tr}_{\chi} \simeq (V_Y^D)^{\trcon}_{\chi}.
\end{equation}

\textbf{Step 2.} Let $\psi_D$ be the vanishing cycle functor associated to a global function defining $Y_{3, D}$, which exists up to further shrinking $W$. 
The specialization morphism 
$\bQ_{Y_{3, D}}
\to
\psi_D \bQ_{Y_3}$
between Hodge modules on $Y_{3, D}$ (cf. \cite[(2.24.3)]{SaitoMHM} or \cite[Thm.~11.29]{peterssteenbrink})
induces a $\mu_d$-equivariant morphism of polarizable variation of mixed Hodge structures
\begin{equation}\label{eq:spmapII}
R^n (f_3)_* \bQ_{Y_{3, D}}
\to
R^n \psi_D \bQ_{Y_3} \hookleftarrow  V^D_{Y, \bQ},
\end{equation}
where the last inclusion follows, e.g., by the remark right before \cite[(5.11) Cor.]{steenbrinklimit}.
The morphisms in \eqref{eq:spmapII} induce an isomorphism \[\gr^0_F R^n (f_3)_* \bC_{Y_{3, D}} \simeq \gr^0_F (V_{Y, \bC}^D)_{\chi},\] which entails the isomorphism of variation of pure Hodge structures
\begin{equation}\label{eq:minmin}
(V_{Y}^{D})^{\trcon} _{\chi}\simeq (R^n (f_3)_* \bQ_{Y_{3, D}})^{\trcon}_{\chi} \simeq (R^n (f_2)_* \bQ_{Y_{2, D}})^{\trcon}_{\chi};
\end{equation}
see \cite[Cor. 5.7, Proof of Prop.~9.2]{KLS2021}. 

\textbf{Step 3.} Write $Y_{2, D} \coloneqq R = \bigcup_{i \in I} R_{i}$ as a union of its irreducible components, and fix an ordering of $I$. Denote by $R^{[k]}$ the disjoint union of the strata that have codimension $k$ in $R$. The Mayer--Vietoris complex
of $\bQ_{R}$ associated to the closed cover $\{ R_{i}\}_{i \in I}$
\begin{equation}\label{eq:resolution}
\bQ_{R^{\bullet}} \colon \bQ_{R^{[0]}} \to \bQ_{R^{[1]}} \to \bQ_{R^{[2]}} \to \dots
\end{equation}
is a $\mu_d$-equivariant resolution of $\bQ_{R}$. Recall that the differential of the complex is induced by the natural restriction $\bQ_{R_J} \to \bQ_{R_{J \cup \{j\}}}$, with a
plus or a minus sign according to the parity of the position of $j$ in $J \cup \{j\}$; cf. \cite[App.~A]{KLSV18}.
There exists a $\mu_d$-equivariant spectral sequence
\begin{equation}\label{eq:spsequenceQ}
E^{p,q}_1 = R^{p} (f_{2})_*\bQ_{R^{[q]}} \implies R^{p+q}(f_2)_*\bQ_{R}.
\end{equation}
Since the differentials of the spectral sequence 
are morphisms of variations of mixed Hodge structures and all strata of $R$ are proper over $D$  with quotient singularities (hence they are rational homology manifolds), the spectral sequence \eqref{eq:spsequenceQ} induces the $\mu_d$-equivariant spectral sequence 
\[
\gr_F^0 E^{p,q}_1 =  \gr_F^0 R^{p} (f_{2})_*\bC_{R^{[q]}} = \gr_F^0 
\gr^W_{p} 
R^{p} (f_{2})_*\bC_{R^{[q]}} 
\implies \gr^0_F R^{p+q}(f_2)_*\bC_{R}, 
\]
which abuts at the $E_2$ page for weight reasons. Together with \Cref{lem:cokerd1}, we obtain
\[
\gr_F^0(R^n(f_2)_*\bC_{R})_{\chi} \simeq \gr_F^0 E^{n}_{\infty, \chi} \simeq \gr_F^0 E^{n}_{2, \chi} \twoheadrightarrow \gr_F^0 E^{m,n-m}_{2, \chi} \twoheadrightarrow \gr_F^0(R^m(f_2)_*\bC_{S_2})_{\chi}.
\]
Since both $\gr_F^0(R^n(f_2)_*\bC_{R})_{\chi}$ and $\gr_F^0(R^m(f_2)_*\bC_{S_2})_{\chi}$ are line bundles, the composition is an isomorphism:
\[
\gr_F^0(R^n(f_2)_*\bC_{R})_{\chi} \simeq \gr_F^0(R^m(f_2)_*\bC_{S_2})_{\chi}
\]
(as an aside, this also shows that $k=m$). Therefore, we conclude
\begin{equation}\label{eq:lastiso}
(R^n (f_2)_* \bC_{Y_{2, D}})^{\trcon}_{\chi} \simeq (R^m (f_2)_* \bC_{S_{2}})^{\trcon}_{\chi} = (V_S)_{\chi}^{\rm tr},
\end{equation}
where the last equality follows again from the fact that $S_2$ has quotient singularities. 
\end{proof}

We prove the lemma used in the proof of \Cref{thm:0.19}. Recall that $n$ (resp.~$m$) is the relative dimension of the morphisms $Y_i \to W$ (resp.~$S_i \to D$). 

\begin{lemma}\label{lem:cokerd1} In the notation of Step 3 of the proof of \Cref{thm:0.19}, we have
\[
\coker(d_1 \colon \gr_F^0 E^{m, n-m-1}_{1,\chi} \to \gr_F^0 E^{m,n-m}_{1,\chi}) \twoheadrightarrow \gr_F^0(R^m(f_2)_*\bC_{S_2})_{\chi}.\]
\end{lemma}
\begin{proof}
Since all strata of $R$ have quotient (hence du Bois) singularities, we write
\[\gr_F^0 E^{p,q}_{1,\chi} \simeq (R^{p} (f_{2})_*\cO_{R^{[q]}})_{\chi} \simeq R^{p} (f_{1})_*\cO_{a(R^{[q]})}(-L_1).\] 
Write $a(R^{[q]}) = LCC^{[q]} \cup a(R^{[q]})'$,
where $LCC^{[q]}$ is the disjoint union of the lc centers of $(Y_1, \Delta_1)$ of dimension $n-q$, and $a(R^{[q]})'$ is the disjoint union of the residual $(n-q)$-dimensional strata of $Y_{1, D}$ that are not lc centers.
For brevity, set $A^{p, q} \coloneqq R^{p} (f_{1})_*\cO_{LCC^{[q]}}(-L_1)$ and $B^{p, q} \coloneqq R^{p} (f_{1})_*\cO_{a(R^{[q]})'}(-L_1)$.
Since the only strata of $a(R^{[n-m-1]})$ containing a minimal lc center are lc centers,  the differential
\[d_1 \colon A^{m, n-m-1} \oplus B^{m, n-m-1} \to A^{m, n-m} \oplus B^{m, n-m} 
\] 
is lower triangular; cf.~\eqref{eq:resolution}.
We determine the upper block of $d_1$, namely
\[d_A \coloneqq pr_{A^{m, n-m}} \circ d_1 \circ i_{A^{m, n-m-1}}  \colon 
 A^{m, n-m-1} \to A^{m, n-m},
\]
where $pr$ and $i$ denote the natural projections and inclusions.

To this end, let $Z_1 \subset Y_1$ be an irreducible component of $LCC^{[n-m-1]}$. The trace of $E_1$ on $Z_1$,
denoted by $E_{Z_1}$, is the restriction to $Z_1$ of the components of $E_1$ not containing $Z_1$.
By \Cref{const:modulioverdivisor}, the log resolution $g \colon Y_1 \to Y'$ is an isomorphism over the snc locus of the dlt pair $(Y', \Delta' + Y'_{D})$, hence at the generic point of $Z_1$. By \cite[(4.6),(4.7.1), Thm.~4.19]{kol13}, the induced map
\[
g \colon (Z_1, E_{Z_1} +(F_1-G_1)|_{Z_1}) \to (Z', \Diff^*_{Z'}(\Delta' + Y'_{D}))\coloneqq (g(Z_1), g_*(E_{Z_1} +(F_1-G_1)|_{Z_1}))
\]
is crepant birational, and $(Z', \Diff^*_{Z'}(\Delta' + Y'_{D}))$ is an (effective) dlt pair. In particular, $g$ must contract $G_1|_{Z_1}$ and maps $E_{Z_1}$ birationally onto $E_{Z'} \coloneqq \Diff^*_{Z'}(\Delta' + Y'_{D})^{=1}$. We obtain
\[
Rg_*\cO_{Z_1}(G_1|_{Z_1}-E_{Z_1}) \simeq g_*\cO_{Z_1}(G_1|_{Z_1}-E_{Z_1}) \simeq \cO_{Z'}(-E_{Z'}),
\]
where the first isomorphism follows from \cite[Cor.~10.38.(1)]{kol13} since $G_1|_{Z_1}-E_{Z_1} \sim_{\bQ, g} K_{Z_1} + F_{1}|_{Z_1}$, and the second isomorphism follows by the normality of $Z'$ and $Z_1$ and since $G_1|_{Z_1} -E_{Z_1}+g^*E_{Z'}$ is effective and $g$-exceptional.
Pushing forward along $f_1$
and using relative duality,
we obtain
\[
(R^m(f_1)_* \mathcal{O}_{Z_1}(-L_1))^{\vee} \simeq R^1(f_1)_* (\omega_{{Z_1}/D} \otimes L_1) \simeq (R^1(f_1)_* \mathcal{O}_{Z_1}(G_1|_{Z_1} - E_{Z_1}))\otimes \omega^{-1}_D \simeq (R^1 f'_* \mathcal{O}_{Z'}(-E_{Z'}))\otimes \omega^{-1}_D.
\]
Note that 
\begin{itemize}
\item $E_{Z'}$ has either one or two irreducible components by \cite[Prop.~4.37]{kol13}; and
\item $R^1 f'_* \cO_{Z'} \hookrightarrow R^1 f'_* \cO_{E_{Z'}}$ by \cite[Lem.~3.2]{BFPT2024}.
\end{itemize}
Pushing forward along $f'$ the short exact sequence $0 \to \mathcal{O}_{Z'}(-E_{Z'}) \to \mathcal{O}_{Z'} \to \mathcal{O}_{E_{Z'}} \to 0$, we obtain
\[
R^1 f'_* \mathcal{O}_{Z'}(-E_{Z'}) \simeq \begin{cases}
R^0 f'_* \cO_{S} & \text{ if }E_{Z'} = S \sqcup S',\\
0 &\text{ if }E_{Z'}\text{ is connected.}
\end{cases}
\]
 Restricting to the source $S_1$, the same argument gives
\[(R^m(f_1)_* \mathcal{O}_{S_1}(-L_1))^{\vee} \simeq R^0(f_1)_* (\omega_{{S_1}/D} \otimes L_1) \simeq (R^0(f_1)_* \mathcal{O}_{S_1}(G_1|_{S_1}))\otimes \omega^{-1}_D \simeq R^0 f'_* \mathcal{O}_{S}\otimes \omega^{-1}_D.\] 
To summarize, if $Z_1$ is an lc center containing
two 
distinct
minimal lc centers $S_1$ and $S'_1$, then 
\begin{equation}\label{eq:graph}
R^m(f_1)_* \mathcal{O}_{Z_1}(-L_1) \simeq R^m(f_1)_* \mathcal{O}_{S_1}(-L_1); 
\end{equation} otherwise $R^m(f_1)_* \mathcal{O}_{Z_1}(-L_1) \simeq 0$. 

Let $\Gamma$ be the (oriented) graph whose vertices $V_{\Gamma}$ are the minimal lc centers of $(Y_1, \Delta_1)$ over $D$, and whose edges are lc centers of dimension $m+1$ joining two minimal lc centers. Then by definition of $d_1$ and \eqref{eq:graph}, the map $d_A$ can be identified with $\delta \otimes \id_{R^m(f_1)_* \mathcal{O}_{S_1}(-L_1)}$, where $\delta \colon \bC[E_\Gamma] \to \bC[V_\Gamma]$, $\delta(e)= e_0 - e_1$  is the boundary map of the graph $\Gamma$. By \cite[Thm.~4.40]{kol13}, $\Gamma$ is connected, so
\begin{align*}
\coker(d^{LCC}_1) = \coker(\delta) \otimes R^m(f_1)_* \mathcal{O}_{S_1}(-L_1) & \simeq H^0(\Gamma, \bC) \otimes  R^m(f_1)_* \mathcal{O}_{S_1}(-L_1) \\
& \simeq R^m(f_1)_* \mathcal{O}_{S_1}(-L_1)\simeq \gr_F^0(R^m(f_2)_*\bC_{S_2})_{\chi}.
\end{align*}
\end{proof}

\section{B-semiampleness conjecture}
\subsection{Proof of the b-semiampleness conjecture} 
\subsubsection{Proof of \Cref{thm:Bsemimain}}
Let $f \colon (Y, \Delta) \to X$ be an lc-trivial fibration inducing the generalized pair $(X, B_X, \bfM)$.
Up to taking a modification of $X$ and the corresponding normalized fiber product of $Y$, we may assume that all varieties involved are quasiprojective.
The b-semiampleness conjecture for lc (or slc) generic fiber is equivalent to the  b-semiampleness conjecture for klt generic fiber, via subadjunction to a source; see \cite[Thm.~1.1]{FG14} or \Cref{rem:FG}.
Moreover, the statement of the conjecture is insensitive to alteration of the base; cf.~\Cref{functoriality}. Therefore, we can suppose:
\begin{enumerate}
 \item[($\ddagger$)] $(Y, \Delta)$ is a klt quasiprojective pair over the generic point of $X$, and that properties $(\dagger)$ and $(\dagger \dagger)$ in \Cref{defn:modulipart} hold.
\end{enumerate}
Then \Cref{defn:modulipart} and \Cref{rmk:rationalCY} imply the existence of:
\begin{enumerate}
\item a snc pair $(X,D)$; 

\item a polarizable integral $\mu_d$-equivariant variation of pure Hodge structures $V_Y$ on $X \setminus D$ with unipotent local monodromy for which $(V_Y)_\chi$ is a complex CY variation with Hodge bundle $M_X$;
\item a polarizable integral $\mu_d$-equivariant CY variation of pure Hodge structures $V_Y'$ on $X\setminus D$ with the same underlying local system as $V_Y$ and for which $(V_Y)_\chi(-2,2)\cong (V'_Y)_\chi$.  In particular, the Hodge bundle of $V_Y'$ is $M_X$.
\end{enumerate}
In view of \Cref{thm:semiample Hodge}, to prove the  b-semiampleness conjecture, it suffices to show that the moduli part is integrable and has torsion combinatorial monodromy.  This is the content of \Cref{thm:integrability} and \Cref{thm:torsion}.

\begin{theorem}[Integrability of $\bfM$]\label{thm:integrability}
Under the assumption ($\ddagger$), the moduli part $M_X$ is integrable.
\end{theorem}

\begin{proof}
It suffices to show that for any integral subvariety $Z \subset \overline{X_{\Sigma}}$ such that the restriction of $M_X$ is not big, the period map of $(V_Y')_{\Sigma}|_Z^\trans$ is not generically immersive.

We first alter $X, Y, Z$ in order to compare the relevant variations of Hodge structure and set the inductive argument on the dimension of the source. By \Cref{lem:alteration}, there exist
\begin{enumerate}
\item a projective alteration $q_1 \colon W_1 \to X$,
\item an lc-trivial locally stable fibration $f_1 \colon (Y_{1}, \Delta_1) \to W_1$, dlt in codimension one,  such that the pullback of the generic fiber of $f$ along $q_1$ is crepant birational to the generic fiber of $f_1$.
\item a prime divisor $E \subset W_{1}$ dominating $Z$; 
\item a proper snc pair $(R, D_{R})$, where $R$ is a general complete intersection in $E$, mapping generically finite onto $Z$,
\end{enumerate}
such that $((q_1^*V_Y')_{E}|_R)^\trans=((q_1|_R)^*(V_Y')_\Sigma)^\trans$
has local unipotent monodromy on $R \setminus D_{R}$. By \S \ref{sec:variation}, up to a further finite dominant morphism $q_2 \colon (W_2,D) \to (R,D_R)$, a source of $(f_1)_{W_2} \colon ((Y_{1})_{W_2}, (\Delta_{1})_{W_2}) \to W_2$, denoted $f_{S} \colon (S, \Delta_S) \to W_2$, is an lc-trivial fibration
\begin{enumerate}
\item inducing the generalized pair $(W_2, B_{W_2}, \bfM')$, 
\item satisfying properties $(\dagger)$ and $(\dagger\dagger)$ (in particular $\bfM'$  descends on $W_2$), and
\item such that
\begin{equation}\label{eq:variationIII}
(q^*(V_{Y})_\Sigma)^{\trans}_\chi\simeq (V_{(Y_1)_{W_2},D})^\trans_\chi \simeq (V_{S})^{\trans}_\chi\quad \left(\hspace{.5em}\Leftrightarrow (q^*(V'_Y)_\Sigma)^\trans\simeq V_S'^\trans\hspace{.5em}\right) \quad \text{and}\quad q^*M_{X} \simeq M'_{W_2},
\end{equation}
\end{enumerate}
where $q \coloneqq q_1 \circ q_2 \colon W_2 \to X$, using Notation \ref{notn:subD}.  The last part follows from \Cref{thm:0.19}. Note also that $f_S$ is locally stable by Lemmas \ref{stabilitylocst}, \ref{stabilitylocstIII} and \cite[cf. proof of Lem. 2.11]{kbook}, so $f_S^*M'_{W_2} \sim_{\bQ} K_{S/W_2} + \Delta_S$ by \Cref{lem:modulipartlctriviallocallystable}.

Now, assume that $M_{X}|_Z$ is not big. Then $M'_{W_2}$ is so too. By \cite[Thm.~3.3]{Amb05} or \cite[Thm.~A.12]{PZ2020}, there exist curves passing through the general point of $W_2$ over which $f_{S}$ is isotrivial, so the period map of $(V_{S})^{\trans}_{\chi}$ (resp.~of $V_{S}^{\trans}$ and $V_{S}'^{\trans}$) is not generically immersive. By \eqref{eq:variationIII}, we conclude that the period map of $(((V_Y)_\Sigma)|_Z)^{\trans}_{\chi}$ (resp.~ $(V_Y')_\Sigma|_{Z}^{\trans}$) is not generically immersive.
\end{proof}

\begin{thm}[Torsion combinatorial monodromy of $\bfM$]\label{thm:torsion}
Under the assumption ($\ddagger$), the moduli part $M_X$ has torsion combinatorial monodromy.
\end{thm}

\begin{proof}
Let $C$ be a proper connected strictly snc curve with normalization $\nu_C$, and $q \colon C\to X$ be a morphism from a proper strictly snc curve $C$ such that $(q\circ \nu_C)^*M_X$ is trivial. We show that the canonical flat connection on $q^*M_X$ has torsion monodromy.
Observe that, in the statement, we can always replace $C$ with proper connected strictly nodal curves dominating $C$. 

\textbf{Step 1.} \emph{In this step, we construct an sdlt modification of $f$ over $C$.}\\
\noindent
By \Cref{lem:alteration}, there exist
\begin{enumerate}
\item a projective alteration $q_1 \colon W \to X$,
\item an lc-trivial locally stable fibration $f_1 \colon (Y_1, \Delta_1) \to W$, dlt in codimension one, such that the pullback of the generic fiber of $f$ along $q_1$ is crepant birational to the generic fiber of $f_1$.
\end{enumerate}
such that
\[
q_1^*\bfM(f)=\bfM(f_1) \qquad\mathrm{and}\qquad  M_{W}\coloneqq M(f_1)_{W} \sim_{\bQ} \frac{1}{k}(f_1)_* (\omega^{[k]}_{Y_1/W}(k\Delta_1)),
\] 
where $k$ is a sufficiently divisible positive integer, and $q_1^{-1}(C)$ has simple normal crossings. Replace $C$ with a (connected) proper strictly nodal complete intersection in $q_1^{-1}(C)$ dominating $C$. In particular, we can suppose that 
the locally stable fibration $f_{1,C} \colon Y_{1,C} \to C$ is sdlt over a dense open set $C^{\circ} \subset C$.
By \Cref{lem:sdltmodification}, there exist a morphism $q' \colon C' \to C$ from a connected stricly snc curves and an lc fibration $f' \colon (Y', \Delta' \coloneqq \Delta_{1, C'}) \to C'$ with the property that 
\begin{enumerate}
\item $(Y', \Delta')$ is an sdlt pair, whose irreducible components each dominate an irreducible component in $C'$;
\item\label{item:connectednesssource_2} the restriction of $f'$ to any sources dominating an irreducible component of $C'$ is an lc-trivial fibration (with connected fibers);
\item\label{item:pushforward_2connectednesssource} $f'_* \omega^{[k]}_{Y'/C'}(k\Delta') \simeq k(q \circ q')^*M_X$,
where $k$ is a sufficiently divisible positive integer.
\end{enumerate} 

\textbf{Step 2.} \emph{In this step, we show that if $f'_* \omega^{[k]}_{Y'/C'}(k\Delta')$
is trivial on each component of $C'$, then the monodromy of $f'_* \omega^{[k]}_{Y'/C'}(k\Delta')$ is torsion.}\\
\noindent To this end, observe that any source $f_S \colon S \to C'_{S}$, for some irreducible component of $C'$, is an lc-trivial fibration with $\bfM \sim_{\bQ} 0$, since $M_{C'} \sim_{\bQ} (q \circ q')^*M_X$ is trivial by assumption along the irreducible components of $C'$. Hence, by 
the compatibility of locally stable families with base change,
\cite[Thm. 3.3 and Prop. 4.4]{Amb05} or \cite[Thm.~A.12]{PZ2020},
$f_S$ is an isotrivial families, i.e., all fibers are isomorphic to one another, which allows to identify minimal lc centers over adjacent nodes of $C'$. 

The intersection complex of an sdlt variety is the dual polyhedron of its dual complex. Denote by $\Delta(C')$ and $\Delta(Y')$ the intersection complex of the snc curve $C'$ and that of the sdlt variety $Y'$. The 1-skeleton of $\Delta(Y')$, denoted $\Delta(Y')_1$, consists of two types of edges: 
\begin{itemize}
\item (dominating edge) those corresponding to sources of $Y'$ dominating a component of $C'$, which are isotrivial families of fiberwise minimal centers; and 
\item (edge of $\bP^1$-link type) those corresponding to $\bP^1$-link between minimal strata of $Y'$, mapping to a node of $C'$.
\end{itemize} 
Note that $\Delta(Y')_1$ is connected by the isotriviality of sources over dominating edges and \cite[Thm.~4.40]{kol13} for edges of $\bP^1$-link type.
Also, by construction, there is a surjective simplicial map $\Delta(Y')_1 \to \Delta(C')_1=\Delta(C')$. In particular, each loop in $\Delta(C')$ admits a (non-unique) lift $\gamma$ in $\Delta(Y')_1$. Fix it once for all. Recall that any vertex of $\gamma$ corresponds to the source $W$. 
A dominating edge $e$ with vertices $s, t$ induces an isomorphism of minimal lc centers 
$(Z_s,\Delta_{Z_s}) \to (Z_t,\Delta_{Z_t})$
by the isotriviality of the family of sources. An edge of $\bP^1$-link type with vertices $s, t$ induces a
crepant
birational map
$(Z_s,\Delta_{Z_s}) \dasharrow (Z_t,\Delta_{Z_t})$
by \cite[Thm.~4.40]{kol13}. Following the birational identification along the loop $\gamma$, we obtain a birational automorphism of a fixed reference minimal lc center
$(Z,\Delta_Z)$ of a general fiber, which induces a representation 
$\mathbb{Z}\gamma \to \mathrm{Aut}(H^0(Z, \omega^{[2k]}_Z(2k\Delta_Z)))$, which is trivial for $k$ large 
and divisible
enough by the finiteness of B-representations \cite[Thm.~1.2]{HX2011}.
\end{proof}

This concludes the proof of \Cref{thm:Bsemimain}.\qed
\subsubsection{B-semiampleness for projective morphisms between complex analytic spaces}In recent years, there has been quite some activity in extending the usual MMP to the context of K\"ahler spaces or analytic varieties.
The  b-semiampleness of the moduli part for projective morphisms of complex analytic spaces can be reduced to the algebraic case of \Cref{thm:Bsemimain}
as follows. In particular, we drop the quasiprojectivity assumption in ($\ddagger$).

\begin{thm}\label{thm:bsemiampleanal}
    Let $(Y,\Delta)$ be a normal complex analytic space with a sub-pair structure, and let $f \colon (Y,\Delta) \to X$ be a projective fibration between complex analytic spaces.
    Assume that $K_Y + \Delta \sim_{\mathbb Q,f} 0$ and that the general fiber of $f$ is an lc pair.
    Then, the moduli part of $f$
    is b-semiample.
\end{thm}

\begin{proof}
Without loss of generality, we can assume that the generic fiber is klt by \Cref{rem:FG}. 
Let $F \colon (\mathcal{Y},\Xi) \to \mathcal{X}$ be a Hilbert scheme parametrizing general fibers of the projective morphism $f \colon (Y, \Delta) \to X$. Up to replacing $X$ with a modification, we can suppose that the classifying map $\Psi \colon X \dashrightarrow \mathcal{X}$ is an analytic morphism, by applying Hironaka's flattening theorem \cite[Cor.~1]{Hir75} to $f$ and the components of $\mathrm{Supp}(\Delta)$ dominating $X$. Replacing $\mathcal{X}$ with the Zariski closure of the image of $\Psi \colon X \to \mathcal{X}$, we can attribute coefficients to the irreducible components of $\Xi$ such that $(Y_{U}, \Delta_U) \simeq (\mathcal{Y},\Xi) \times_{\mathcal{X}} U$, over a dense open set $U \subset X$. In particular, there exists a dense open subset $\mathcal{U} \subset \mathcal{X}$ such that  $(\mathcal{X}_{\mathcal{U}},\Xi_{\mathcal{U}})$ is a pair (\cite[Prop.~2.4]{HX15}) with klt singularities, and $(\mathcal{X}_{\mathcal{U}},\Xi_{\mathcal{U}}) \to \mathcal{U}$ is an lc-trivial fibration. 

 Up to an alteration, the moduli part of $f$ is the Hodge bundle of a variation of Hodge structures obtained by the algorithm in \Cref{defn:modulipart}; see \Cref{rmk:analcanonicalbundle}.
\Cref{defn:modulipart} applied to $F_{\mathcal{U}}$ pulls back to the analogous construction for $f_U$, up to eventually shrinking $U$. Replacing $X$ and $\cX$ with compatible alterations, there exists a morphism $\Psi \colon X \to \cX$, and polarizable variations of Hodge structures $V_{U}$ and $V_{\cU}$ on dense open subsets $U \subset X$ and $\cU \subset \cX$, whose complements are snc divisors, and such that: (1) $\Psi|_{U}^* V_{\cU} = V_{U}$, and (2) the Schmid extension of the deepest piece of the Hodge filtration of $V_{U}$ (resp.~ $V_{\cU}$) is the moduli part. 
By the functoriality of the Schmid's extension, we have $\Psi^* M_{\cX} = M_{X}$. Hence, the semiampleness of the moduli part of $f$ in the analytic case follows from the semiampleness of the moduli part of $F_{\mathcal{U}}$ in the algebraic case, proved in \Cref{thm:Bsemimain}. 
\end{proof}
\begin{remark}\label{rem:FG}
    Let $f \colon (Y, \Delta) \to X$ be a morphism as in \Cref{thm:bsemiampleanal} and $f_S \colon (S, \Delta_S) \to X$ be a source of $(Y, \Delta)$. Note that 
    $\bfM(f) \sim_{\bQ} \bfM(f_S)$ as in \eqref{eq:lastiso}. This means that the moduli part of an lc-trivial fibration from an lc pair is the moduli part of an lc-trivial fibration from a klt pair.
    Together with \Cref{canonicalbundleformula}, this reproves  \cite[Thm.~1.1]{FG14}. In particular, it holds over an analytic base as well.
\end{remark}

\begin{remark}
  \Cref{thm:bsemiampleanal} fails in general for non-projective
fibrations between complex analytic spaces, except in relative dimension 1; see \cite[Ex.~4.17 and Prop.~4.16]{LS2025}. It is unclear whether \Cref{thm:bsemiampleanal} continues to hold for K\"{a}hler morphisms. 
\end{remark}

\subsection{Applications of the  b-semiampleness conjecture and open questions}\label{sec:applicationbsemiampleness}

In this section, we collect some immediate applications of the  b-semiampleness of the moduli part of an lc-trivial fibration.

\subsubsection{Image of lc pairs}
We show that the image of an lc pair under an lc-trivial fibration is again lc, generalizing \cite[Thm. 4.1]{Amb05} and \cite[Lem.~1.1]{FG14}. 
\begin{thm}\label{thm:lcimage}
Let $(Y, \Delta_Y)$ be an lc (resp.~klt) pair and $f \colon (Y,\Delta_Y)\to X$ be a projective fibration with $K_{Y} + \Delta_{Y} \sim_{\bQ, f} 0$.
Then there exists an lc (resp.~klt) pair $(X, \Delta_X)$ such that $K_{Y} + \Delta_Y \sim_{\bQ} f^*(K_{X} + \Delta_X)$.  
\end{thm}
\begin{proof}
The moduli part of the generalized lc pair $(X, B_{X}, \bfM)$ induced by $f$ (cf.~\Cref{canonicalbundleformula}) is  b-semiample by \Cref{thm:Bsemimain}. Apply then \cite[Lem.~4.3]{EFGMS2025}.
\end{proof}
We expect that \Cref{thm:lcimage} could be a key ingredient in inductive arguments in birational geometry.
In the klt case, a version of \Cref{thm:lcimage} for klt sub-pairs is used to reduce the finite generation of canonical rings to the general type case; see \cite[Cor.~1.1.2]{bchm}.
According to \cite[\S 1.8]{BGLM24}, \Cref{thm:lcimage} was one of the missing ingredients to prove that reductive quotients of lc pairs are again lc.

\subsubsection{Adjunction and inversion of adjunction}
The lc centers of a dlt pair $(X,\Delta)$ coincide with the irreducible components of the strata of $\Delta^{=1}$.
In particular, adjunction to an lc center of any dimension can be performed using the usual residue theory iteratively, as with prime divisors;
see \cite[\S 4.2]{kol13}.
In general, to induce a structure of a pair on an lc center $Z$ of an arbitrary lc (not dlt) pair $(X, \Delta)$ is more complicated.
Roughly speaking, one performs the following steps:
\begin{enumerate}
    \item take a dlt modification $(X',\Delta') \to (X,\Delta)$;
    \item choose a prime divisor $S$ in $\mathrm{Supp}(\Delta')^{=1}$ dominating $Z$;
    \item perform dlt adjunction of $(X',\Delta')$ to $S$, thus obtaining a pair $(S,\Delta_S)$;
    \item consider the lc-trivial fibration $(S,\Delta_S) \to W$, where $W$ denotes the Stein factorization of $S \to Z$;
    \item utilize the canonical bundle formula to induce a pair structure on $W$; and
    \item descend this latter structure to the normalization $Z^{\nu}$ of $Z$.
\end{enumerate}
For more details, we refer to \cite[\S 4]{FG2012}.
For this reason, so far, it was only possible to induce the structure of klt pair on minimal lc centers of lc pairs. Thanks to \Cref{thm:Bsemimain}, we can generalize this construction to any lc center.

\begin{theorem}[Adjunction and inversion of adjunction]
\label{adj_inv_adj}
    Let $(X,\Delta)$ be a pair and $Z$ be an lc center.
    Then, the normalization $Z^{\nu}$ can be endowed with a pair structure $(Z^{\nu},\Delta_{Z^{\nu}})$ with the following properties:
    \begin{enumerate}
        \item\label{crepant condition} $K_{Z^{\nu}}+\Delta_{\nu} \sim_{\mathbb Q} (K_X+\Delta)|_{Z^{\nu}}$; and
        \item\label{inv_of_adj} $(X,\Delta)$ is lc in a neighborhood of $Z$ if and only if $(Z^{\nu},\Delta_{Z^{\nu}})$ is lc.
    \end{enumerate}
\end{theorem}

\begin{proof}
    The construction in \cite[\S 4]{FG2012} and \Cref{thm:lcimage} \eqref{crepant condition} gives \eqref{crepant condition}.
    By \Cref{thm:lcimage} and the construction adopted, the ``only if'' part of \eqref{inv_of_adj} follows.
    Thus, we are left with showing that, if $(Z^{\nu},\Delta_{Z^{\nu}})$ is lc, then so is $(X,\Delta)$ in a neighborhood of $Z$.
   In the past, before \Cref{thm:Bsemimain},
    adjunction for general lc centers could only be formulated by using b-divisors on $Z^{\nu}$; see \cite{Hac14,FH22}.  
    The approaches in \cite{Hac14,FH22} are proved equivalent in \cite{FH23}.
    The b-divisor considered in \cite{Hac14} is exactly the boundary b-divisor of the canonical bundle formula considered in \cite[\S 4]{FG2012}. In particular, in \cite{Hac14}, inversion of adjunction is formulated by requiring that the boundary b-divisor has coefficient at most 1 on any model.
    By \Cref{thm:lcimage} and the construction adopted, this 
    condition is equivalent to requiring that the pair $(Z^{\nu},\Delta_{Z^{\nu}})$ is lc.
    Then, the claim follows.
\end{proof}

\subsubsection{Comment about boundary with $\bR$-coefficients}
Throughout this work, for a pair $(X,\Delta)$ we assume that $\Delta$ has coefficients in $\mathbb{Q}$.
For many applications, it is important to extend results to pairs with coefficients in $\mathbb R$.
We remark that \Cref{thm:lcimage} and \Cref{adj_inv_adj} also hold for 
pairs with real coefficients, provided that all linear equivalences are taken to be over $\mathbb R$.
This relies on the approximation of pairs with real coefficients by means of convex combinations of pairs with rational coefficients;
see, e.g., \cite{HanLiu}.

Indeed, let $f \colon (Y,\Delta) \to X$ be an lc-trivial fibration where $(X,\Delta)$ is a quasi-projective $\bR$-pair.
By \cite[Lem.~4.1]{HanLiu}, we may write the $\bR$-divisor $K_Y+\Delta = \sum c_i (K_Y+\Delta_i)$ as a real convex combination of finitely many $\bQ$-Cartier divisors $K_Y+\Delta_i$ such that: (1) the non-lc and non-klt loci of the pairs $(Y,\Delta_i)$ agree with the corresponding ones of $(Y,\Delta)$ for all $i$; and (2) $K_Y+\Delta_i \sim_{\mathbb Q,f} 0$ for all $i$. 
Let $(X_i,B_i)$ be the lc $\bQ$-pair obtained applying \Cref{thm:lcimage} to the lc-trivial fibration $f \colon (Y,\Delta_i) \to X$. To deduce \Cref{thm:lcimage} for $(Y,\Delta)$, take the lc $\bR$-pair $(X,\sum c_i B_i)$.  Moreover, to obtain the corresponding version of \Cref{adj_inv_adj} for $(Y, \Delta)$, we utilize the pair just constructed whenever the canonical bundle formula is invoked in the proof, together with inversion of adjunction for fiber spaces and the fact that $(Y,\Delta)$ and $(Y,\Delta_i)$ have the same classes of singularities.

\subsubsection{Effective  b-semiampleness}
Let $f \colon (Y, \Delta) \to X$ be an lc-trivial fibration inducing the generalized pair $(X, B_{X}, \bfM)$. The effective  b-semiampleness conjecture predicts the existence of a universal positive integer $c$, only depending on the relative dimension of $f$ and the coefficients of the horizontal part of $\Delta$, such that $c \bfM$ is  b-free, or eventually weaker statements involving other topological invariants of the general fiber. The conjecture was formulated by Kawamata, Prokhorov and Shokurov  \cite[Conj 1]{AmbroPhD} and \cite[Conj.~7.13.1]{PS09}, and proved in \cite[Cor.~8.15]{PS09} in relative dimension $1$.
Recently, the conjecture has also been solved for lc-trivial fibrations whose general fiber is an abelian or primitive symplectic variety with bounded second Betti number of fixed dimension; see \cite[Thm.~C]{EFGMS2025}. This was a key step towards the proof of boundedness results for certain $K$-trivial fibrations in \cite{EFGMS2025}.
In particular, combining \cite{EFGMS2025} and \cite[Thm.~1.4]{babwild}, the conjecture is also settled in relative dimension 2.

It is natural to ask whether a fixed positive power of the Schmid extension of the Griffiths bundles of a polarizable integral variation of Hodge structures sharing the same period domain is free. Analogously, whether the same holds for the Hodge bundle of integrable polarizable CY variations of Hodge structure with torsion combinatorial monodromy sharing the same period domain. A positive answer to this question may entail boundedness results for more general $K$-trivial fibrations.

While the effective version of the b-semiampleness conjecture remains open at the moment,
\Cref{cor:CYmod} below provides a tool to prove it when the general fibers belong to a given bounded family of pairs.
Thus, we deduce the effective version of the conjecture when the general fiber is a klt log CY pair of Fano type.
In particular, we obtain the following application, which was kindly pointed out to us by Shokurov.
\begin{cor}\label{cor:effective_b_semiampleness_fano}
Let $f \colon (Y,\Delta)\to X$ be an lc-trivial fibration of relative dimension $n$ from a pair $(Y,\Delta)$
such that $\Delta$ is effective over the generic point of $X$.
Further assume the following:
\begin{itemize}
    \item[(i)] the generic fiber of $f$ is a klt pair;
    \item[(ii)] $\Delta$ is big over $X$.
\end{itemize}
Let $\bfM$ denote the moduli b-divisor induced by $f$.
Then, there is a constant $I$, only depending on $n$ and the horizontal multiplicities of $\Delta$, such that $I\bfM$ is b-Cartier and b-free.
\end{cor}

\begin{proof}
    Since the moduli b-divisor is determined by the general behavior of the fibration, we may shrink $X$ so that $(Y,\Delta)$ is a klt pair and $\Delta$ has no $f$-vertical components.
    Then, we may replace $(Y,\Delta)$ with a small $\mathbb Q$-factorialization.
    In particular, we may assume that $\Delta$ is a $\mathbb Q$-Cartier divisor that is $f$-big.
    Notice that $(Y,(1+\epsilon)\Delta)$ is klt for $0< \epsilon \ll 1$ and $K_Y+(1+\epsilon)\Delta$ is $f$-big.

    Thus, we may apply \cite{bchm} to $(Y,(1+\epsilon)\Delta)$ and the morphism $f$ and replace $Y$ with the relatively ample model of $\Delta$.
    In particular, $Y$ may no longer be $\mathbb Q$-factorial, but we gain the fact that $\Delta$ is $f$-ample.
    Notice that all these operations do not affect $\bfM$.

    Up to shrinking $X$, we may assume that all the fibers $(Y_x,\Delta_x)$ of $f$ are klt pairs.
    Then, these belong to a bounded family of pairs by \cite{HX15}.
    In particular, again up to further shrinking $X$, there is a constant $C$, independent of $f$, such that  
    $C\Delta$ is Cartier and $f$-ample.

    Then, we may invoke \Cref{cor:CYmod} for these log CY pairs and the polarization given by the Cartier divisor $C\Delta_x$.
    Then, $\bfM$ is pulled back from a unique space as constructed in \Cref{cor:CYmod}, and the claim follows.
\end{proof}

\subsubsection{Connections to the theory of complements}
The theory of complements has been introduced by Shokurov to study flips and, more generally, morphisms of Fano type;
we refer to \cite{shokurov20} for details.
In recent years, this theory has proved very powerful;
for instance, Birkar's proof of the BAB conjecture relies heavily on this theory;
see \cite{bab1,bab2}.
On the other hand, in Birkar's strategy, it is important to relax the category of pairs to also include generalize pairs, and then study complements for these more general objects, originally introduced in \cite{BZ16}.
Indeed, in \cite{bab1},
complements are built with an inductive approach, and one of the possible scenarios includes lifting complements from the base of an lc-trivial fibration with fibers of Fano type;
see \cite[\S6.4]{bab1}.
In particular, even when interested in pairs, the approach in \cite{bab1} needs to introduce generalized pairs for the inductive argument to go through.
Now, by \Cref{cor:effective_b_semiampleness_fano}, we may induce a pair with controlled coefficients on the base of such lc-trivial fibrations.
Thus, it would be interesting to explore whether boundedness of complements could be proved without resorting to generalized pairs.

\subsubsection{B-semiampleness for GLC fibrations}\label{GLC} Let $f \colon (Y, \Delta) \to X$ be a generically log canonical (GLC) fibration, i.e., $(Y, \Delta)$ is lc over the generic point of $X$.
Note that, contrary to the lc-trivial case, the general fiber is no longer assumed to be log Calabi--Yau. In the GLC case, the moduli part is the  b-divisor on the total space $Y$ given by $N_Y \coloneqq K_Y + \Delta - f^*(K_{X}+B_X)$, up to flatification of $Y$; see \cite[\S 2.2]{ACSS2021} for details.
It is known that $\bfN$ is  b-nef, relatively  b-semiample, but not  b-semiample in general; see \cite{ACSS2021}.
However, Shokurov conjectured that $\bfN$ is  b-semiample after a small perturbation by an ample divisor coming from the moduli of the general fiber; see \cite[Conj.~1]{shokurov2021logadjunctionmoduli}.
The second-named author and Spicer proved a variant of Shokurov's conjecture in \cite[Thm.~1.2]{FS22}: for 
GLC fibrations with klt generic fibers that are locally stable families of good minimal models,
the  b-divisor $\bfN+ \epsilon f^* \det(f_* m\bfN)$ is  b-semiample, conditional to the  b-semiampleness conjecture in the lc-trivial case.
Here, $m$ is sufficiently divisible and $\epsilon$ is arbitrarily small and positive.
The result now holds unconditionally by \Cref{thm:Bsemimain}.

\subsection{Moduli of Calabi--Yau varieties}\label{sec:moduli} In this section we deduce \Cref{cor:introCYmod}.  In fact, we prove a more precise statement allowing for mild singularities.

Let $\cY$ be a $\mathbb{G}_m$-rigidified 
algebraic stack of finite type  parametrizing  polarized klt log Calabi--Yau pairs, i.e., triples $(X,\Delta;L)$, where $(X,\Delta)$ is a klt log Calabi--Yau pair, and $L$ is an ample line bundle on $X$. Note that here we rigidify with respect to the automorphisms of the line bundle $L$. Families of triples $(X,\Delta;L)$ are families of locally stable pairs in the sense of Koll\'ar, together with a polarization, i.e., the datum of compatible relatively ample line bundles defined \'etale locally over the the base; see, e.g., \cite[\S 4.2]{AH2011} or \cite[Def.~8.40]{kbook}.

We list noteworthy properties of the moduli stack $\cY$:
\begin{enumerate}[label=(\roman*)]
\item The moduli stack $\cY$ is a separated (cf., \cite[Thm.~11.40]{kbook}) Deligne--Mumford stack. Indeed, the group of polarized automorphisms of a klt log Calabi--Yau pair is finite; see, e.g., \cite[Prop.~10.1]{PZ2020}. Hence, $\cY$ admits a coarse moduli space $Y$ which is a separated algebraic space of finite type by \cite{KM1997}.
 \item Let $f \colon (\mathcal{X}, \Delta_{\mathcal{X}}) \to S$ be a polarized family of klt log Calabi--Yau pairs admiting a classifying map $\psi \colon S \to Y$. If $f$ is isotrivial, then $\psi$ is locally constant. Indeed, since the Albanese fibration of any projective klt log Calabi--Yau pair $(X, \Delta)$ is surjective and isotrivial by \cite[Cor.~A.14]{PZ2020}, there exists an \'{e}tale cover of $(X, \Delta)$, with Galois group $G$, which can be decomposed into a product of an abelian variety $A$, isogenous to the Albanese variety, and a klt log Calabi--Yau pair isomorphic to the fiber of the Albanese fibration. The $G$-equivariant translation group of $A$ descends to automorphisms of $(X, \Delta)$, which identify numerically equivalent ample line bundles on $X$ up to finite ambiguity. Therefore, for any given fixed pair $(X, \Delta)$, there are at most finitely many points of $Y$ parametrizing triple $(X, \Delta; L)$, with an arbitrary line bundle $L$.  This means that $Y$ does not contain non-constant isotrivial families of pairs. 
\item The Hodge bundle of the universal family is well-defined on $\cY$ as it coincides with the relative canonical bundle of the family. Its powers descend to a $\bQ$-line bundles $M^{(k)}_{Y}$ on the coarse moduli space $Y$; see \cite[Lem.~3.2]{KV04}. 
\item \label{item:stictlynef} The $\bQ$-line bundles $M^{(k)}_{Y}$ are strictly nef by \cite[Thm.~3.3]{Amb05} or \cite[Thm.~A.12]{PZ2020} and (ii).
\item There is a dense Zariski open substack $\cU \subset \cY^{\mathrm{red}}$ over which the topology of the universal family of the pairs $(X, \Delta)$ (more precisely, that of the construction in \Cref{defn:modulipart}) is locally constant. The variation of Hodge structures associated to the universal family, as constructed in \Cref{defn:modulipart}, gives a period map $\phi \colon \cU \to U \to \Gamma\backslash \bD$, which factors through the coarse moduli space $U$ of $\cU$; cf. \cite[Proof of Cor.~7.3]{BBT23}.
\item \label{item:quasifinite} The map $U \to \Gamma\bs\bD$ is quasifinite again by \cite[Thm.~3.3]{Amb05} or \cite[Thm.~A.12]{PZ2020} and (ii).

\item
    Consider a family of polarized klt log Calabi--Yau pairs $(\mathcal{X}_U,\Delta_{\mathcal{X}_U};\mathcal{L}_U) \to U$ over a punctured curve $U \coloneqq C \setminus \{0\}$.
    Up to a ramified base change, the family can be completed to a family of log Calabi--Yau pairs $(\mathcal{X},\Delta_\mathcal{X}) \to C$.
    In general, the natural limit $\mathcal{L}_0$ on the special fiber $\mathcal{X}_0$ may just correspond to the linear equivalence class of an ample Weil $\mathbb Q$-Cartier divisor;
    see, e.g., \cite{KX20}.
    Thus, assuming that the limit $(\mathcal{X}_0,\Delta_{\mathcal{X}_0})$ is klt, there may be klt pairs that would be expected to appear in the fixed moduli problem, but they actually do not since $\mathcal{L}_0$ may not be Cartier. 
    Yet, if $\mathcal{L}_0$ fail to be Cartier, its Cartier index is however bounded by \cite[Thm.~1.5]{Bir23}, up to replacing the polarization defining the original moduli problem with a bounded tensor multiple, we may find a moduli stack of finite type in the above sense parametrizing all $\mathbb Q$-polarized klt log Calabi--Yau pairs in a given deformation class
    (note that the $\epsilon$ in {\it loc. cit.} may be chosen uniformly by global ACC \cite[Thm.~1.5]{HMX14}).
    
\end{enumerate} 

\begin{cor}\label{cor:CYmod}
Let $\cY$ be an
algebraic stack of finite type  parametrizing  polarized klt log Calabi--Yau pairs, and $Y$ be its coarse moduli space. Then the normalization $Y^{\nu}$ of the reduction $Y^{\mathrm{red}}$ has a unique normal compactification $Y^\BBH$ for which some power $M_Y^{(k)}$ of the Hodge bundle of the variation of Hodge structures on middle cohomology extends to an ample bundle $\cO_{Y^\BBH}(k)$ and such that for any family $g:\testsp\bs D_\testsp\to \cY$
for a log smooth algebraic space $(\testsp,D_\testsp)$, the resulting morphism $\testsp\bs D_\testsp\to Y$ lifts to $Y^{\nu}$ and extends to extends to $\bar g \colon \testsp\to Y^\BBH$ with the property that $\bar g^*\cO_{Y^\BBH}(k)$ pulls back to $(M^k_{\testsp\bs D_\testsp})_\testsp$.
\end{cor}
\begin{proof}
The Hodge bundle of the universal family over $\cY$ is integrable and has torsion combinatorial monodromy by \Cref{thm:integrability} and \Cref{thm:torsion}, since these can be checked on an alteration.  These properties of the Hodge bundles, together with properties \ref{item:stictlynef} and \ref{item:quasifinite}, allows us to apply  \Cref{thm:BBHodge}: there is a normal compactification $(U^{\nu})^{\BBH}$ of the normalization $U^{\nu}$ of $U$ satisfying the properties of \Cref{thm:BBHodge}. In particular, the Hodge bundle $M_{U^{\nu}}^{(k)}$ extends amply on $(U^{\nu})^{\BBH}$. 
The Hodge bundle on $\cY$ agrees with the Schmid extension of  the Hodge bundle on a log smooth resolution $S \to \cY$. By the universal property of coarse moduli space, the extension $S \to (U^{\nu})^{\BBH}$ factors through $Y^{\nu}\to (U^{\nu})^{\BBH}$, which is birational and quasifinite, 
hence an open immersion. 
\end{proof}

\begin{proof}[Proof of \Cref{cor:introCYmod}]
By Bogomolov--Tian--Todorov \cite{bogomolov,tian,todorov}, $\cY$ is smooth, so $Y$ is normal.  \end{proof}

\subsubsection{Stratification of Baily--Borel compactifications}\label{rem:CYnumerics} In the context of \Cref{cor:CYmod}, the following description of the underlying set of points of $Y^\BBH$ follows from the proof in $\S 5$.  For any choice of proper log smooth algebraic space $(X,D)$ with a proper birational morphism $X\bs D\to Y$, we will have $Y^\BBH=X(\bC)/R_\curve$.  For a particular choice $(X',D')$, the compactification $Y^\BBH$ will have a natural stratification by the images $Y^\BBH_S$ of the Hodge strata $X_S'$, and for each one there will be a period map whose projection $\widetilde{Y_S^\BBH}^{V_{S,\bQ}^\trans}\to\bP(V_{X',S,\bC,x_S}^\trans)^\an$ has discrete fibers.  Thus, for any stratum $X_\Sigma$ of $X$ whose image meets $Y_S^\BBH$ we have 
\begin{equation}\dim Y^\BBH_S\leq \rk \gr_F^{m-1}V^\trans_{X',S}\leq \rk \gr_F^{m-1}\gr^W_{k_\Sigma}V_{X,\Sigma}=\rk \gr_F^{m-1}\gr^W_{k_\Sigma}\psi_\Sigma V\label{eq:bound}\end{equation}
where $\psi_{\Sigma} V$ is the limit mixed Hodge structure at any point of $X_{\Sigma}$.  The last equality follows since any part of $\gr_F^{m-1}\gr^W_{k_\Sigma}\psi_\Sigma V$ is necessarily primitive, since $\gr_F^m\gr^W_{k_\Sigma+2}\psi_\Sigma V=0$.  In particular, in the case of the Baily--Borel compactification of a moduli space of $d$-dimensional Calabi--Yau varieties as in \Cref{cor:CYmod} (where $V$ is the variation on degree $d$ cohomology), we have $\dim Y=\rk \gr_F^{d-1} V$ and so
\[\codim Y^{\BBH}_S\geq \sum_{k\neq  k_\Sigma}\rk \gr_F^{d-1}\gr^W_k\psi_\Sigma V.\]
As it can be seen from the case of $\cA_g$ or the moduli of $K3$s, this codimension can be quite large in practice. In fact, in these two cases, equality is achieved (though in general it won't be, as either of the inequalities in \eqref{eq:bound} might be strict).

\subsubsection{
BB vs BBH compactifications}\label{eg:BBvsBBH}  We give an example of a moduli stack of Calabi--Yau manifolds $\cY$ for which the morphism $Y^\BB\to Y^\BBH$ has positive-dimensional fibers.  This is easy to do on the boundary.  For example, let $Y$ be the coarse moduli space of smooth quintic threefolds, and consider a degeneration to a transverse union of a hyperplane $L$ and a smooth quartic threefold $T$, meeting along a quartic K3 surface $S$.  The associated graded of the limit mixed Hodge structure in degree three then includes the primitive cohomology $H^3(T,\bQ)_{\mathrm{prim}}$ in weight 3 and the primitive cohomology $H^2(S,\bQ)_{\mathrm{prim}}$ in weight 2.  The transcendental part of the limit mixed Hodge structure is then the transcendental part of $H^2(S,\bQ)$.  Thus, since quartic threefolds and quartic surfaces satisfy an infinitesimal Torelli (for middle cohomology), and the period map of the linear systems $|\cO_T(1)|$ on a fixed (general) $T$ is immersive, this boundary piece survives in $Y^\BB$ with dimension 
$4+(\binom{4+4}{4}-5^2)=49$.  On the other hand, this boundary piece has dimension $19$ in $Y^\BBH$.

It is even possible for $Y^\BB\to Y^\BBH$ to contract curves in the interior of the period domain.  Precisely, $Y^\BB$ will always contain as an open set $Y\subset \breve{Y}^\BB\subset Y^\BB$ the normalization (in $Y$) of the closure of the image of $Y^\an\to \Gamma\bs \bD$.  This is the locus of $Y^\BB$ where the Hodge structure does not degenerate (or equivalently, where the limit mixed Hodge structure is pure), and $\breve {Y}^\BB\to Y^\BBH$ may have positive-dimensional fibers. This is indeed the case of the coarse moduli space $Y$ of smooth quintic threefolds. 

The following is an example due to Radu Laza, and we warmly thank him for sharing it with us. For a generic choice of $\underline{a} \coloneqq (a_1, a_2, a_3, a_4) \in \bC^4$,\footnote{e.g., $(a_1, a_2, a_3, a_4)=(1,1,1,2)$ works but $(a_1, a_2, a_3, a_4)=(1,1,1,1)$ does not.} the quintic threefold $X$ in $\bP^4$
cut by the polynomial
\[
f_{\underline{a}}(x_0, x_1, x_2, x_3, x_4) \coloneqq x_0^2 x^3_1 + x_0^3(x_2^2+ x_3^2+x_4^2) + a_1 x_1^5+a_2 x_2^5+a_3 x_3^5+ a_4 x_4^5
\]
has a unique isolated $A_2$ singularity at $p \coloneqq [1:0:0:0:0]$. The blowup $\widetilde{X} \to X$ along $p$, with exceptional divisor $E \simeq \bP(1,1,2)$, is a resolution of singularities. The Mayer--Vietoris exact sequence for the mapping cylinder of this resolution reads \[H^{i-1}(\widetilde{X}, \bQ) \twoheadrightarrow H^{i-1}(E, \bQ) \to H^{i}(X, \bQ) \to H^{i}(\widetilde{X}, \bQ),\]
which implies that $H^i(X, \bQ)$ carries a pure Hodge structure. 

Now, let $\cX$ be the hypersurface in $\bP^4 \times \bA^2_{(b_1, b_2)}$ cut by 
\[
f_{\underline a}(x_0, \ldots, x_4) - (b_2^2+b_2+1)b_1^2 x_1x^4_0 + b_2(b_2+1)b_1^3 x^5_0 = 0,
\]
and let $f \colon \cX \to \bA^2_{(b_1, b_2)}$ be the natural projection. Note that the restriction of $\cX$ along the curve $(b_1=0)$ is a trivial family with fiber $X$. Instead, the restriction of $\cX$ along the curve $\{b_2=m\}$, with $m$ general, is a family with a unique isolated compound Du Val singularity of type $cD_4$ along its central fiber $\cX_{(0,m)}$. The local model is
\begin{equation}\label{eq:DuVal}
x_3^2+x_4^2 + [x^3_1 + x_2^2 - (b_2^2+b_2+1)b_1^2 x_1 + b_2(b_2+1)b_1^3]=0,
\end{equation}
i.e., a double suspension of \cite[Eq.~(7.3)]{CML2013}, which is a blowup of the Weyl cover of the miniversal deformation of a cuspidal curve. By the Thom--Sebastiani's formula for Hodge modules (cf., e.g., \cite[Thm.~1.2]{MSS2020}), the vanishing cohomology of \eqref{eq:DuVal} and \cite[Eq.~(7.3)]{CML2013} are isomorphic Hodge structures, up to a Tate shift. By \cite[\S 7.7]{CML2013}, the vanishing cohomology is then a Tate shift of the first cohomology group of an elliptic curve with $j$-invariant $256(b^2_2+3)^3/(b^2_2-1)^3$. 

Globally, the nearby cohomology for $f \colon \cX_{(b_1, m)} \to \bA^1_{b_1}$
is an extension of the pure Hodge structures of the vanishing cohomology of the isolated hypersurface singularity and of the cohomology of $\cX_{(0,m)}$, as one can check taking cohomology of the specialization triangle \cite[p.276]{peterssteenbrink} and using the purity of the vanishing cohomology and of $H^*(X, \bQ)$. 
By the purity of the limiting mixed Hodge structure, the period map $Y^{\mathrm{an}} \to \Gamma\bs \bD$ extends to $(\bA^2)^{\mathrm{an}}_{(b_1, b_2)} \to \Gamma\bs \bD$. 

We conclude that the curve $(b_1=0)$ maps generically finitely in $\breve Y^\BB$, but it is contracted by $Y^{\BB} \to Y^{\BBH}$. Indeed, the Hodge bundle is trivial along $(b_1=0)$ since $f$ is trivial along the curve, but the Griffiths bundle of $R^3 f_* \bQ_{\cX}|_{(\bA^1)^*_{b_1} \times \bA^1_{b_2}}$ is positive since the $j$-invariant of the vanishing cohomology as a summand of the nearby cohomology varies.

\bibliographystyle{amsalpha}
\bibliography{main}
\end{document}